\documentclass[a4paper, 11pt]{article}
\usepackage{graphicx}
\usepackage{amsmath, amssymb, amsthm}
\usepackage{subfig}
\usepackage{color}
\usepackage{mathrsfs}
\usepackage{multirow}

\newtheorem{theorem}{Theorem}
\newtheorem{lemma}[theorem]{Lemma}
\newtheorem{proposition}[theorem]{Proposition}
\newtheorem{definition}{Definition}

\newcommand\bx{\boldsymbol{x}}
\newcommand\bw{\boldsymbol{w}}
\newcommand\bn{\boldsymbol{n}}
\newcommand\bu{\boldsymbol{u}}
\newcommand\bv{\boldsymbol{v}}

\newcommand\bg{\boldsymbol{g}}

\newcommand\bh{\boldsymbol{h}}
\newcommand\bbR{\mathbb{R}}
\newcommand\bbN{\mathbb{N}}
\newcommand\bbS{\mathbb{S}}

\newcommand\bq{q}
\newcommand\bp{p}

\newcommand\dd{\,\mathrm{d}}

\newcommand\mQ{\mathcal{Q}}
\newcommand\mM{\mathcal{M}}
\newcommand\mF{\mathcal{F}}
\newcommand\halpha{\hat{\alpha}}

\newcommand\balpha{{\alpha}}
\newcommand\bbeta{{\beta}}
\newcommand\be{{e}}
\newcommand\blambda{{\lambda}}
\newcommand\bkappa{{\kappa}}
\newcommand{\imag}{\mathrm{i}}

\numberwithin{equation}{section}

\graphicspath{{images/}} 
{\theoremstyle{remark} }
\newtheorem{corollary}{Corollary} 

\title{Approximation to Singular Quadratic Collision Model in
  Fokker-Planck-Landau Equation}

\author{Ruo Li\thanks{CAPT, LMAM \& School of Mathematical Sciences,
    Peking University, Beijing, China, email: {\tt
      rli@math.pku.edu.cn}.},~~ Yanli Wang\thanks{Department of
    Engineering, Peking University, Beijing, China, 100871, email:
    {\tt wang\_yanli@pku.edu.cn}.}, ~~Yixuan Wang\thanks{School of
    Mathematical Sciences, Peking University, Beijing, China, 100871,
    email: {\tt roywangyx@pku.edu.cn}.}, }

\begin{document}
\maketitle
% vim: tw=70:spell
\begin{abstract}
  We propose a Hermite-Galerkin spectral method to numerically solve
  the spatially homogeneous Fokker-Planck-Landau equation with
  singular quadratic collision model. To compute the collision model,
  we adopt a novel approximation formulated by a combination of a
  simple linear term and a quadratic term very expensive to
  evaluate. Using the Hermite expansion, the quadratic term is
  evaluated exactly by calculating the spectral coefficients. To deal
  with singularities, we make use of Burnett polynomials so that even
  very singular collision model can be handled smoothly. Numerical
  examples demonstrate that our method can capture low-order moments
  with satisfactory accuracy and performance.

\vspace*{4mm}
\noindent {\bf Keywords:} Fokker-Planck-Landau equation;
Hermite-Galerkin spectral method; Burnett polynomials; Quadratic collision
operator; Super singularity
\end{abstract}

\section{Introduction}

The Fokker-Planck-Landau (FPL) equation is a common kinetic model in
plasma physics, accelerator physics and astrophysics. It describes
binary collision between charged particles with long-range Coulomb
interaction, and is represented by a nonlinear partial
integro-differential equation.

As a classical result, the FPL operator is the limit of the Boltzmann
operator for a sequence of scatting cross sections which converge in a
convenient sense to a delta function at zero scattering angle
\cite{comparisonBuet1999}. The original derivation of the equation
based on this idea is due to Landau \cite{Landau1965}, and then
several work has been devoted to this problem, such as
\cite{ArsenBuryak, DegondLucquinDesreux, Rosenbluth}. Recently Villani
\cite{Villaniweak1998} has obtained a rigorous proof of this
asymptotic problem in the space homogeneous scenario. For the
mathematical properties of the FPL equation, such as the existence of
the solutions, we refer the reader to Villani \cite{Villanireview} and
the reference therein.

The numerical solution of the nonlinear kinetic equations, such as FPL
equation, also represents a real challenge for numerical method. This
is essentially due to the non-linearity, as well as the high dimension
of variables, which is seven for the full problem. Moreover, the
complex three dimensional integro-differential stiff
advection-diffusion operator in velocity space is also remarkably
difficult to deal with due to the high singularity.  Besides, this
integration has to be handled carefully since it is closely related to
the macroscopic properties, for example the collision term does not
change the total mass, momentum and energy. Several numerical
approaches have been brought up to solve FPL equation. Generally
speaking, there are two kinds of methods, stochastic methods and the
deterministic methods. For the stochastic methods, DSMC method which
is widely in the simulation of Boltzmann equation\cite{bird} is
adopted to solve FPL equation.  A detail discussion about the
stochastic method is beyond the scope of this paper, and we refer the
reader to \cite{DirectDimarco2010, BobylevMonte2013} for a much more
complete treatment. For the deterministic methods, due to the complex
form of the FPL operator, several numerical approaches are devoted to
the simpler diffusive Fokker-Planck model \cite{GuoTang2000,
  ZhangWei1997}, the space homogeneous situations in the isotropic
case \cite{BuetConservative1998} or cylindrically symmetric problems
\cite{ConservativePekker1984}. Moreover, Villani \cite{Villani1998on}
has brought up a linear collision model for the Maxwell molecules. The
construction of conservative and entropy schemes for the space
homogeneous case has been proposed in \cite{Degond1994,
  BuetNumerical1999}, where the main physical properties are all
satisfied. But the direct implementation of such schemes are all quite
expensive. Several fast approximated algorithms to reduce the
complexity of these methods, based on multipole expansion
\cite{Lemou1998} or multigrid techniques \cite{BuetConservative1998}
have been proposed. A fast spectral method based on Fourier spectral
approximation of the collision operator is introduced in
\cite{PARESCHI2000}, and it is then also utilized to solve
inhomogeneous FPL equation \cite{Deterministiczhang2016,
  Filbet2002numerical}. For the numerical stiffness of the
Fokker-Planck collision operator, the implicit time scheme is also
studied \cite{ImplicitLemou2005, massTaitano2015}. There is a certain
kind of asymptotic-preserving method that seeks to accelerate the
solution of the FPL equation by the so-called penalization techniques
\cite{JinYan2011}.

As another kind of spectral method, Hermite spectral method is also
utilized to solve FPL equation. Hermite method, where the basis functions
with weighted orthogonality in $\bbR^3$ are employed, 
dates back to Grad's work \cite{Grad} where it is used to solve the 
Boltzmann equation and is known as the moment method ever since. 
Besides, the expansion with respect to Burnett polynomials was proposed in 
\cite{Burnett1936, Gamba2018} to find the coefficients of the collision term
of the expansion in the Hermite basis. Using the
Hermite expansion, it is still a tough job to evaluate the exact
coefficients in the expansion of the collision operator, since the
computational cost for the quadratic from is hardly bearable 
and novel models need to be introduced. In a recent
work \cite{QuadraticCol}, the explicit expressions of all the
coefficients in the Hermite spectral method for the quadratic
Boltzmann collision operator are presented, and the new collision
model which can preserve the physical properties and reduce
computational cost at the same time was brought up using these
coefficients. It is much harder to evaluate these coefficients for the
quadratic FPL collision operator compared to the Boltzmann equation,
because of the high singularity and the operator of partial
derivative. In \cite{ExactPfefferle2017}, the coefficients for the
Coulombian case were evaluated numerically, and the explicit form was
listed for the first few moments.

Inspired by these work, we in this paper are devoted to the numerical
method for FPL equation with quadratic collision model, which may be
very singular. Following the approach in \cite{QuadraticCol}, we
approximate the collision model as the combination of a simple linear
term and a quadratic term. The idea is to take only a portion in the
truncated series expansion to be treated ``quadratically'', and the
remaining part is approximated by the linear collision operator
brought up by Villani \cite{Villani1998on}. This may greatly reduce
the computational cost and we can still capture the evolution of
physical variables accurately. The linear term can be handled easily,
while the difficulty imposed by the singularity in the quadratic
collision model remains. We reveal that by making use of Burnett
polynomials, the singular part of the integral in the collision
operator can be handled smoothly. For the typical case that the
repulsive force between molecules is proportional to a negative power
of their distance, our method can handle problems where the index for
the power of distance is as great as $-5$, in comparison to the index
fixed as $-3$ in \cite{ExactPfefferle2017}. To deal with the remaining
part in the quadratic term without singularity, the Hermite-Galerkin
spectral method is then adopted. We derive the explicit formulae for
all the coefficients in the Hermite expansion of the collision
operator, and these formulae can all be evaluated exactly offline for
immediate applications. Thus eventually the quadratic term is able to
be evaluated efficiently.

The rest of this paper is organized as follows. In Section
\ref{sec:Bol_Her}, we briefly review the FPL equation and the Hermite
expansion of the distribution function. In Section \ref{sec:general},
we first give an explicit expression of the series expansion of the
quadratic collision operator and then introduce precisely how to deal
with the singularity by Burnett polynomials. The construction of the
approximated collision model is presented in Section
\ref{sec:construction}. Some numerical experiments verifying the
effectiveness of our methods are carried out in Section
\ref{sec:numerical}. The concluding remarks and detailed derivation of
the expansions are given in Section \ref{sec:conclusion} and
\ref{sec:appendix} respectively.

% vim: tw=70:spell
\section{FPL equation and Hermite expansion}
\label{sec:Bol_Her}

We will first give a brief review of the FPL equation, and then
introduce the Hermite spectral method for the expansion of the distribution function.

\subsection{FPL equation}
\label{sec:FPL}
The Fokker-Planck-Landau equation is a prevalent kinetic model in
plasma physics, describing the state of the particles in terms of a
distribution function $f(t,\bx,\bv)$, where $t$ is the time
coordinate, $\bx$ represents the spatial coordinates, and $\bv$ stands
for the velocity of particles. The governing equation of $f$ is
\begin{equation}
\frac{\partial f}{\partial t} + \nabla_{\bx} \cdot (\bv f)  = \mQ[f],
  \qquad t\in \mathbb{R}^+, \quad \bx \in \mathbb{R}^3,
    \quad \bv \in \mathbb{R}^3,
\end{equation}
where $\mQ[f]$ is the collision operator with a quadratic form
\begin{equation} \label{eq:quad_col}
  \mQ[f](t,\bx,\bv) =  \nabla_{\bv} \cdot \int_{\mathbb{R}^3}
  A(\bv - \bv_{\ast})(f(\bv_{\ast})\nabla_{\bv}f(\bv) -
  f(\bv)\nabla_{\bv_{\ast}}f(\bv_{\ast})) \dd \bv_{\ast},
\end{equation}
where $A$ depends on the interaction between particles and is a
$3\times 3$ negative and symmetric matrix in the form
\cite{PARESCHI2000} of
\begin{equation}
  \label{eq:A}
  A(\bv) = \Psi(|\bv|) \Pi (\bv), 
\end{equation}
where $\Psi$ is a non-negative radial function, and $\Pi(\bv)$ is the
orthogonal projection upon the space orthogonal to $\bv$, as
$\Pi_{ij}(\bv) = \delta_{ij} - \frac{v_iv_j}{|\bv|^2}$.

We are primarily concerned with the IPL model, for which the force
between two molecules is always repulsive and proportional to a
negative power of their distance. In this case, the function
$\Psi(\bv)$ has the form
\begin{equation}
  \label{eq:IPL}
  \Psi(\bv) := \Lambda |\bv|^{\gamma + 2},
\end{equation}
where $\Lambda >0$ is a constant and $\gamma$ is the index of the
power of distance. This equation is obtained as a limit of the
Boltzmann equation, when all the collisions become grazing
\cite{Desvillettes1992}. In the case of the Boltzmann equation,
different $\gamma$ lead to different models.  The case $\gamma > 0$
corresponds to the ``hard potential'' case, whereas for $\gamma < 0$,
it corresponds to the case of ``soft potential''. In the critical case
$\gamma = 0$, the gas molecules are referred to as ``Maxwell
molecules''. Another case of interest is when $\gamma = -3$ of the
Coulombian case, which is a very important model for applications in
plasma.

We shall focus on the numerical approximation of $\mQ[f]$, especially
when $\gamma$ is very small. Our model of approximating the collision operator 
is best illustrated in the spatially homogeneous FPL equation case, namely
\begin{equation}
  \label{eq:homo}
  \frac{\partial f}{\partial t}   = \mQ[f],
  \qquad t\in \mathbb{R}^+, \quad \bv \in \mathbb{R}^3.
\end{equation}
As a classical result in kinetic equations, the steady state solution of this equation takes
the form of the Maxwellian:
\begin{equation}
  \label{eq:general_Maxwellian}
  \mathcal{M}_{\rho,\bu,\theta}(\bv)
    := \frac{\rho}{(2\pi\theta)^{3/2}}
      \exp \left( -\frac{|\bv - \bu|^2}{2\theta} \right),
\end{equation}
where the density $\rho$, velocity $\bu$ and temperature $\theta$ 
are defined as follows
\begin{equation}
  \label{eq:macro_var}
  \rho = \int_{\bbR^3} f(t,\bv) \dd \bv, \quad \bu =
  \frac{1}{\rho}\int_{\bbR^3}\bv f(t,\bv) \dd \bv, \quad \theta =
  \frac{1}{3\rho}\int_{\bbR^3}|\bv -\bu|^2 f(t,\bv) \dd \bv. 
\end{equation} 
Moreover, the physical variables such as the heat flux $q_i$ and the
stress tensor $\sigma_{ij}$ are also of interest. They are defined as
\begin{gather*}
  q_i = \frac{1}{2}\int_{\bbR^3}|\bv -\bu|^2(v_i-u_i)f\dd \bv, \qquad i = 1, 2, 3, \\
\sigma_{ij} = \int_{\bbR^3}
  \left( (v_i-u_i)(v_j -u_j) - \frac{1}{3} \delta_{ij} |\bv-\bu|^2 \right) f
\dd \bv, \quad i,j = 1,2,3. 
\end{gather*}
Similar to the Boltzmann equation, the collision operator preserves in
time the macroscopic quantities mass, momentum and energy. Therefore,
those are invariant quantities under evolution, and
\eqref{eq:macro_var} holds for any $t$. Thus we can
obtain
\begin{equation}
  \label{eq:u_theta}
  \rho = 1, \quad \bu  = 0, \quad \theta = 1,
\end{equation}
by selecting proper frame of
reference and applying appropriate non-dimensionalization.
Now the Maxwellian \eqref{eq:general_Maxwellian} is simply reduced to
\begin{equation}
  \label{eq:Maxwellian}
  \mathcal{M}(\bv) := \frac{1}{(2\pi)^{3/2}}
    \exp \left( -\frac{|\bv|^2}{2} \right).
\end{equation}
The heat flux and stress tensor are reduced into
\begin{gather*}
  q_i = \frac{1}{2}\int_{\bbR^3}|\bv|^2v_if\dd \bv, \qquad \sigma_{ij}
  =\int_{\bbR^3} \left( v_iv_j - \frac{1}{3} \delta_{ij} |\bv|^2
  \right) f \dd \bv, \quad i,j = 1,2,3.
\end{gather*}
The normalization \eqref{eq:u_theta} shall always be
assumed in the following context.

In the literature, the complicated
form of the collision operator $\mQ[f]$ is handled by introducing 
approximations of less complexity. For instance, for the Maxwell molecules
with $\Lambda=1$, if the distribution function $f$ is radially
symmetric, which is a property to be preserved under time evolution,
the collision operator can be rewritten as
\begin{equation} \label{eq:Linear}
\mQ^{\mathrm{linear}}[f] =
(D - 1) \nabla_{\bv} \cdot (\nabla f + f \bv), 
\end{equation}
which was proposed by C. Villani \cite{Villani1998on} . Here $D$ is the
dimension of the velocity space, and we always set $D =3$ in the
context. In this case, the FPL equation is reduced into the linear
Fokker-Planck equation (FP), which can be used to describe the
relation of Brownian molecules in a gas.

Due to the complex form of the FPL operator, several numerical
approaches are devoted to the simpler diffusive Fokker-Planck model or
on the reduced collision models \cite{POTAPENKO1999115,
  BEREZIN1987163}. Hence it is of high necessity to develop efficient numerical methods for
the original FPL equation with quadratic collision operator.

\subsection{Series expansion of distribution function}
\label{sec:expansion}
Our numerical discretization shall be based on the series
expansion in the weighted $L^2$ space of the distribution function 
$\mF = L^2(\bbR^3; \mM^{-1}\dd\bv)$:
\begin{equation} 
  \label{eq:expansion}
  f(t, \bv) = \sum_{|\balpha|=0}^{+\infty}f_{\balpha}(t)H^{\balpha}(\bv)\mM(\bv), 
\end{equation}
where $\mM(\bv)$ is the Maxwellian, and $\balpha=(\alpha_1, \alpha_2,
\alpha_3)^T$ is a  three-dimensional multi-index, and $|\balpha| = \alpha_1
+ \alpha_2 + \alpha_3$. 
In \eqref{eq:expansion}, $H^{\balpha}(\bv)$ are the Hermite
polynomials defined as follows:
\begin{definition}[Hermite polynomials]
  For $\alpha_i \in \mathbb{N}, i = 1,2, 3$, the Hermite polynomial
  $H^{\balpha}(\bv)$ is defined as
  \begin{equation}
    \label{eq:basis}
     H^{\balpha}(\bv) = \frac{(-1)^n}{\mathcal{M}(\bv)}
     \frac{\partial^{|\balpha|}}{\partial v_1^{\alpha_1} \partial
       v_2^{\alpha_2} \partial v_3^{\alpha_3}} \mathcal{M}(\bv),
  \end{equation}
  where $\mathcal{M}(\bv)$ is given in \eqref{eq:Maxwellian}.
\end{definition}
The expansion \eqref{sec:expansion} was introduced to solve Boltzmann
equations \cite{Grad}, where such an expansion was invoked to
deploy moment methods. We can derive moments
 based on the coefficients $f_{\balpha}$  from the orthogonality
of Hermite polynomials
\begin{equation}
  \label{eq:Her_orth}
  \int_{\bbR^3}H^{\balpha}(\bv)H^{\bbeta}(\bv)\mM(\bv)  \dd \bv
  = \delta_{\balpha,\bbeta} \balpha!,
\end{equation}
where $\delta_{\balpha, \bbeta}$ is defined as
$\delta_{\balpha,\bbeta} =
\prod\limits_{i=1}^3\delta_{\alpha_i,\beta_i}$
and $\balpha ! = \prod\limits_{i=1}^3\alpha_i!$. For example, by the
orthogonality aforementioned, we can insert the expansion
\eqref{eq:expansion} into the definition of $\rho$ in
\eqref{eq:macro_var} to get $f_{\boldsymbol{0}} = \rho$, where
$\boldsymbol{0} = (0, 0,0)$. In our case, the normalization
\eqref{eq:u_theta} gives us $f_{{\boldsymbol{0}}} = 1$. In a similar manner, we can see
from the other two equations in \eqref{eq:macro_var}
and \eqref{eq:u_theta} that
\begin{equation}
  f_{\be_i} = 0, \quad i = 1, 2,3, \qquad \sum_{i=1}^3f_{2\be_i}= 0, 
\end{equation}
where $\be_i$ is a three dimensional index whose $i$-th entry equals
$1$ and other entries equal zero.  The heat flux and stress tensor are
related to the coefficients by
\begin{displaymath}
q_j =  2 f_{3\be_j} + \sum_{k=1}^3f_{\be_j + 2\be_k}, \qquad \sigma_{ij} =
(1+\delta_{ij})f_{\be_i + \be_j}, \qquad i, j = 1, 2,3.
\end{displaymath}

%%% Local Variables: 
%%% mode: latex
%%% TeX-master: "article"
%%% End: 

% vim: tw=70:spell
\section{Approximation of quadratic collision term}
\label{sec:general}
In order to investigate the evolution of the coefficients
$f_{\balpha}$ in the expansion \eqref{eq:expansion}, we shall expand
the collision term under the same function space. The expansion of
collision operator of the linear type is rather straightforward. As
an example, the explicit form of expansion of \eqref{eq:Linear} in
three dimensional case is
\begin{equation}
  \label{eq:linear_exp}
  \mathcal{Q}^{\mathrm{linear}}[f] = 
  \sum_{|\balpha|=0}^{+\infty}Q_{\balpha}^{\mathrm{linear}} H^{\balpha}(\bv)\mM(\bv),
  \qquad  Q_{\balpha}^{\mathrm{linear}} = -(D-1)|\balpha| f_{\balpha}. 
\end{equation}
which comes as a consequence of the property that Hermite polynomials
can diagonalize the linear FP operator.  It is also intrinsically
implied by the fact that Fokker-Planck equation can be used in the
context of stochastic process while Hermite polynomials play a crucial
role in Brownian motion, but we shall not take the stochastic
perspective here.
 
We shall first discuss the series expansion of the quadratic collision
term $\mQ[f]$ defined in \eqref{eq:quad_col}, and then combine the
quadratic result with the linear-type collision operators to construct
collision models with better accuracy and less computational
complexity.

\subsection{Series expansions of quadratic collision terms}
Suppose the binary collision term $\mQ[f]$ is expanded into the following form
\begin{equation}
  \label{eq:S_expan}
  \mQ[f](\bv) = \sum_{|\balpha|=0}^{+\infty}
  Q_{\balpha}H^{\balpha}(\bv)\mM(\bv).
\end{equation}
Due to the orthogonality of Hermite polynomials, we get
\begin{equation}
\label{eq:S_k1k2k3}
Q_{\balpha} 
= \frac{1}{\balpha!}\int H^{\balpha}(\bv)\mQ[f](\bv) \dd \bv  
=\sum\limits_{|\blambda|=0}^{+\infty}\sum\limits_{|\bkappa|=0}^{+\infty} 
  A_{\balpha}^{\blambda, \bkappa}
  f_{\blambda}f_{\bkappa},
\end{equation}
where the last equality can be derived by inserting
\eqref{eq:expansion} into \eqref{eq:quad_col}, and
\begin{equation}
  \label{eq:coeA_detail}
  \begin{aligned}
    &    A_{\balpha}^{\blambda,\bkappa} = \frac{1}{\balpha!} 
    \int_{\mathbb{R}^3} H^{\balpha}(\bv) \nabla_{\bv}\cdot
    \int_{\mathbb{R}^3} A(\bv -
    \bv_{\ast})  \\
    &\qquad \Big(H^{\blambda}(\bv_{\ast})
    \mM(\bv_{\ast})\nabla_{\bv}\left(H^{\bkappa}(\bv)\mM(\bv)\right)
    -H^{\blambda}(\bv)\mM(\bv)\nabla_{\bv_{\ast}}
    \left(H^{\bkappa}(\bv_{\ast})\mM(\bv_{\ast})\right)\Big) \dd
    \bv_{\ast} \,\mathrm{d}\bv.
\end{aligned}
\end{equation}
The above formula is of an extremely complex form, with the evaluation
of every single coefficient requiring a six dimensional integration, as well as
differential operations. Granted this can be computed by numerical quadrature; 
the computational cost would be unbearble for getting all these coefficients.
Recently, in \cite{ExactPfefferle2017}, a strategy to simplify the
above integral is introduced for the Coulombian case
$\gamma = -3$, and the explicit values are given with 
small indices. In order to deal with this integral, we give the explicit
expressions of all the coefficients $A_{\balpha}^{\blambda, \bkappa}$
and enlarge the applicable region of these expressions to
$\gamma > -5$ for the quadratic collision kernel, which incorporates
the domain of definition for $\gamma$ in the IPL model. The main
results are summarized in the following theorem:

\begin{theorem} 
  \label{thm:coeA}
  The expansion coefficients of the collision operator $\mQ[f](\bv)$
  defined in \eqref{eq:S_k1k2k3} have the following form:
  \begin{equation}
    \label{eq:coeA}
    \begin{aligned}
      A_{\balpha}^{\blambda,\bkappa} =& 2^{(\gamma + 3 -
        |\balpha|)/2}\sum\limits_{s,t = 1}^3
      \sum\limits_{|{p}|=0}^{|\balpha|-1}
      \frac{\Lambda}{{q^{[s]}}!}
      \left(a_{p, r^{[t]}}^{\bkappa+\be_t,\blambda} -
        a_{p, r^{[t]}}^{\lambda, \kappa+e_t}\right)
      B_{r^{[t]}}^{q^{[s]}}(\gamma, s, t),
    \end{aligned}
  \end{equation}
  where $p = (p_1, p_2, p_3)^T$ is a three-dimensional multi-index and
  \begin{equation}
    \label{eq:ijl_A}
    q^{[s]} = \alpha -  e_s - p,
    \quad r^{[t]}= \lambda + \kappa + e_{t}  -
    p, \quad a_{p,q}^{\lambda, \kappa} 
    =  \prod_{i=1}^3a_{p_i q_i}^{\lambda_i\kappa_i},  \qquad  s, t = 1, 2, 3.    
  \end{equation}
  the sum is taken for the indices in the range if and only if each
  subindex is non-negative.
  
The coefficients $a_{pq}^{\lambda\kappa}$ and
$B_{p}^{q}(\gamma, s, t)$ are defined by
  \begin{equation}
    \label{eq:coea}
    a_{pq}^{\lambda\kappa} = 2^{-(p+q)/2} \lambda! \kappa!
    \sum_{s=\max(0,p-\kappa)}^{\min(p,\lambda)}
    \frac{(-1)^{q-\lambda+s}}{s!(\lambda-s)!(p-s)!(q-\lambda+s)!},
  \end{equation}
  and
  \begin{equation}
    \label{eq:coe_gamma}
    B_{p}^{q}(\gamma, s, t) := 
    -G_{st}(\gamma, p, q) +
    \delta_{st}\sum_{r = 1}^3G_{rr}(\gamma,  p, q), 
  \end{equation}
  where
  \begin{equation}
    \label{eq:eta}
    G_{st}(\gamma, p, q) = \int_{\bg \in
      \bbR^3} |\bg|^{\gamma}
    g_sg_tH^{p}(\bg)H^{q}(\bg)\mM(\bg) \dd \bg,
    \qquad s,t = 1, 2,3.
  \end{equation}
\end{theorem}

The proof of Theorem \ref{thm:coeA} can be found in Appendix
\ref{sec:Appendix_coeA}. Hence, we only have to compute
\eqref{eq:eta}. When $\gamma > -3$, it can be computed directly by the
recursive formula of the Hermite Polynomials following the method in
\cite{QuadraticCol}. However, for the Coulombian case $\gamma = -3$,
the recursive formula can not be adopted directly due to the
singularity induced by the small value of $\gamma$. In
\cite{ExactPfefferle2017}, the Coulombian case $\gamma = -3$ is
evaluated by adopting the special form of the quadratic collision term
there. In the next section, we will introduce a new
method to deal with the super singularity for a large region of
$\gamma$.

\subsection{Derivation of exact coefficients in super singular integral}
In this section, we will introduce a different method to calculate these
coefficients exactly and the applicable area of $\gamma$ is 
enlarged as well. In order to deal with the singularity , Burnett polynomials,
products of Sonine polynomials and solid spherical harmonics
\cite{Ikenberry}, are utilized here. Burnett polynomials are
introduced in \cite{Burnett1936} to approximate the distribution
function of Boltzmann equation, and was adopted in \cite{BurnettCol,
  Gamba2018} to reduce the quadratic collision operator. To be
concrete, the normalized form of the Burnett polynomials is
\begin{displaymath}
  B_{\hat{\balpha}}(\bv) = \sqrt{\frac{2^{1-\hat{\alpha}_1} \pi^{3/2} \hat{\alpha}_3!}
    {\Gamma(\hat{\alpha}_3+\hat{\alpha}_1+3/2)}}
  L_{\hat{\alpha}_3}^{(\hat{\alpha}_1+1/2)} \left( \frac{|\bv|^2}{2} \right) |\bv|^{\hat{\alpha}_1}
  Y_{\hat{\alpha}_1}^{\hat{\alpha}_2} \left( \frac{\bv}{|\bv|} \right), 
\end{displaymath}
where the index $\hat{\balpha}$ is defined as
$$\hat{\balpha} = (\hat{\alpha}_1, \halpha_2, \halpha_3)^T,
\quad \halpha_1,\halpha_3 \in \bbN, \quad \halpha_2 =
-\halpha_1,\cdots,\halpha_1.$$
Here $L_n^{(\beta)}(x)$ is the Laguerre polynomials
\begin{displaymath}
L_n^{(\beta)}(x) = \frac{x^{-\beta} \exp(x)}{n!}
  \frac{\mathrm{d}^n}{\mathrm{d}x^n}
  \left[ x^{n+\beta} \exp(-x) \right],
\end{displaymath} 
and $Y_l^m(\bn)$ is spherical harmonics
\begin{displaymath}
Y_l^m(\bn) = \sqrt{\frac{2l+1}{4\pi} \frac{(l-m)!}{(l+m)!}}
  P_l^m(\cos \theta) \exp(\mathrm{i} m \phi), \qquad
\bn = (\sin \theta \cos \phi, \sin \theta \sin \phi, \cos \theta)^T
\end{displaymath}
with $P_l^m$ the associate Legendre polynomial
\begin{displaymath}
P_l^m(x) = \frac{(-1)^m}{2^l l!} (1-x^2)^{m/2}
  \frac{\mathrm{d}^{l+m}}{\mathrm{d}x^{l+m}} \left[ (x^2-1)^l \right].
\end{displaymath}
By the orthogonality of Laguerre polynomials and spherical harmonics,
one can find that
\begin{equation}
\label{eq:orthogonality_Burnett}
\int_{\bbR^3} \overline{B_{\hat{\balpha}}(\bv)} B_{\hat{\bbeta}}(\bv)
  \mathcal{M}(\bv) \dd\bv =
\delta_{\hat{\balpha}, \hat{\bbeta}}.
\end{equation}

In order to reduce complexity, the symmetry of
$G_{st}(\gamma, \bp, \bq)$, which is stated in Lemma \ref{thm:G}, is
utilized first to reduce the cost of computation and storage.

\begin{lemma}
\label{thm:G}
For the expressions $G_{st}(\gamma, \bp, \bq)$, it holds that 
\begin{equation}
  \label{eq:G_1}
  G_{st}(\gamma, \bp, \bq)=G_{ts}(\gamma, \bp, \bq), \qquad s, t = 1,
  2, 3,
\end{equation}
and 
\begin{equation}
\begin{split}
 G_{11}(\gamma, \bp,  \bq)= G_{22}(\gamma, \Pi_{2}^1\bp,
 \Pi_{2}^1\bq) =  G_{33}(\gamma, \Pi_{3}^1\bp,
 \Pi_{3}^1\bq), \\
 G_{12}(\gamma, \bp,  \bq)= G_{13}(\gamma, \Pi_{3}^2\bp,
  \Pi_{3}^2\bq)  = G_{23}(\gamma, \Pi_{3}^1\bp,
  \Pi_{3}^1\bq).
\end{split}
 \end{equation}
 Here $\Pi_{i}^jp$ is a permutation operator which exchanges the
 $i$-th and $j$-th entries of $p$.
\end{lemma}

Based on Lemma \ref{thm:G}, we only have to compute two cases
$G_{33}(\gamma, \bp, \bq)$ and $G_{13}(\gamma, \bp, \bq)$.  In order
to handle the singularity in $G_{st}(\gamma, \bp, \bq)$, Hermite
polynomials in \eqref{eq:eta} is expressed by a linear combination of
the Burnett polynomials, precisely
\begin{equation}
  \label{eq:Her_Bur}
  H^{\balpha}(\bv) = \sum_{|\hat{\balpha}|_B = |\balpha|}
  C_{\hat{\balpha}}^{\balpha}B_{\hat{\balpha}}(\bv), \qquad   C_{\hat{\balpha}}^{\balpha} =
  \int_{\bbR^3}B_{\hat{\balpha}}(\bv)H^{\balpha}(\bv)\mM(\bv) \dd \bv,
\end{equation}
where $|\hat{\balpha}|_B = \halpha_1 + 2\halpha_3.$ Since both Hermite
and Burnett polynomials are orthogonal polynomials associated with the
same weight function, thus the coefficients
$C_{\hat{\balpha}}^{\balpha}$ defined in \eqref{eq:Her_Bur} are
nonzero only when the degrees of $H^{\balpha}$ and $B_{\hat{\balpha}}$
are equal, precisely $|\hat{\balpha}|_B = |\balpha|$. The detailed
algorithm to compute the coefficient $C_{\hat{\balpha}}^{\balpha}$ can
be found in \cite{BurnettCol} and we also explain that briefly in
Appendix \ref{sec:Appendix_CoeC}.  With the help of the Burnett
polynomials, we can finally get the exact value of
$G_{st}(\gamma, \bp, \bq).$ 

\begin{proposition}
  When $\gamma > -5$, $G_{st}(\gamma, \bp, \bq)$ defined in
  \eqref{eq:eta} can be simplified as
\begin{equation}
  \label{eq:detail_D}
  \begin{aligned}
    G_{st}(\gamma, \bp, \bq) & = 2^{(\gamma + 2)/2}\sum_{|\hat{\bp}|_B
      = |\bp|}\sum_{|\hat{\bq}|_B =
      |\bq|}C_{\hat{\bp}}^{\bp}C_{\hat{\bq}}^{\hat{\bq}}
    D_{\hat{p}_3,\hat{q}_3}^{\hat{p}_1\hat{q}_1} \\
    & K\left(\frac{\gamma + \hat{p}_1 + \hat{q}_1+3}{2}, \hat{p}_1 +
      \frac{1}{2}, \hat{q}_1 + \frac{1}{2}, \hat{p}_3,
      \hat{q}_3\right) F_{st}(\hat{p}_1, \hat{p}_2, \hat{q}_1,
    \hat{q}_2),
  \end{aligned}
\end{equation}
where
\begin{equation}
  \label{eq:F}
  F_{st}(\hat{p}_1, \hat{p}_2, \hat{q}_1, \hat{q}_2) =
  \int_{\bbS^2}n_sn_tY_{\hat{p}_1}^{\hat{p}_2}(\bn)Y_{\hat{q}_1}^{\hat{q}_2}(\bn) \dd \bn,
\qquad s, t = 1, 2, 3,
\end{equation}
where $n_s$ and $n_t$ are the $s$-th and $t$-th entries of the unit
vector $\bn$ in spherical coordinates
$\bn = (\sin \theta \cos \phi, \sin \theta \sin \phi, \cos \theta)^T$.
The parameters in \eqref{eq:detail_D} are defined as
$$D_{n_1n_2}^{l_1l_2} =  \sqrt{\frac{n_1! n_2!}{\Gamma(n_1+l_1+3/2)
    \Gamma(n_2+l_2 + 3/2)}}$$ and 
\begin{equation} 
  \begin{split}
    K(\mu, \alpha, \kappa, m, n) =
    (-1)^{m+n}\Gamma(\mu+1)\sum_{i=0}^{\min(m,n)} \binom{\mu -
      \alpha}{m - i} \binom{\mu - \kappa}{n - i} \binom{i + \mu}{i}.
  \end{split}
\end{equation}
\end{proposition}
\begin{proof}
  Substituting \eqref{eq:Her_Bur} into \eqref{eq:eta}, and adopting
  the formula introduced in \cite[eq.(10)]{Srivastava}
\begin{equation} \label{eq:no_hg}
  \begin{split}
    \int_0^{+\infty} L_{m}^{(\alpha)} (s) L_{n}^{(\kappa)} (s) s^{\mu}
    \exp(-s) \dd s = (-1)^{m+n}\Gamma(\mu+1)\sum_{i=0}^{\min(m,n)}
    \binom{\mu - \alpha}{m - i} \binom{\mu - \kappa}{n - i} \binom{i +
      \mu}{i},
  \end{split}
\end{equation}
we have thus validated this proposition. 
\end{proof}

Finally, as stated previously in Lemma \ref{thm:G}, we only need to
compute $G_{33}(\gamma, \bp, \bq)$ and $G_{13}(\gamma, \bp, \bq)$.
Therefore, only these two corresponding cases of \eqref{eq:F} are
discussed in the following theorem and the proof is presented in
Appendix \ref{sec:Appendix_coeF}.

\begin{theorem}
\label{thm:coeF}
  Define $\eta_{lm}^{\mu}$ as
  \begin{equation}
    \label{eq:coe_eta}
    \eta_{lm}^{\mu} = \sqrt{\frac{[l + (2\delta_{1, \mu} - 1)m +
        \delta_{1, \mu}] [l - (2\delta_{-1, \mu} - 1)m + \delta_{-1,
          \mu}]} {2^{|\mu|}(2l -1)(2l+1)}}.
\end{equation}
Then the coefficients $F_{13}(l_1, m_1, l_2, m_2)$ and
$F_{33}(l_1, m_1, l_2, m_2)$ have the following explicit form
\begin{equation}
  \begin{aligned}
   & F_{13}(l_1, m_1, l_2, m_2)   =
    \frac{(-1)^{m_2+1}}{\sqrt{2}}\sum_{k,j,l=0,1}
    (-1)^{l+j}\eta_{\delta_{0k}+(-1)^kl_2, m_2}^0
    \eta_{(-1)^jl_1+\delta_{0j}, m_1}^{(-1)^l}
    \delta_{l_1 + \delta_{1k} - \delta_{1j} ,
      l_2  -\delta_{1k} +\delta_{1j}}^{m_1+(-1)^l, -m_2}, \\
  &  F_{33}(l_1, m_1, l_2, m_2) =(-1)^{m_2} \sum_{k, j=0,1} \eta_{\delta_{0k}+(-1)^kl_2, m_2}^0
    \eta_{(-1)^jl_1+\delta_{0j}, m_1}^0  \delta_{l_1 + \delta_{1k} - \delta_{1j} ,
      l_2  -\delta_{1k} +\delta_{1j}}^{m_1, -m_2}.
    \end{aligned}
\end{equation}
\end{theorem}

The above analysis shows that for the FPL collision operator, the
coefficients $A_{\balpha}^{\blambda,\bkappa}$ can be calculated
exactly for $\gamma > -5$, which makes it much easier to build the
high order scheme to numerically solve FPL equation. Moreover, this
algorithm for the coefficients here is readily applicable for offline
numerical evaluation, the effectiveness of which is corroborated by
our numerical examples.

\section{Construction of novel collision model}
\label{sec:construction}
Until now, we already obtain a complete algorithm to calculate the
coefficients $A_{\balpha}^{\blambda,\bkappa}$, which can
be utilized either to discretize the quadratic  collision term or to construct new
collision models.  We will now discuss both topics.

\subsection{Discretization of homogeneous FPL equation}

We will discrete the homogeneous FPL equation
by the Galerkin spectral method in terms of the expansion of the distribution function
\eqref{eq:expansion}. For any positive integer $M$, we
define as the functional space of numerical solution
\begin{equation}
  \label{eq:set}
\mathcal{F}_M=\text{span}\{H^{\balpha}(v)M(v)|\balpha\in I_M\}
\subset \mathcal{F}=L^2(\mathbb{R}^3;\mM^{-1}dv),
\end{equation}
where
$I_M=\{(\alpha_1, \alpha_2, \alpha_3)| 0 \leqslant |\balpha| \leqslant
M, \alpha_i \in \mathbb{N},i=1,2,3\}$.
Then the semi-discrete function $f_M(t, \cdot)\in \mF_M$
satisfies
\begin{equation} \label{eq:var}
\int_{\mathbb{R}^3} \frac{\partial f_M}{\partial t} \varphi
  \mathcal{M}^{-1} \,\mathrm{d}\bv =
\int_{\mathbb{R}^3} \mQ(f_M, f_M) \varphi \mathcal{M}^{-1} \dd \bv,
  \qquad \forall \varphi \in \mF_M.
\end{equation}
Suppose
\begin{equation} 
  \label{eq:f_h}
  f_M(t,\bv) = \sum\limits_{\balpha \in I_M}
    f_{\balpha}(t)H^{\balpha}(\bv)\mM(\bv) \in \mF_M.
\end{equation}
The equations \eqref{eq:S_expan} and \eqref{eq:S_k1k2k3} imply that the
variational form \eqref{eq:var} is equivalent to the ODE
system below:
\begin{equation} \label{eq:ODE}
\frac{\mathrm{d}f_{\balpha}}{\mathrm{d}t} =
  \sum\limits_{\blambda  \in I_M} \sum\limits_{\bkappa \in I_M} 
    A_{\balpha}^{\blambda,\bkappa}f_{\blambda}f_{\bkappa},
\qquad \balpha \in I_M.
\end{equation}

We thus obtain the formulation of the ODE system in its full form \eqref{eq:ODE}, 
with the help of the exact coefficients $A_{\balpha}^{\blambda, \bkappa}$ for all
$\balpha, \blambda, \bkappa \in I_M$. 
With $M$ fixed, these coefficients need to be computed only
once, and then can be used repeatedly for multiple numerical examples.

\subsection{Approximation to general collision model}

In practical computation, the storage cost of computing the
coefficients $A_{\balpha}^{\blambda, \bkappa}$ is formidably
expensive, as the number of coefficients increases significantly with
$M$ increasing. Moreover, the computational cost $O(M^9)$ is an issue
especially when solving the spatially non-homogeneous problems.

To overcome this difficulty, the method in \cite{QuadraticCol} is
utilized to reduce the computational cost, precisely that the coefficients
$A_{\balpha}^{\blambda, \bkappa}$ for a small number $M_0$ are
computed and stored, when the computational cost for solving
\eqref{eq:ODE} is acceptable. As for $\balpha \not\in I_{M_0}$, we apply
the linear model \eqref{eq:Linear} brought up by Villani and compute as:
\begin{equation}
  \label{eq:linear_ode}
   \frac{\dd f_{\balpha}}{\dd t}  = -(D-1)|\balpha| f_{\balpha}, \qquad
   \balpha \not\in I_{M_0}.
\end{equation}
Combining \eqref{eq:ODE} and \eqref{eq:linear_ode}, we obtain a
novel collision operator 
\begin{equation}
  \label{eq:collision_new}
  \mQ^{M_0}[f] = P_{M_0}\mQ[P_{M_0}f]  - \mQ^{\rm linear}[(I - P_{M_0})f], \qquad
  \forall f \in \mF,
\end{equation}
where $P_{M_0}$ is the orthogonal projection from $\mF$ onto
$\mF_{M_0}$. After applying spectral method 
to this collision operator in the functional space $\mF_M$, 
where $M$ is chosen to be larger than ${M_0}$
the final ODE system for the new model is 
\begin{equation}
  \label{eq:new_collision}
  \frac{\dd f_{\balpha}}{\dd t} = \mQ^{M, M_0}_{\balpha}, 
\end{equation}
where 
\begin{equation}
  \label{eq:coe_newcollision}
  \mQ^{M, M_0}_{\balpha} = \left\{ 
    \begin{array}{cc}
      \sum\limits_{\blambda  \in I_{M_0}} \sum\limits_{\bkappa \in I_{M_0}} 
    A_{\balpha}^{\blambda,\bkappa}f_{\blambda}f_{\bkappa}, 
      & \balpha \in I_{M_0}, \\
      -(D-1)|\balpha| f_{\balpha}, 
      &  \balpha \in I_M \setminus I_{M_0}
    \end{array}
\right.
\end{equation}

By now, we have obtained a series of new collision models
\eqref{eq:collision_new}. It can be expected that such combination
could reduce the time cost significantly due to the simple form of the
linear FP collision operator in the Hermite basis, while at the same
time manage to maintain a high level of accuracy since the evolution
function already captures the most crucial information in coefficients
of lower order and performs a satisfactory approximation in the other
coefficients. This will be observed in the numerical examples.
%%% 

%%% Local Variables: 
%%% mode: latex
%%% TeX-master: "article"
%%% End: 

\section{Numerical examples}
\label{sec:numerical}
In this section, we shall present several results in our numerical
computation. In all of the numerical experiments, we shall adopt the newly
proposed collision operator \eqref{eq:collision_new}, and solve the
equation
\begin{displaymath}
  \frac{\partial f}{\partial t} = \mQ^{M_0}[f],
\end{displaymath}
numerically for some positive integer $M_0$. Such an equation is solved
by the Galerkin spectral method for solution defined in the functional space
$\mF_M$, with $M$ chosen to be greater than $M_0$. 
Namely we solve the system of ODE \eqref{eq:coe_newcollision}.
As for the discretization in time, we use the 4th-order Runge-Kutta method in
the examples, and the time step is chosen as $\Delta t = 0.01$. In
the examples, we shall set $\Lambda = 1.$

Finally, we would like to mention that the derivation of the expansion coefficients 
in the Hermite basis are exact in each case by mathematical derivation
 instead of numerical integration in order to achieve high accuracy.

\subsection{BKW solution}
For the Maxwell gas $\gamma = 0$, the original 
FPL  equation admits an exact solution with the following expression:
\begin{displaymath}
f(t,\bv) = (2\pi \tau(t))^{-3/2} \exp \left( -\frac{|\bv|^2}{2\tau(t)} \right)
  \left[ 1 + \frac{1-\tau(t)}{\tau(t)}
    \left( \frac{|\bv|^2}{2\tau(t)} - \frac{3}{2} \right) \right],
\end{displaymath}
where $\tau(t) = 1 - 0.4\exp\left(-4t\right)$. As a good
approximation of the initial distribution function, we use $M = 15$
($816$ degrees of freedom) in our simulation. For the visualization
purpose, we define the marginal distribution functions (MDFs)
\begin{displaymath}
  g(t,v_1) = \int_{\mathbb{R}^2} f(t,\bv) \,\mathrm{d}v_2 \,\mathrm{d}v_3, \qquad
  h(t,v_1,v_2) = \int_{\mathbb{R}} f(t,\bv) \,\mathrm{d}v_3.
\end{displaymath}
The initial MDFs are plotted in Figure \ref{fig:ex1_init}, in which
the lines for exact functions and their numerical approximation are
hardly distinguishable.

\begin{figure}[!ht]
\centering
\subfloat[Initial MDF $g(0,v_1)$\label{fig:ex1_init_1d}]{%
  \includegraphics[width=.28\textwidth]{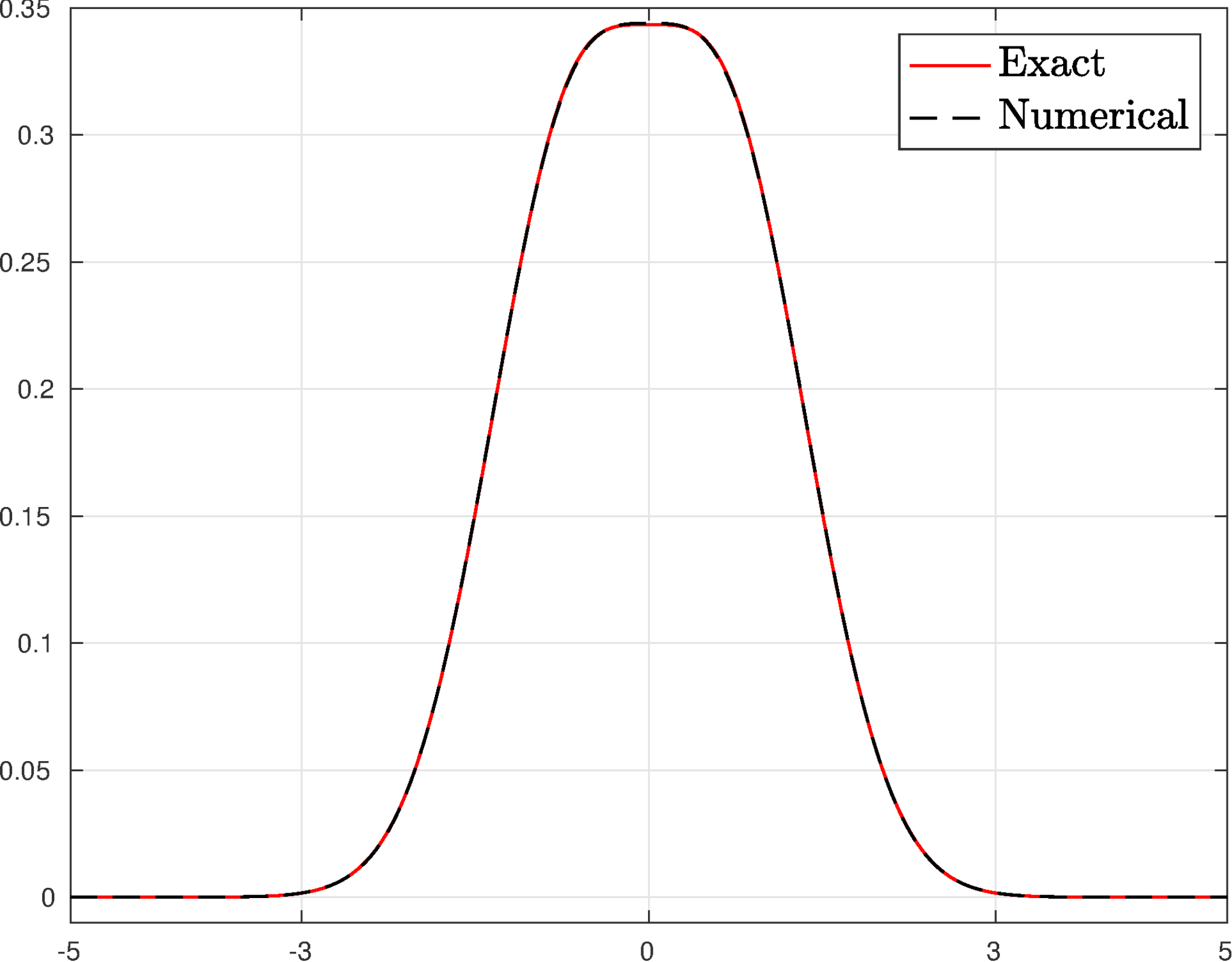}
} \quad 
\subfloat[Contours of $h(0,v_1,v_2)$\label{fig:ex1_init_2d_contour}]{%
  \includegraphics[width=.29\textwidth]{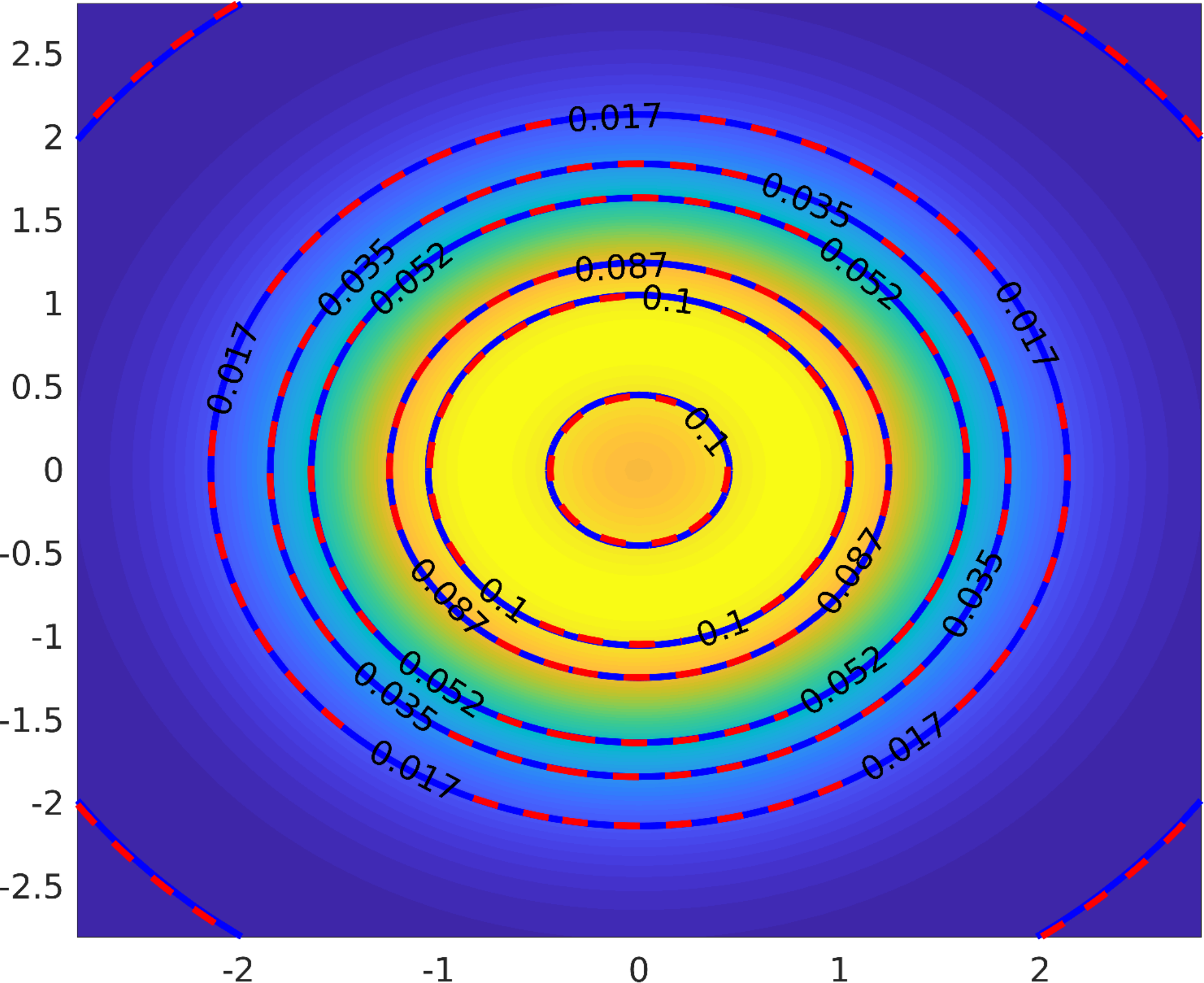}
} \quad
\subfloat[Initial MDF $h(0,v_1,v_2)$\label{fig:ex1_init_2d}]{%
  \includegraphics[width=.29\textwidth]{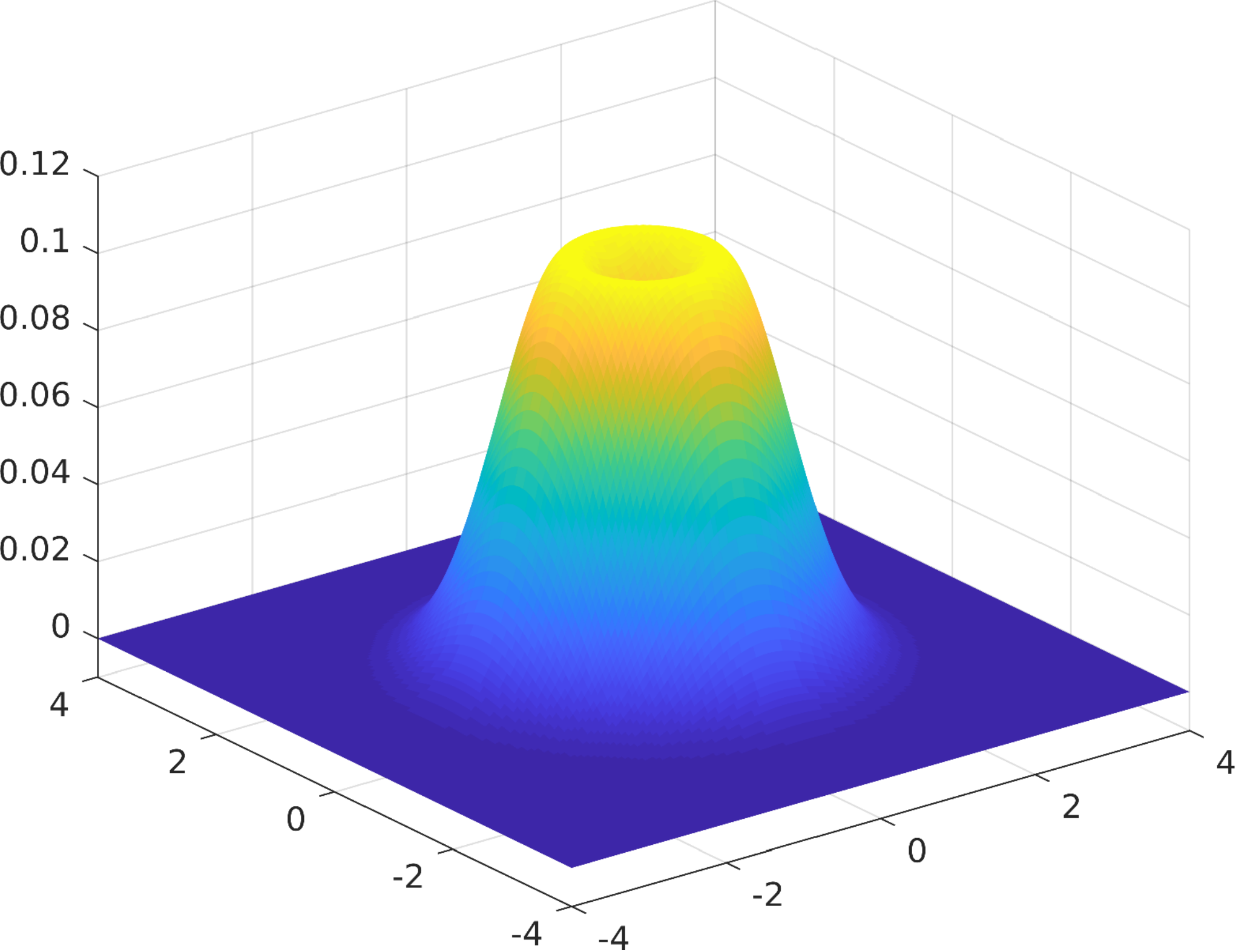}
}
\caption{Figure (a) is the initial marginal distribution functions
  $g(0, v_1)$.  The red solid line corresponds to the exact solution,
  and the blue dashed line corresponds to the numerical approximation.
  Figure (b) is the initial marginal distribution functions
  $h(0, v_1, v_2)$. The blue solid lines correspond to the exact
  solution, and the red dashed lines correspond to the numerical
  approximation. Figure (c) shows only the numerical approximation.}
\label{fig:ex1_init}
\end{figure}

Numerical results for $t = 0.01$, $0.02$ and $0.06$ are given in
Figures \ref{fig:ex1_1d} and \ref{fig:ex1_2d}, respectively for the
marginal distribution function $g(t, v_1)$ and $h(t, v_1, v_2)$. Here
we set $M_0$ as $M_0 = 5$ and $15$. For $M_0 = 5$, the numerical
solution provides a reasonable approximation, but still with
noticeable deviations, while for $M_0 = 15$, the two solutions match
perfectly in all cases.
\begin{figure}[!ht]
\centering
\subfloat[$t = 0.01$]{%
  \includegraphics[width=.3\textwidth]{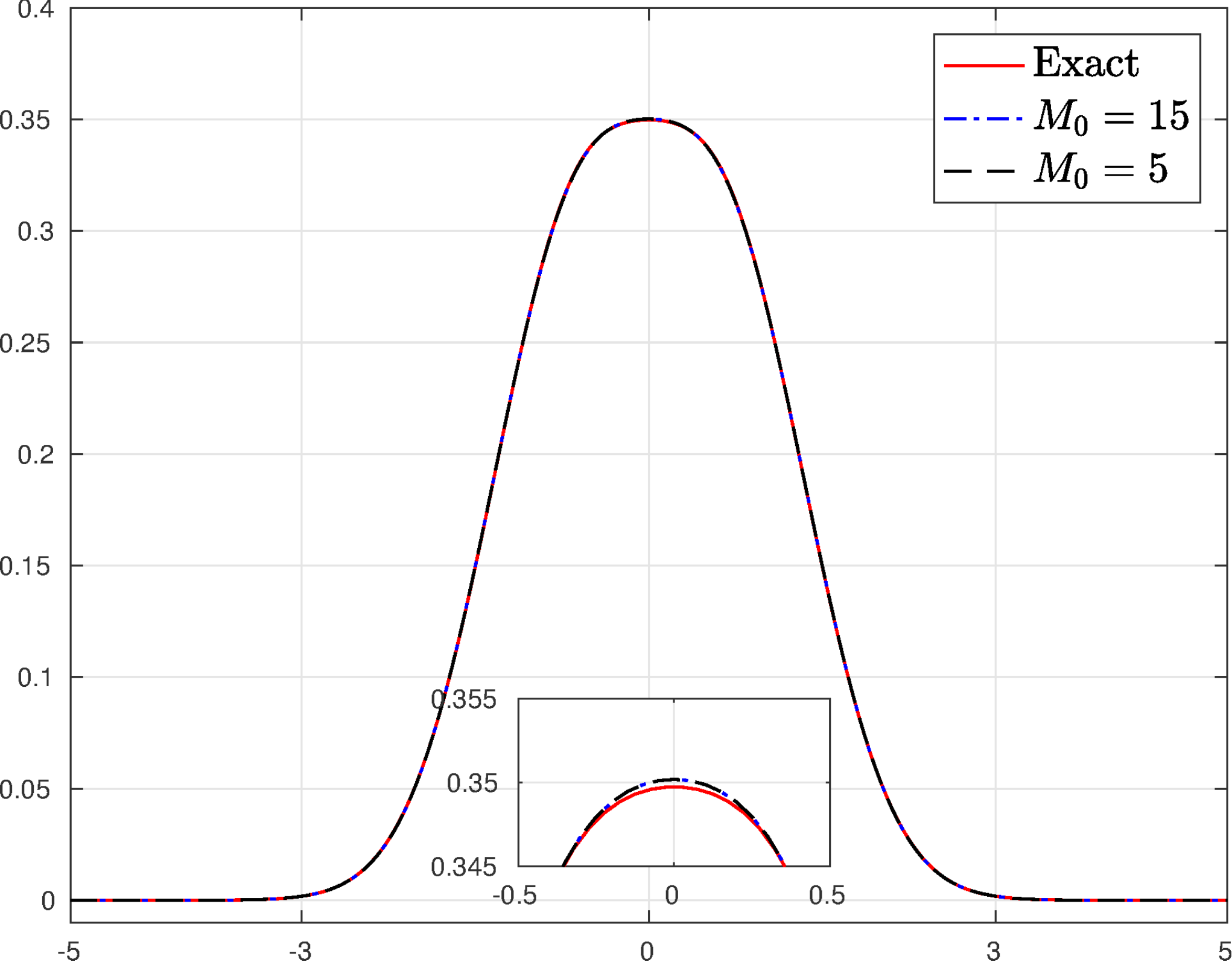}
} \hfill
\subfloat[$t = 0.02$]{%
  \includegraphics[width=.3\textwidth]{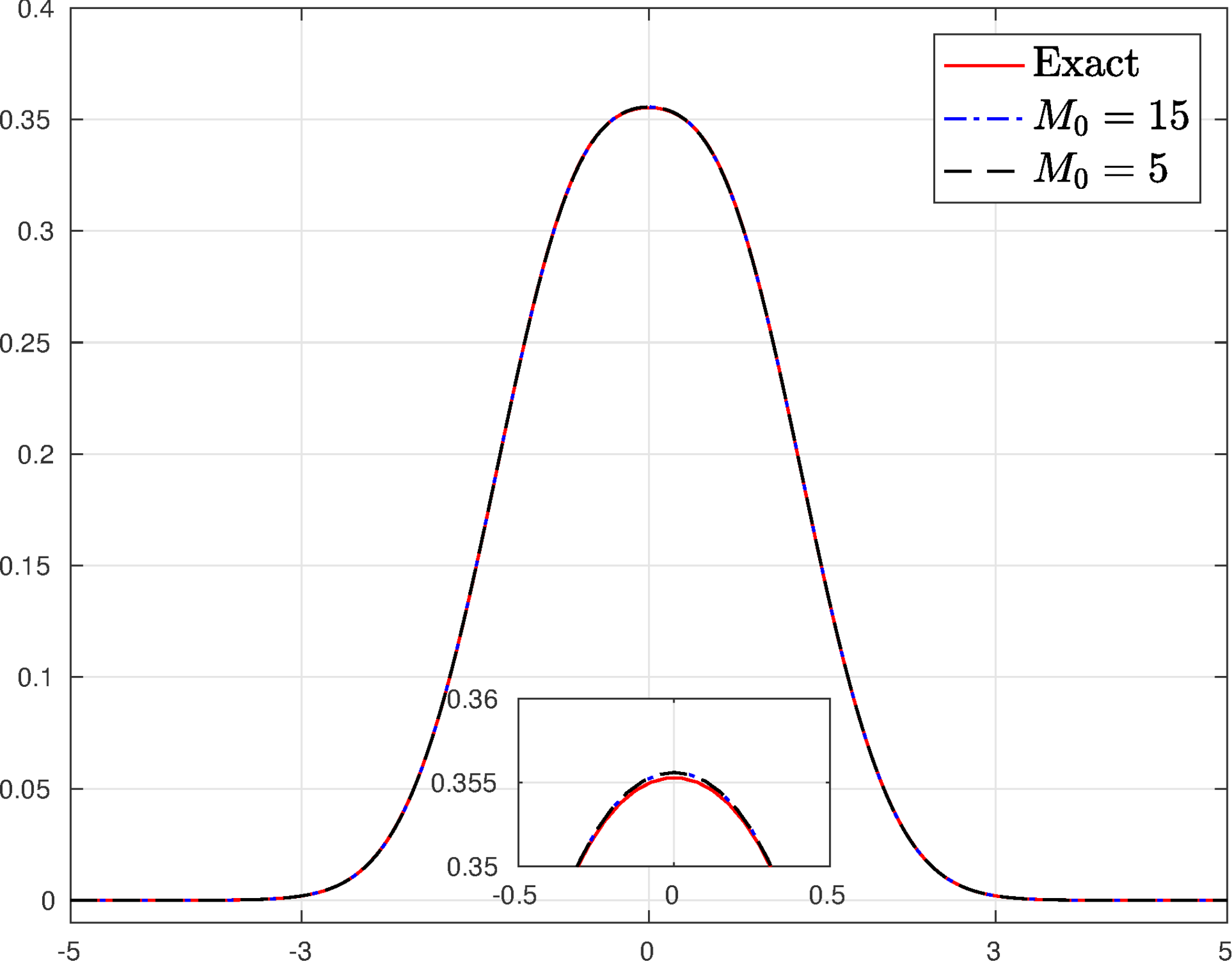}
} \hfill
\subfloat[$t = 0.06$]{%
  \includegraphics[width=.3\textwidth]{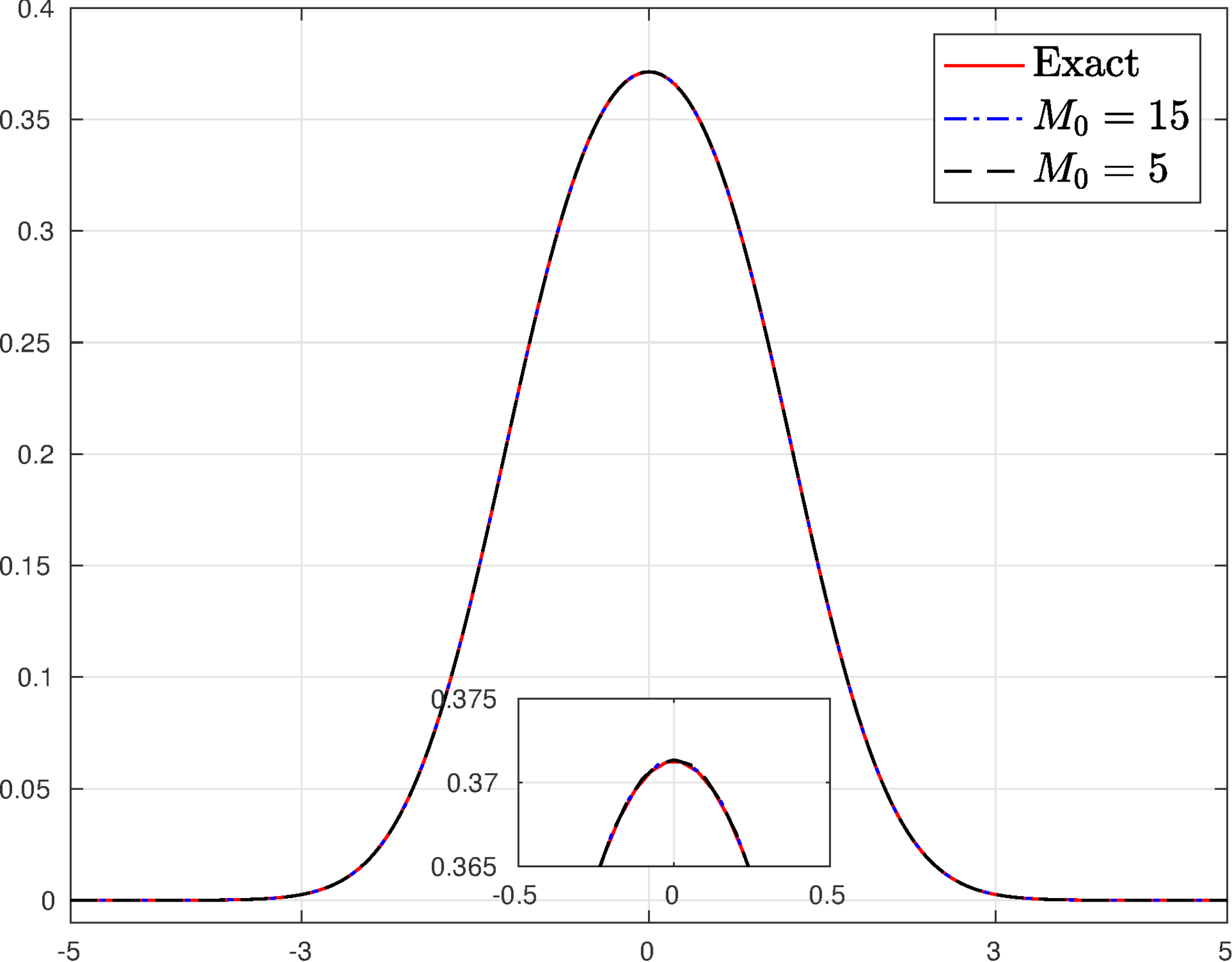}
} 
\caption{Marginal distribution functions $g(t, v_1)$ for $M_0 = 5$ and
  $15$ at $t = 0.01$, $0.02$ and $0.06$. The red solid lines
  correspond to the exact solution, and the blue dot dashed and black
  dashed lines correspond to the numerical solutions with $M_0 = 15$
  and $5$ respectively.}
\label{fig:ex1_1d}
\end{figure}

\begin{figure}[!ht]
\centering
\subfloat[$t = 0.01, M_0 = 5$]{%
  \includegraphics[width=.33\textwidth]{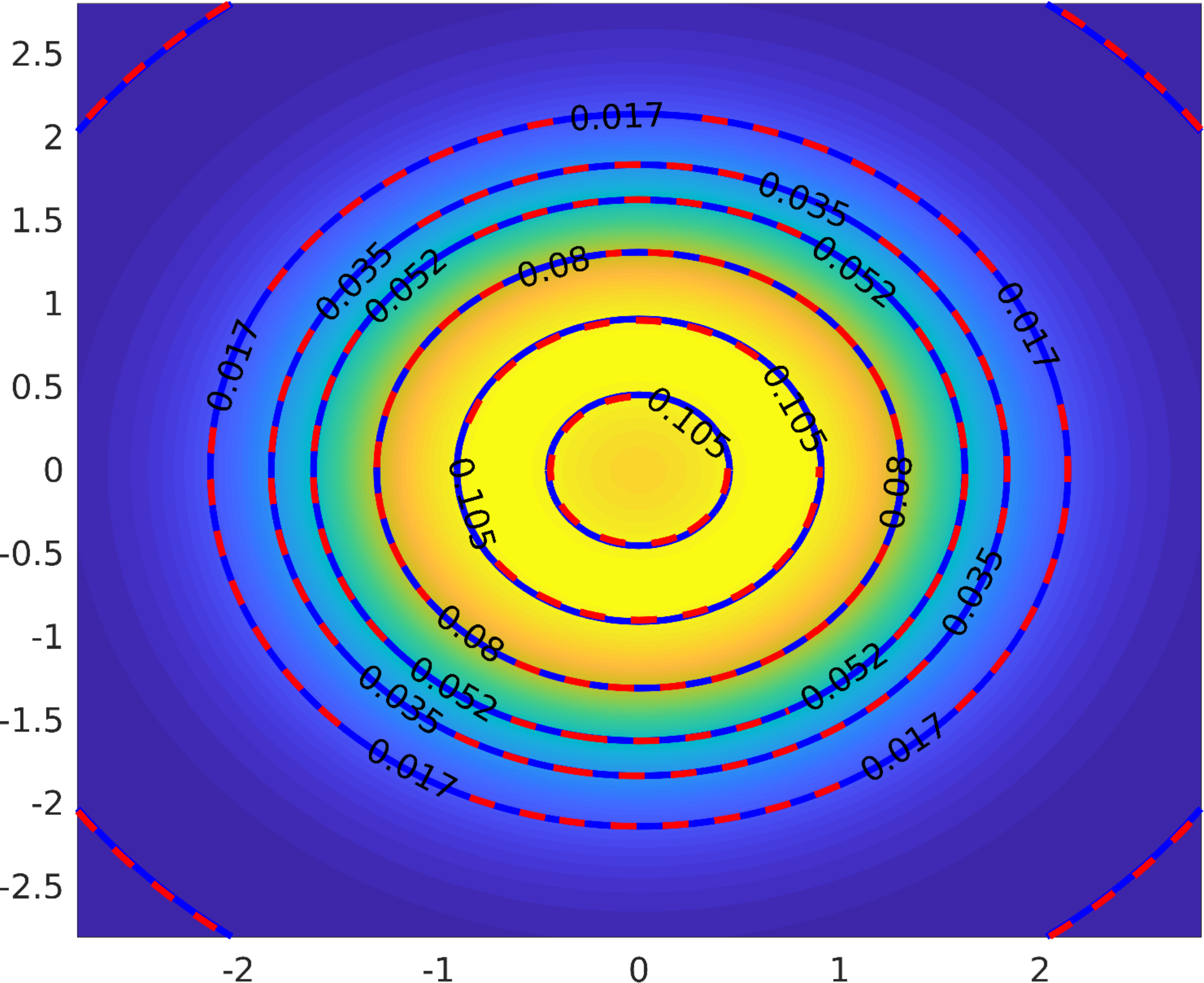}
} 
\subfloat[$t = 0.02, M_0 = 5$]{%
  \includegraphics[width=.33\textwidth]{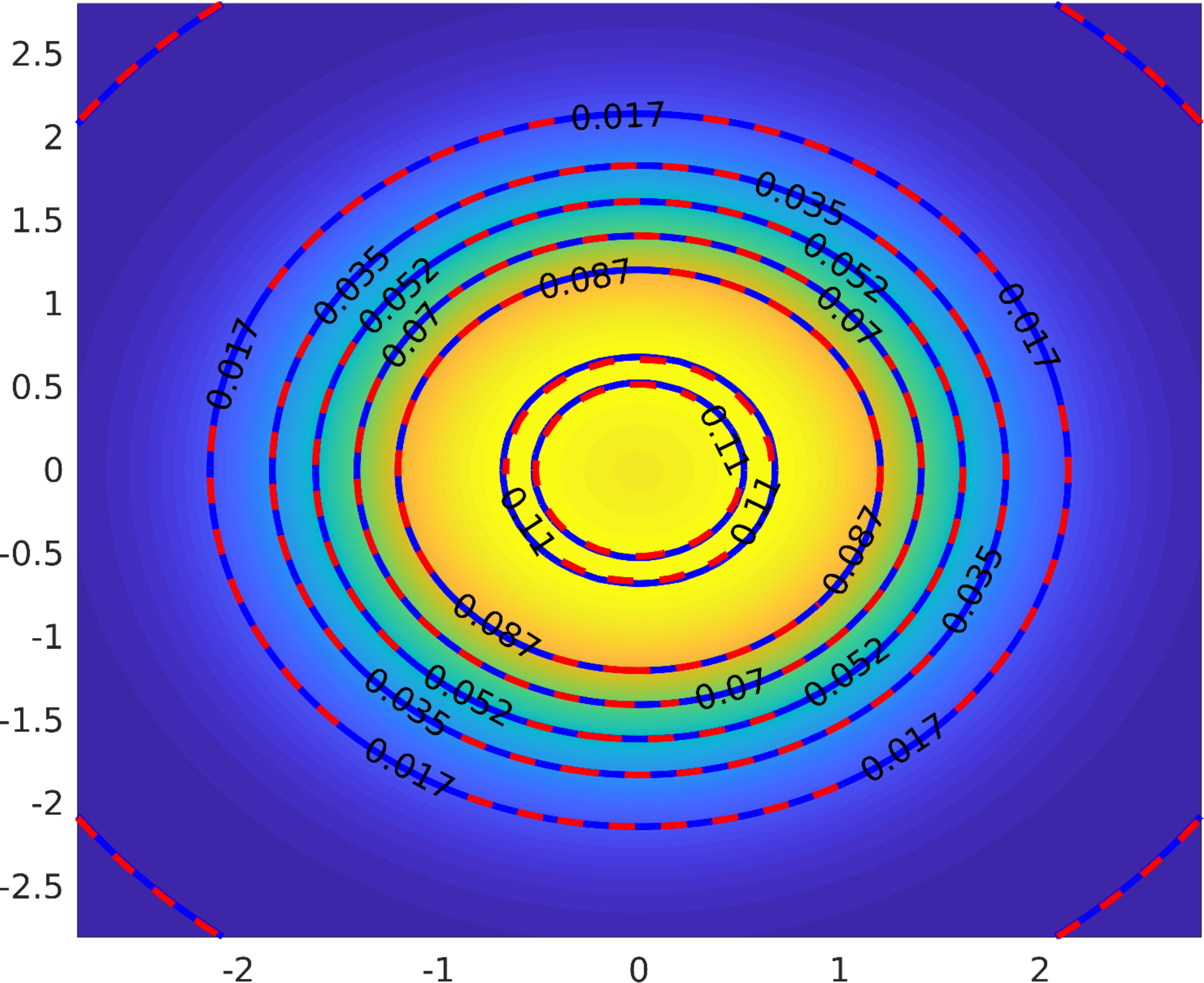}
} 
\subfloat[$t = 0.06, M_0 = 5$]{%
  \includegraphics[width=.33\textwidth]{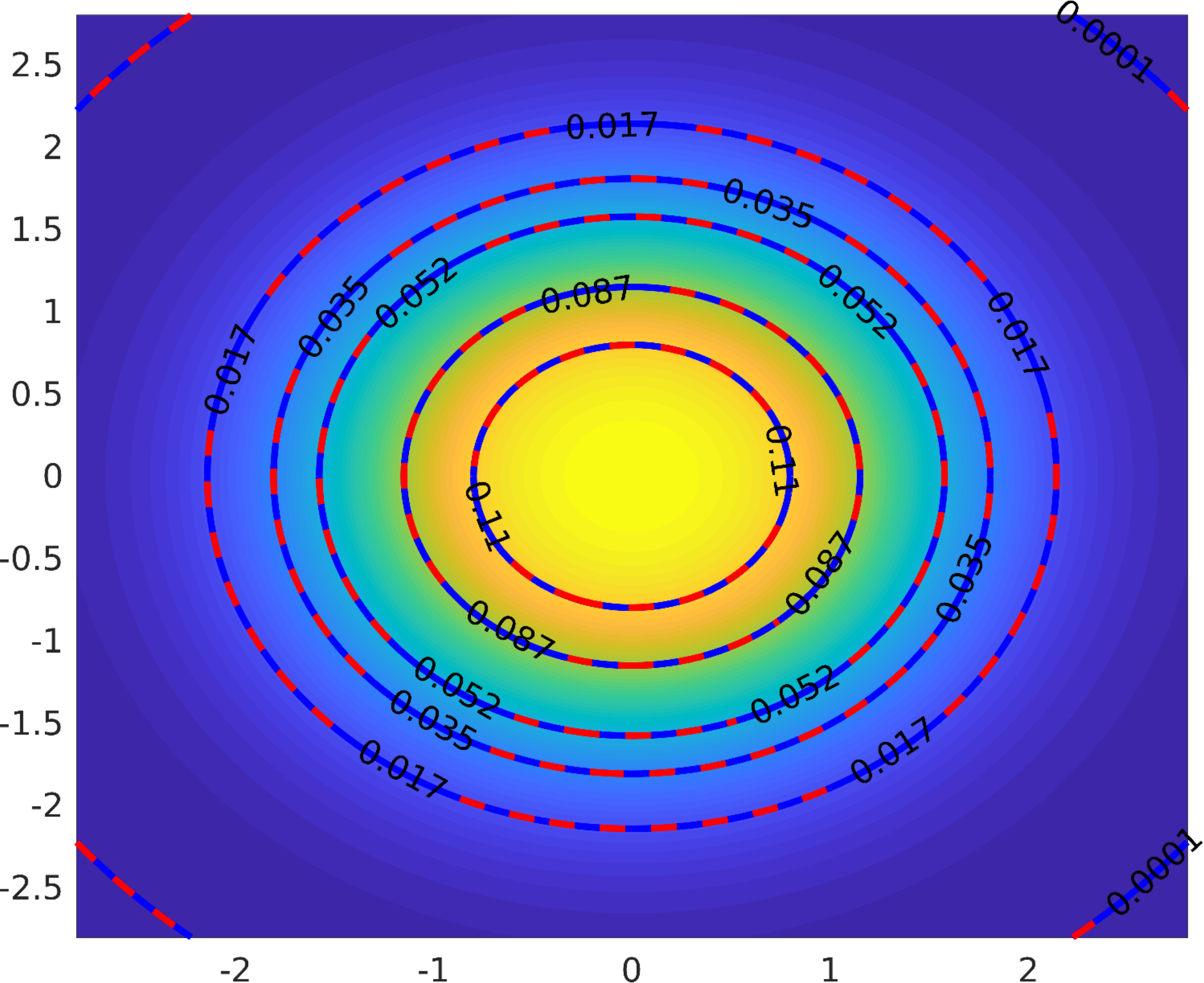}
} \\
\subfloat[$t = 0.01, M_0 = 15$]{%
  \includegraphics[width=.33\textwidth]{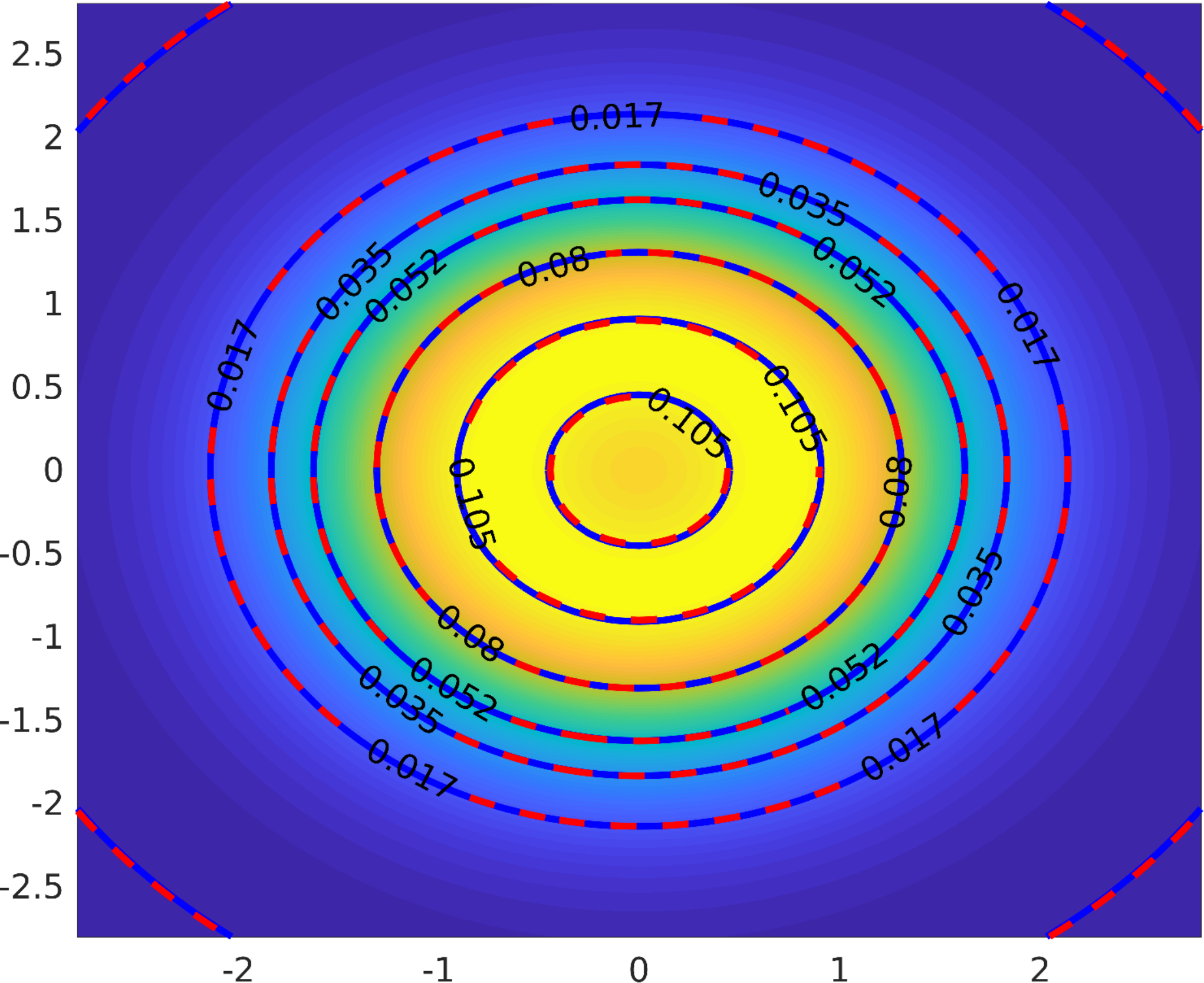}
} 
\subfloat[$t = 0.02, M_0 = 15$]{%
  \includegraphics[width=.33\textwidth]{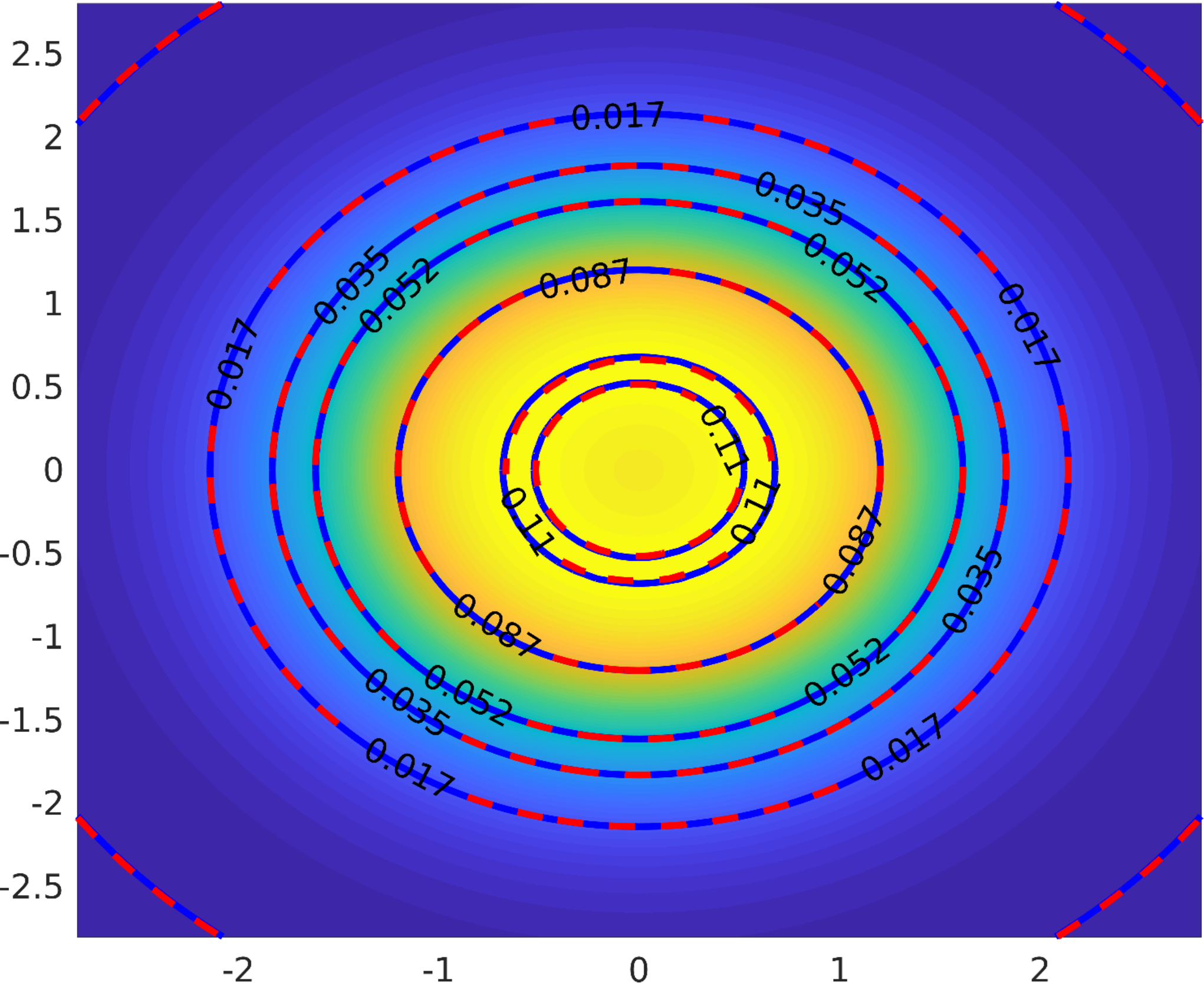}
} 
\subfloat[$t = 0.06, M_0 = 15$]{%
  \includegraphics[width=.33\textwidth,clip]{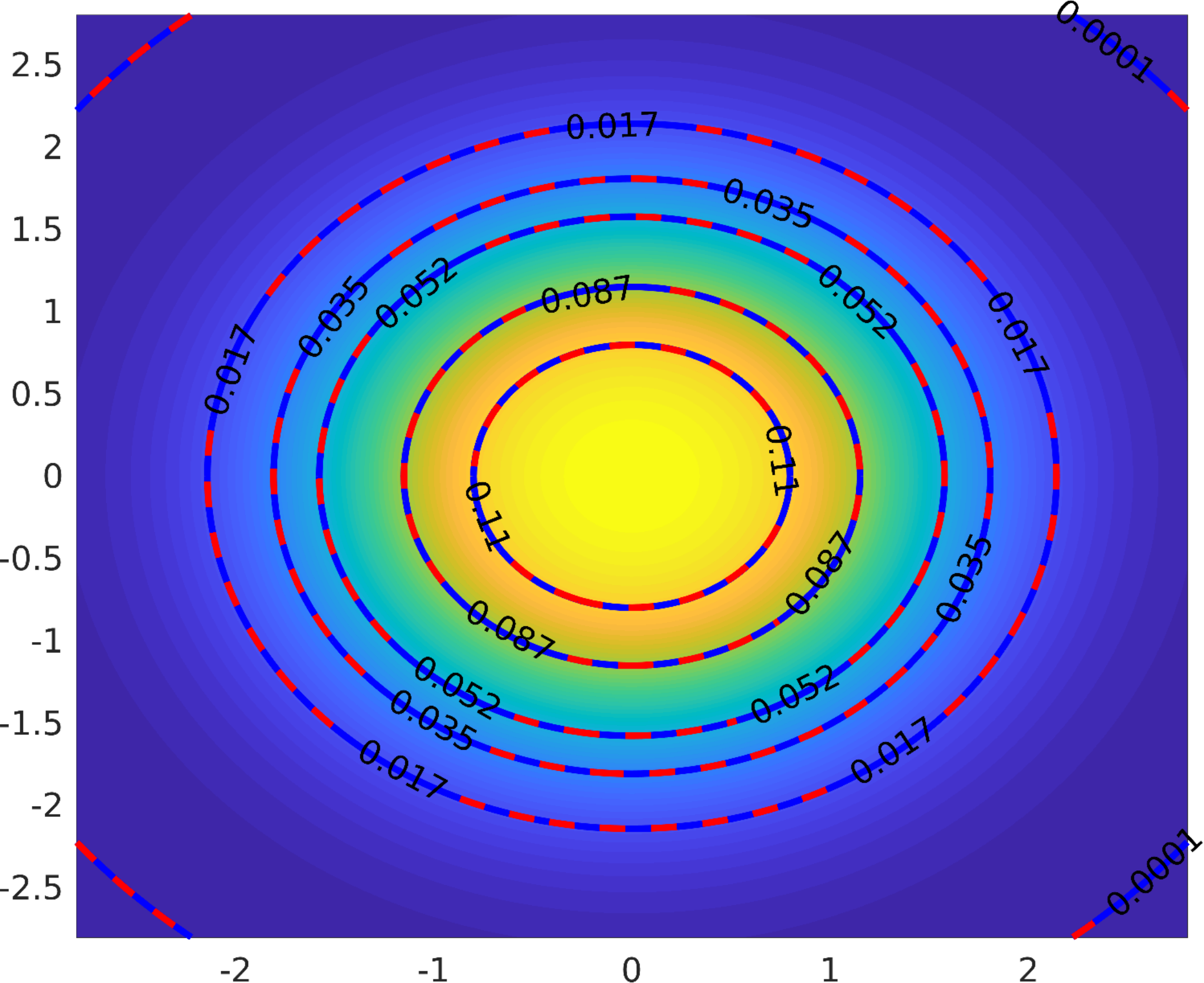}
} 
\caption{Marginal distribution functions $h(t, v_1, v_2)$ for
  $M_0 = 5$ and $15$ at $t = 0.01$, $0.02$ and $0.06$. The red dashed lines
  correspond to the exact solution, and the blue solid lines at different
  columns correspond to the numerical solutions $M_0 = 5$ and
  $M_0 = 15$. }
\label{fig:ex1_2d}
\end{figure}

Now we consider the time evolution of the expansion coefficients. By
expanding the exact solution into Hermite series, we get the exact
solution for the coefficients:
\begin{equation}
\label{eq:ex1_true}
f_{\alpha}(t) = \left\{\begin{array}{ll}
  \left[
    -0.2\exp \left( -4 t \right)
  \right]^{\frac{|\alpha|}{2}}
    \dfrac{1 - |\alpha|/2}{(\alpha_1/2)!(\alpha_2/2)!(\alpha_3/2)!}, &
  \text{if } \alpha_1, \alpha_2, \alpha_3 \text{ are even}, \\[13pt]
  0, & \text{ otherwise}.
\end{array} \right.
\end{equation}
From \eqref{eq:ex1_true}, we can find that the coefficients
$f_{\alpha}$ are zero for any $t$ if
$1 \leqslant |\alpha| \leqslant 3$. Hence we will focus on the
coefficients $f_{400}$ and $f_{220}$ here.
% For Maxwell
% molecules, the discrete kernel $A_{\alpha}^{\lambda,\kappa}$ is
% nonzero precisely when $|\alpha| = |\lambda| + |\kappa|$.
% Therefore, for any $M \geqslant M_0 \geqslant 4$, the numerical
% results for these two coefficients $f_{400}$ and $f_{220}$ are
% exactly the same (disregarding the round-off errors).
Figure \ref{fig:ex1_moments} gives the comparison between the
numerical solution and the exact solution for these two coefficients.
In both plots, all the three lines coincide perfectly.

\begin{figure}[!ht]
\centering
\subfloat[$f_{400}(t)$]{%
  \includegraphics[width=.4\textwidth]{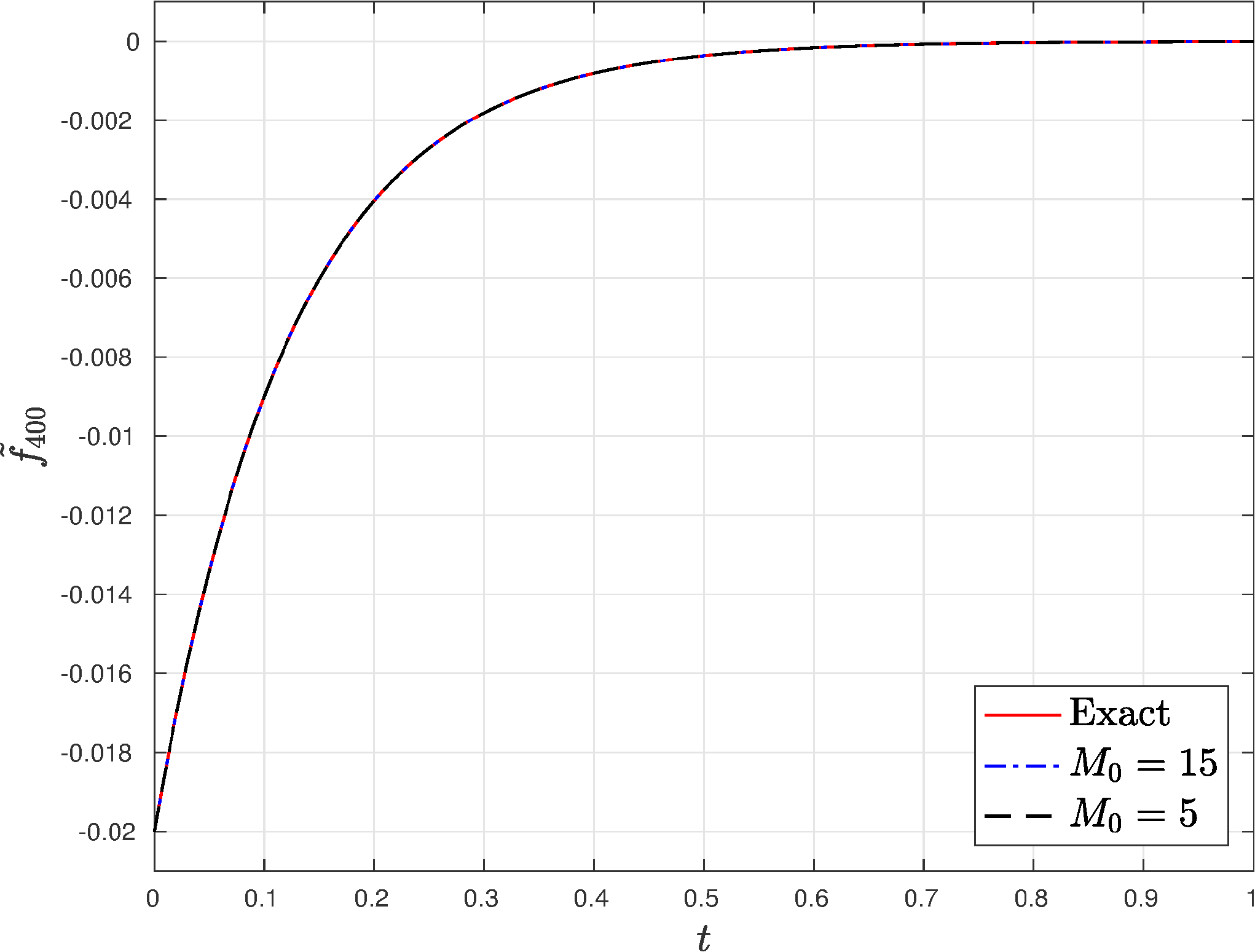}
} \hspace{20pt}
\subfloat[$f_{220}(t)$]{%
  \includegraphics[width=.4\textwidth]{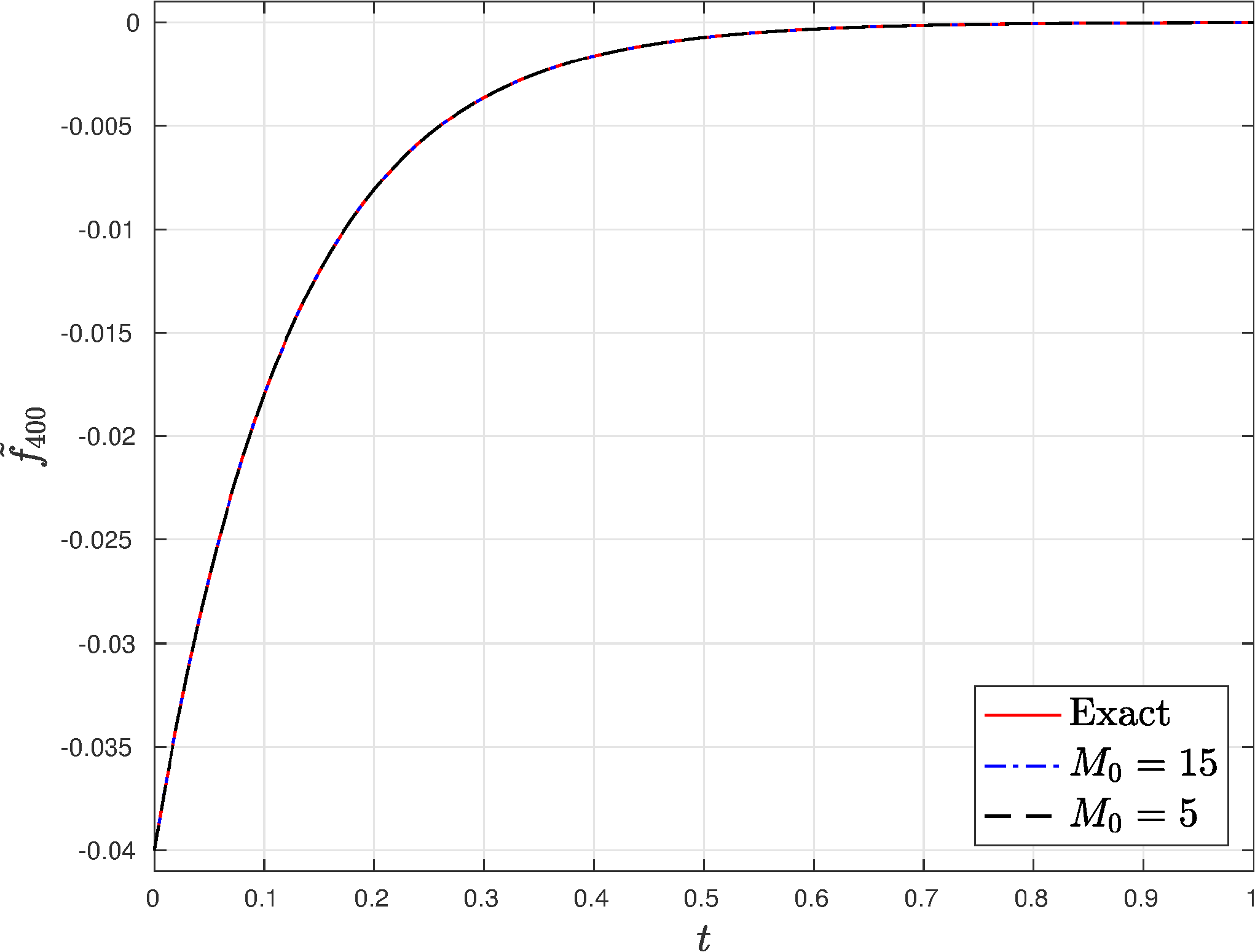}
}
\caption{The evolution of the coefficients. The red lines correspond
  to the reference solution, and the blue dot dashed and black dashed
  lines correspond to the numerical solutions of $M_0=15$ and $M_0=5$
  respectively.}
\label{fig:ex1_moments}
\end{figure}

\subsection{Bi-Gaussian initial data} 
\label{sec:ex2}
In this example, we perform the numerical test for the Bi-Gaussian
problem. Here the Coulombian case $\gamma = -3$ are tested. The initial
distribution function is
\begin{displaymath}
f(0,\bv) = \frac{1}{2\pi^{3/2}} \left[
  \exp \Big( -(v_1 + \sqrt{3/2})^2 + v_2^2 + v_3^2 \Big) +
  \exp \Big( -(v_1 - \sqrt{3/2})^2 + v_2^2 + v_3^2 \Big)
\right].
\end{displaymath}
In this numerical test, we use $M = 20$ which gives a good
approximation of the initial distribution function (see Figure
\ref{fig:ex2_init}).
\begin{figure}[!ht]
\centering
\subfloat[Initial MDF $g(0,v_1)$\label{fig:ex2_init_1d}]{%
  \includegraphics[height=.24\textwidth]{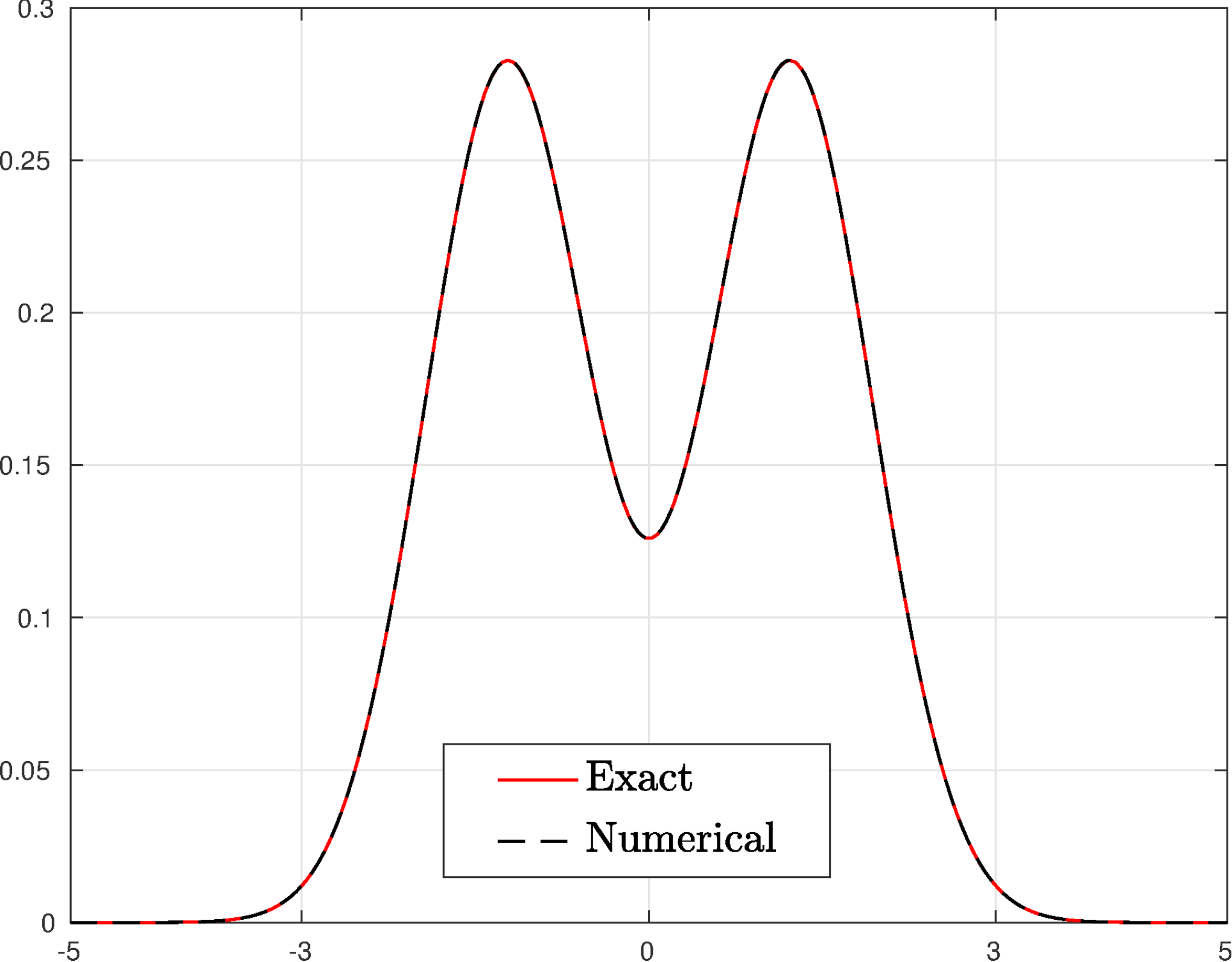}
}\hfill
\subfloat[Contours of $h(0,v_1,v_2)$\label{fig:ex2_init_2d_contour}]{%
  \includegraphics[height=.24\textwidth]{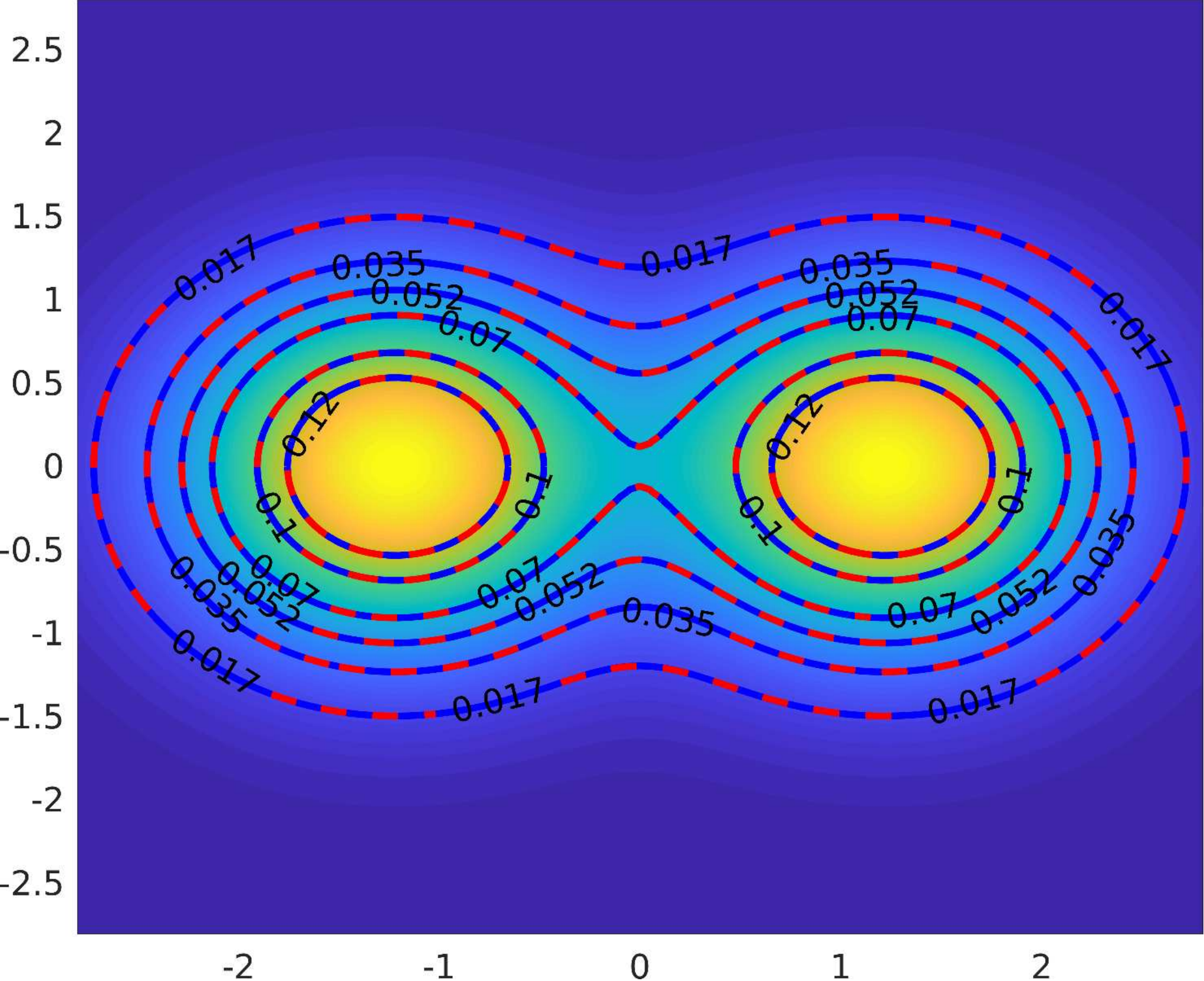}
}\hfill
\subfloat[Initial MDF $h(0,v_1,v_2)$\label{fig:ex2_init_2d}]{%
  \includegraphics[height=.25\textwidth]{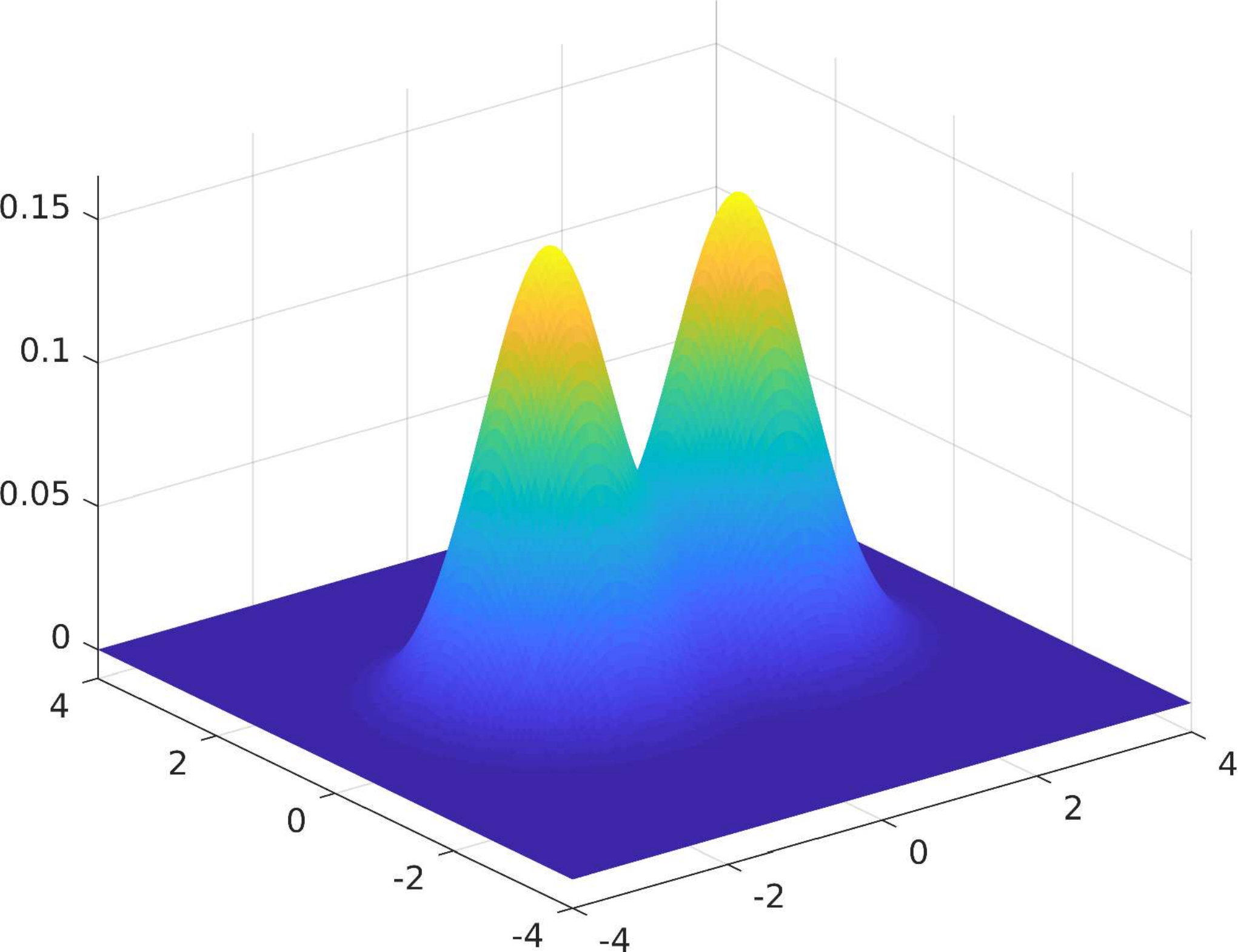}
}
\caption{Figure (a) is the initial marginal distribution functions
  $g(0, v_1)$.  The red solid line corresponds to the exact solution,
  and the blue dashed line corresponds to the numerical approximation.
  Figure (b) is the initial marginal distribution functions
  $h(0, v_1, v_2)$. The blue solid lines correspond to the exact
  solution, and the red dashed lines correspond to the numerical
  approximation. Figure (c) shows only the numerical approximation.}
\label{fig:ex2_init}
\end{figure}

For this example, we consider the three cases $M_0 = 5, 10, 15$, and
the corresponding one-dimensional marginal distribution functions at
$t = 0.4$, $1$ and $3$ are given in Figure \ref{fig:ex2_1d}. In all
the results, the numerical results are converging to the reference
solution as $M_0 = 15$, and the lines for $M_0 = 10$ and $M_0 = 15$
are very close to each other. To get a clearer picture, similar
comparisons of two-dimensional results are also provided in Figure
\ref{fig:ex2_2d}.

\begin{figure}[!ht]
\centering
\subfloat[$t=0.4$]{%
  \includegraphics[width=.3\textwidth]{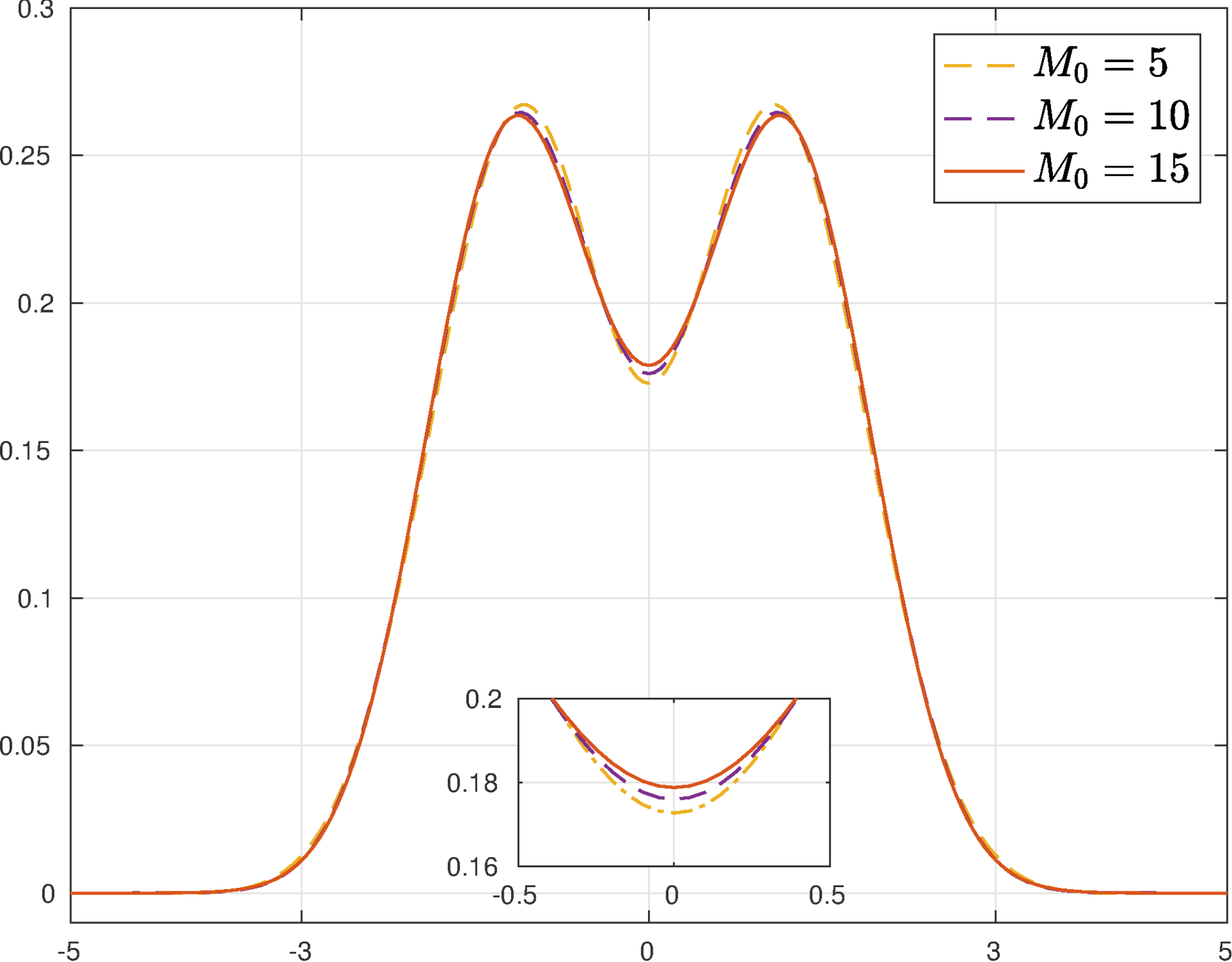}
} \hfill
\subfloat[$t=1$]{%
  \includegraphics[width=.3\textwidth]{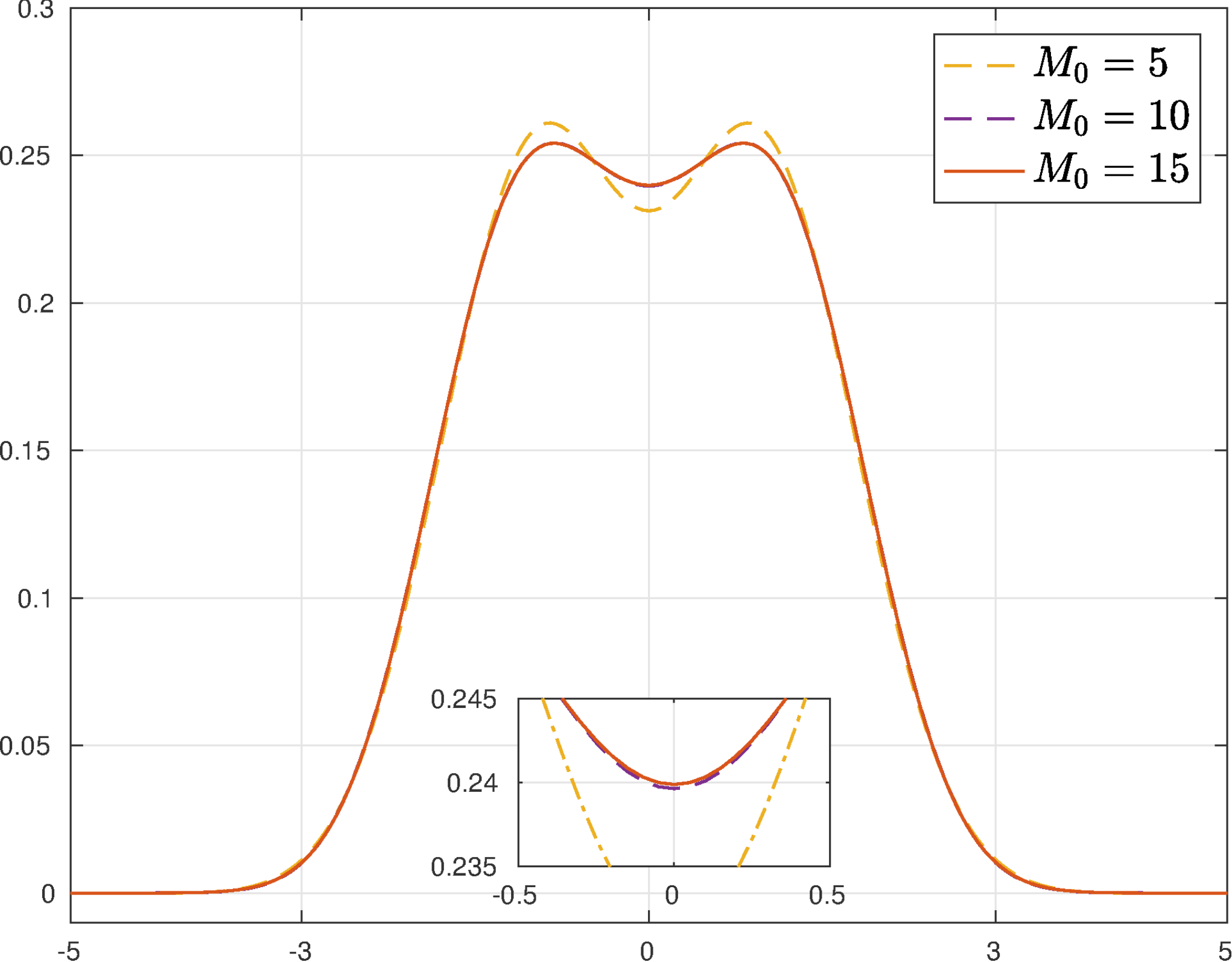}
} \hfill
\subfloat[$t=3$]{%
  \includegraphics[width=.3\textwidth]{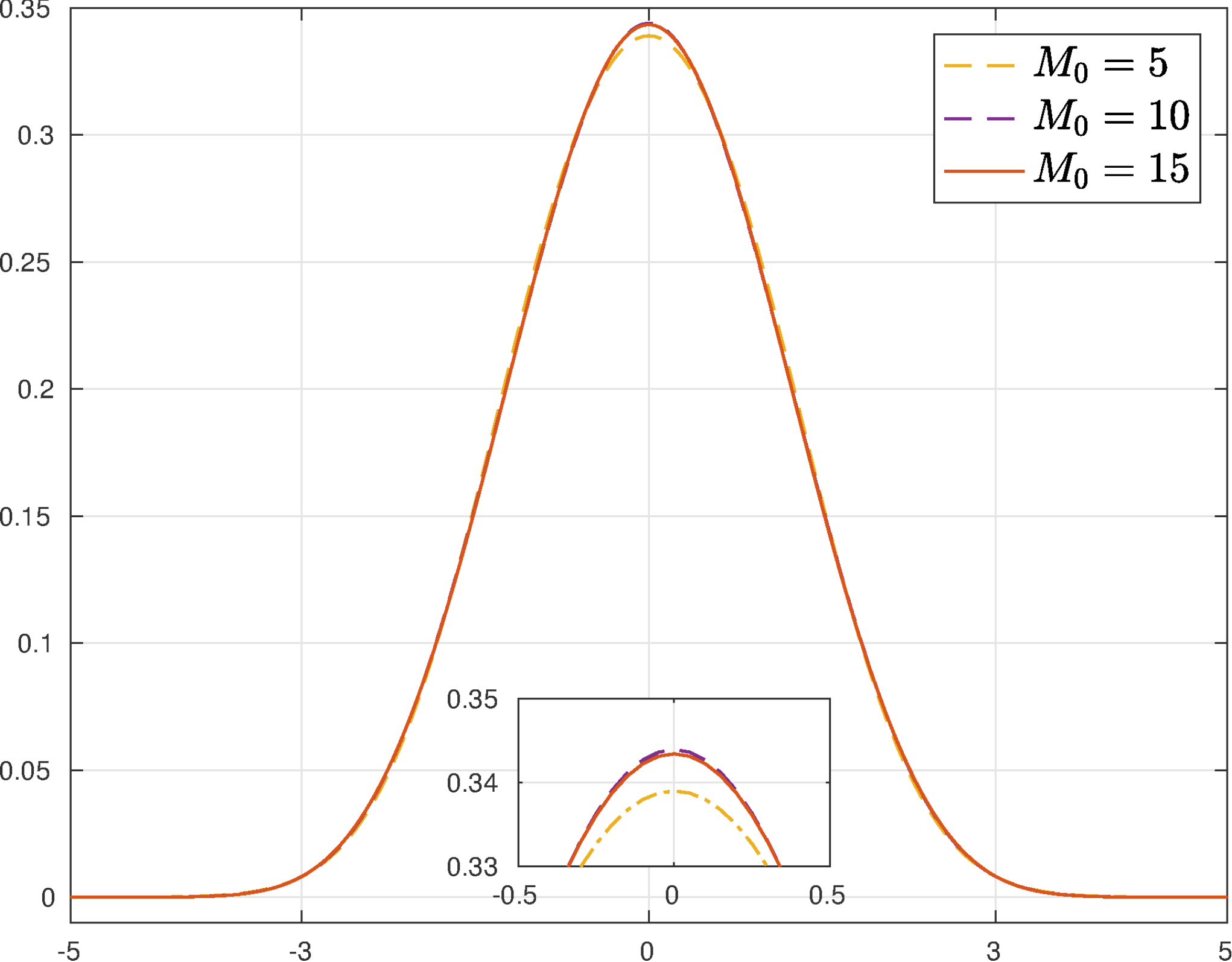}
}
\caption{The Coulombian case $\gamma = -3$.  Marginal distribution
  functions at different times.}
\label{fig:ex2_1d}
\end{figure}

\begin{figure}[!ht]
\centering
\subfloat[$t=0.4, M_0 = 5$]{%
  \includegraphics[width=.33\textwidth]{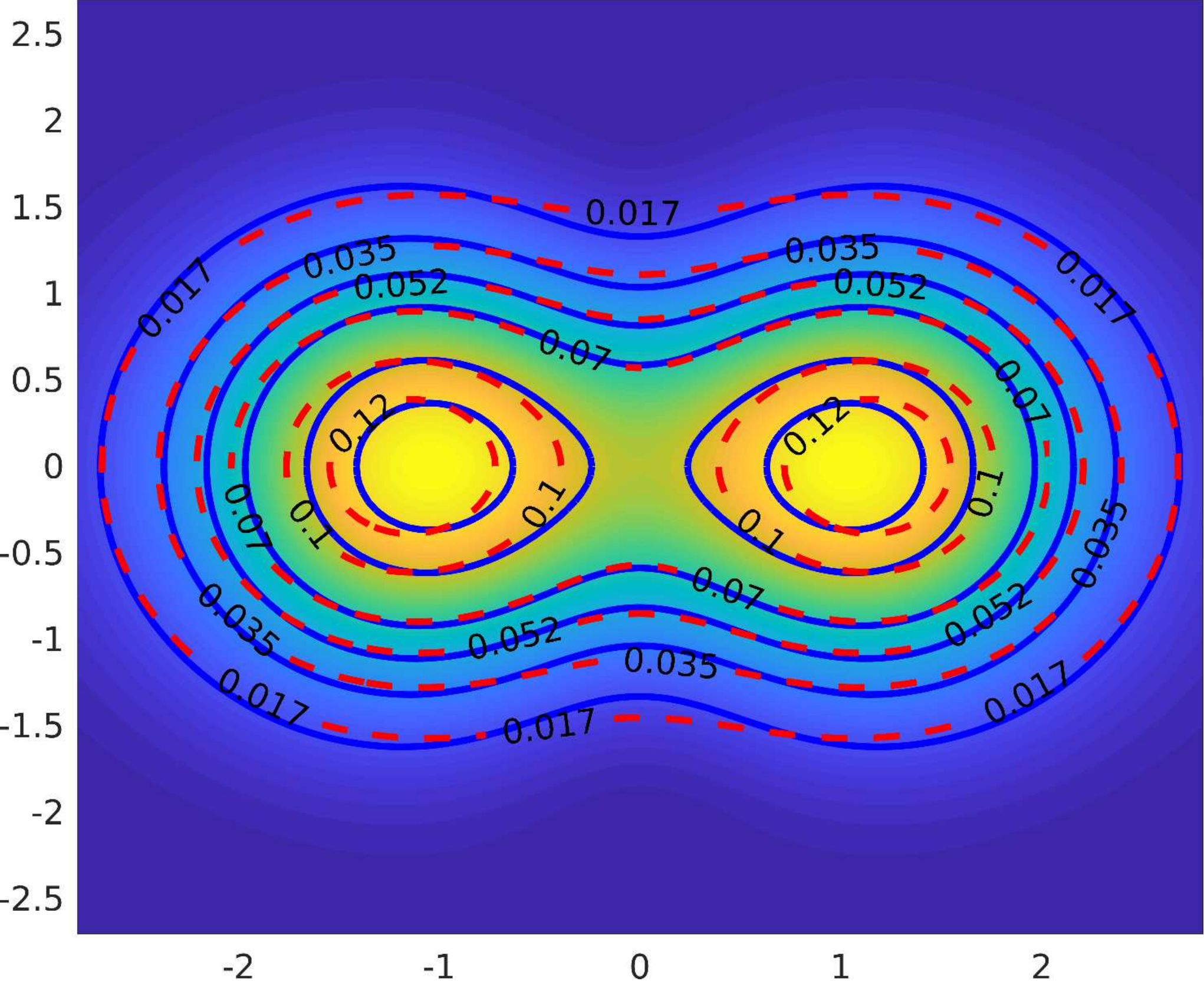}
} 
\subfloat[$t=1, M_0 = 5$]{%
  \includegraphics[width=.33\textwidth]{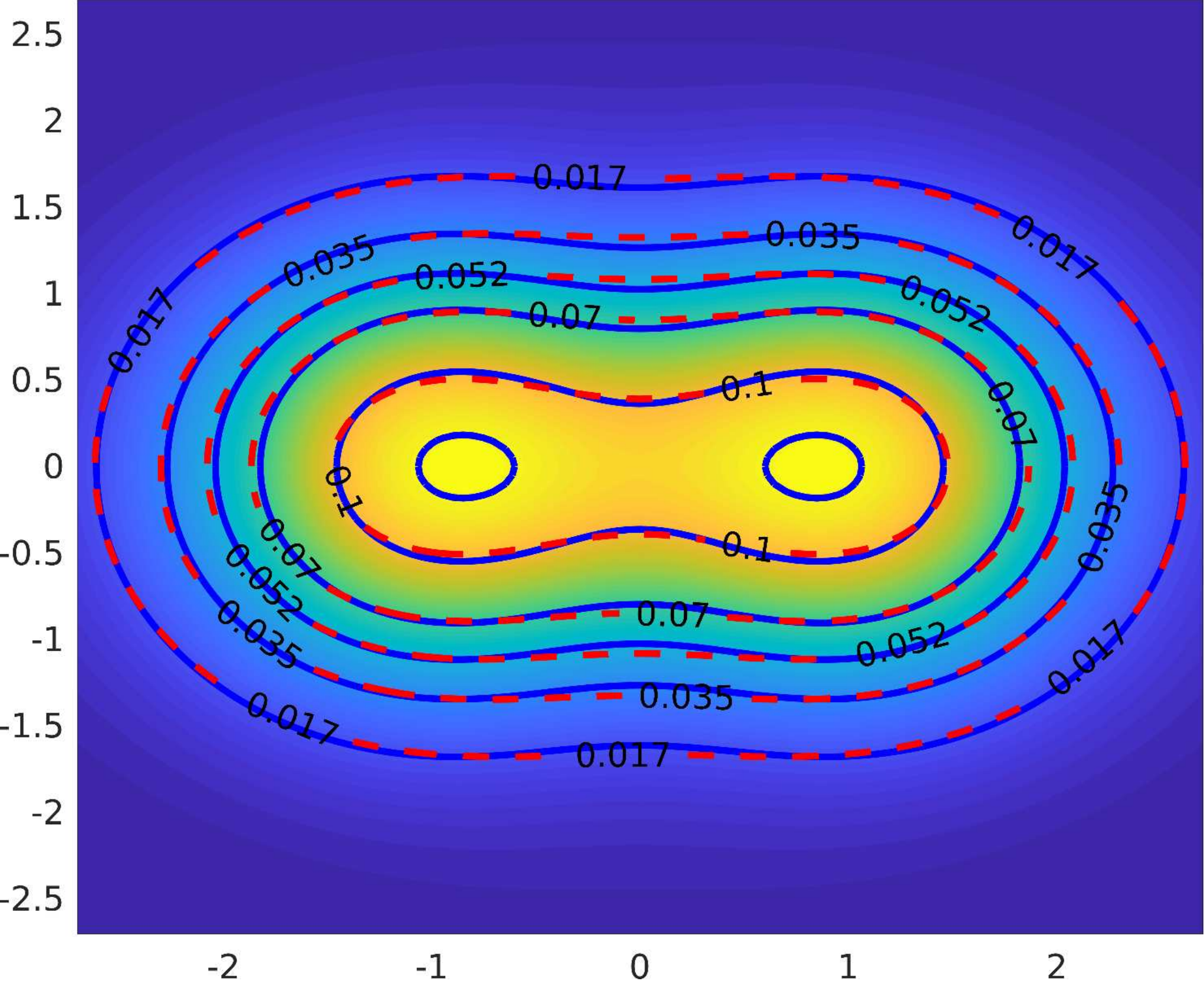}
} 
\subfloat[$t=3, M_0 = 5$]{%
  \includegraphics[width=.33\textwidth]{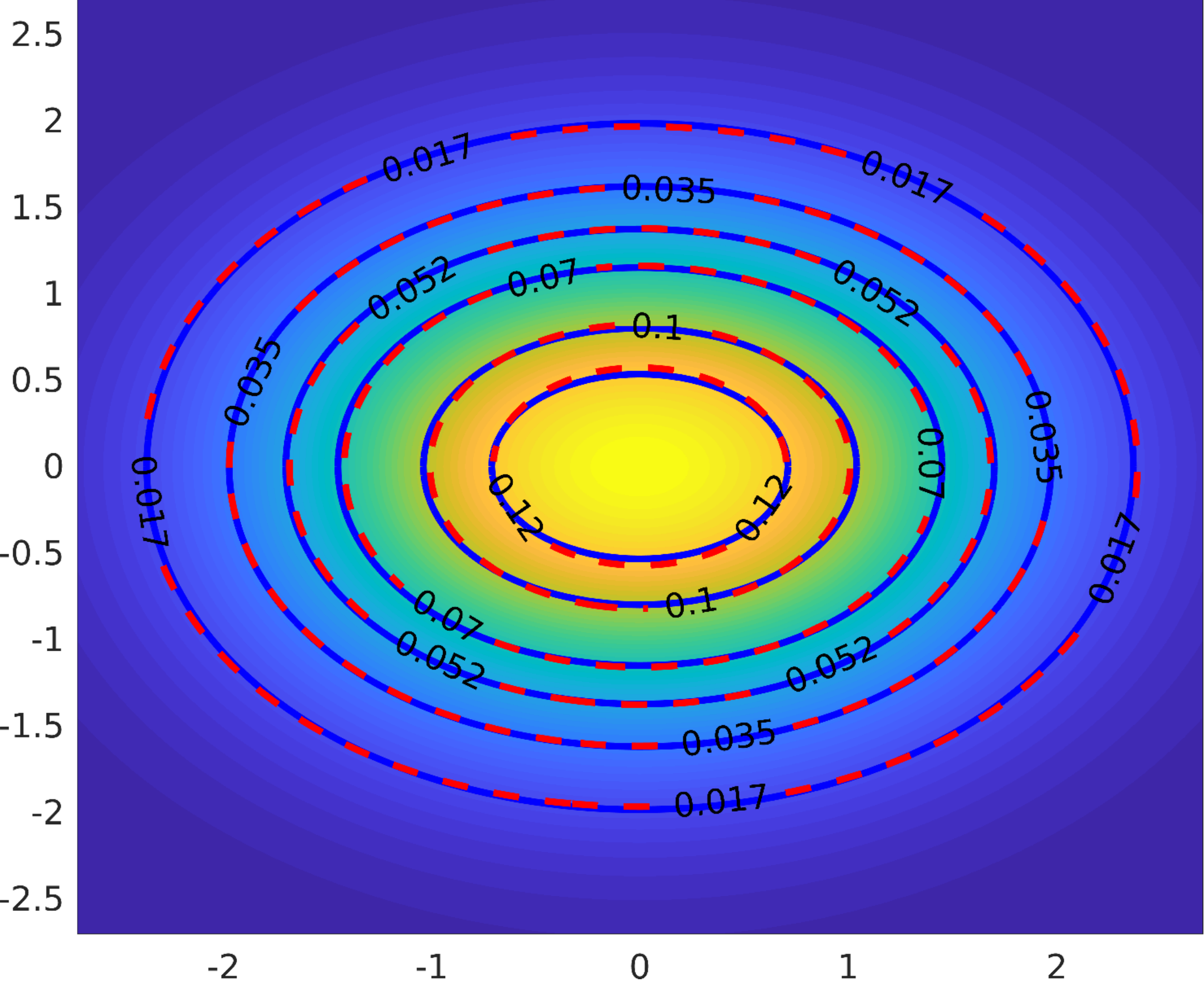}
}\\
\subfloat[$t=0.4, M_0 = 10$]{%
  \includegraphics[width=.33\textwidth]{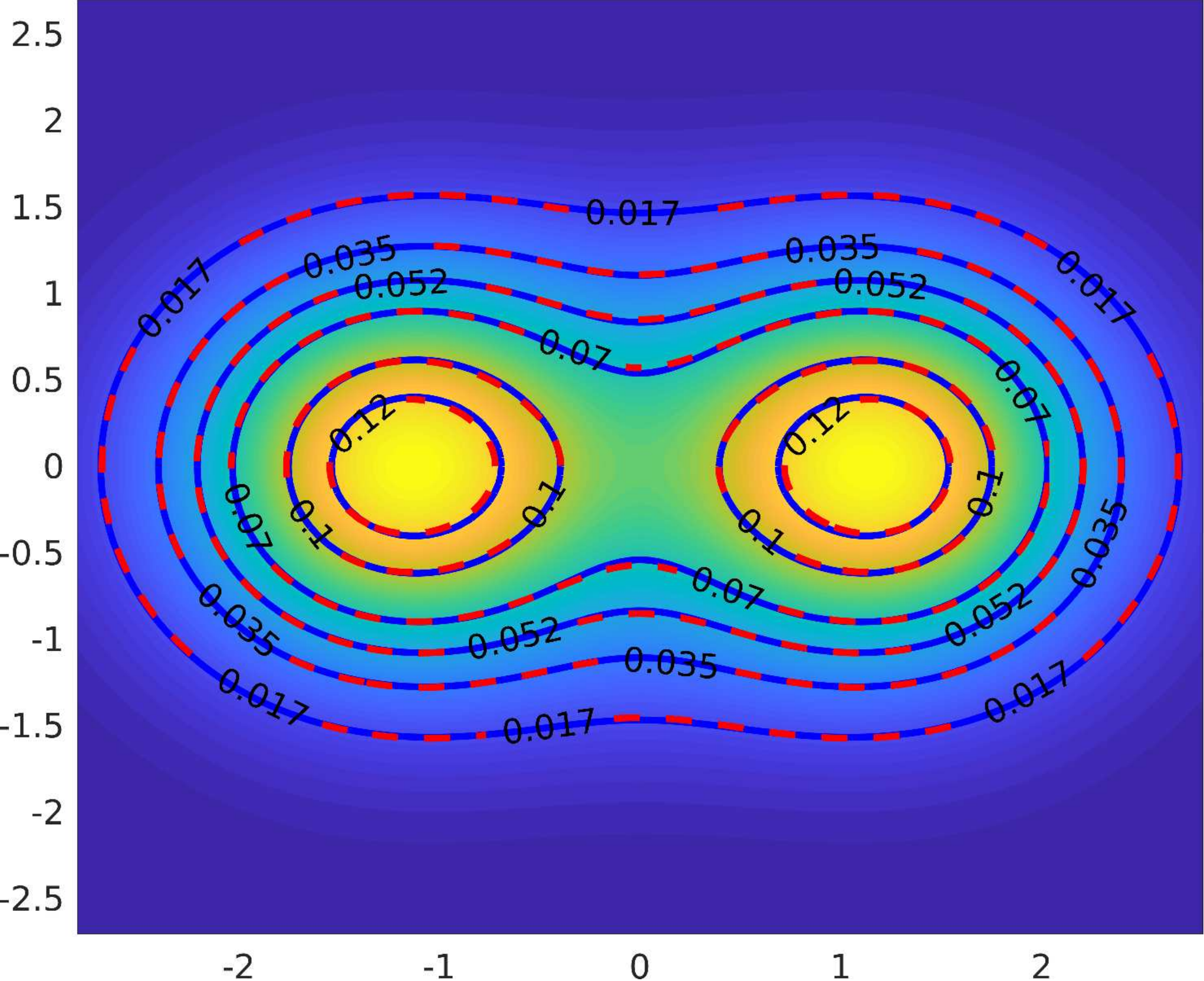}
} 
\subfloat[$t=1, M_0 =10$]{%
  \includegraphics[width=.33\textwidth]{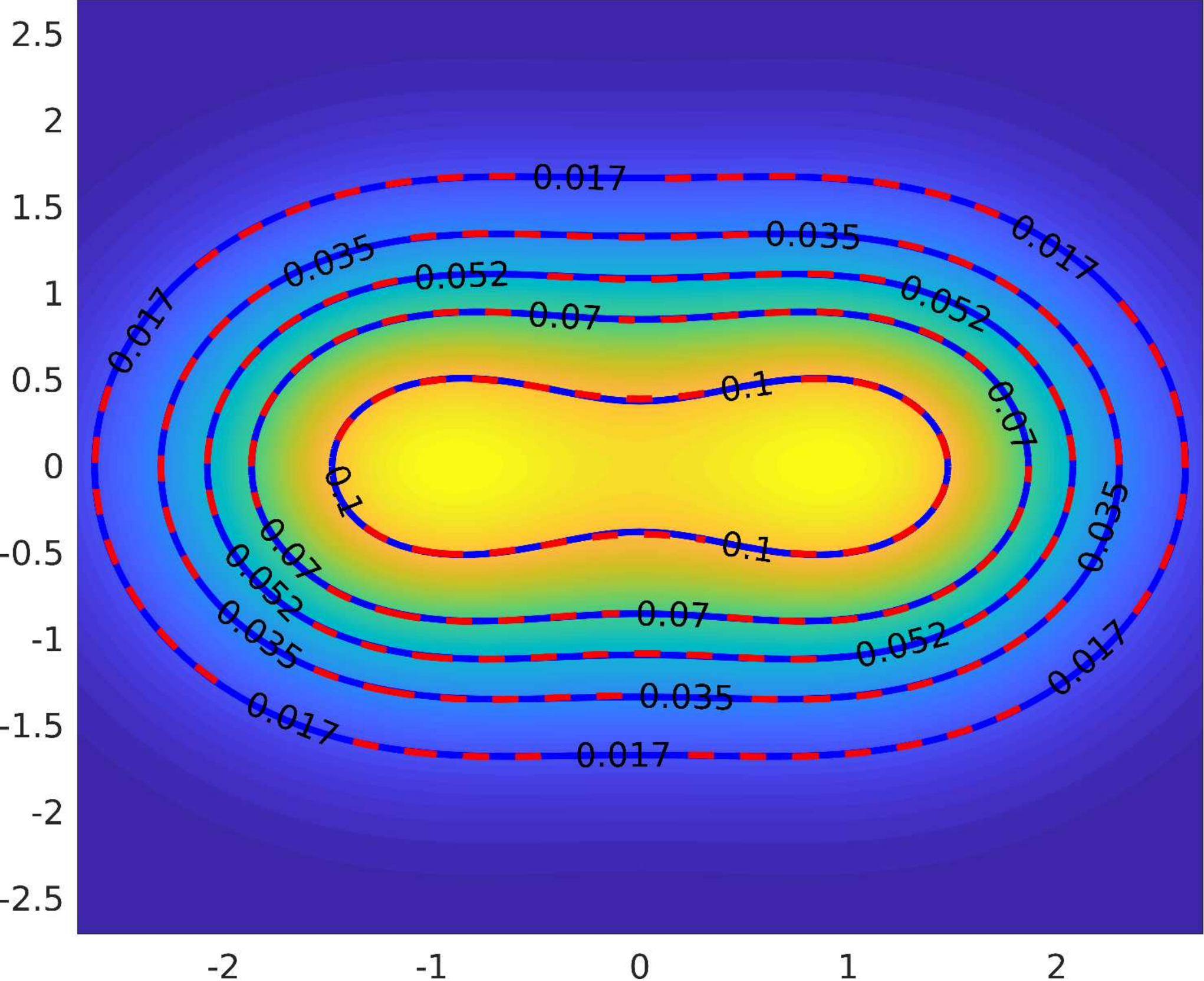}
} 
\subfloat[$t=3, M_0 = 10$]{%
  \includegraphics[width=.33\textwidth]{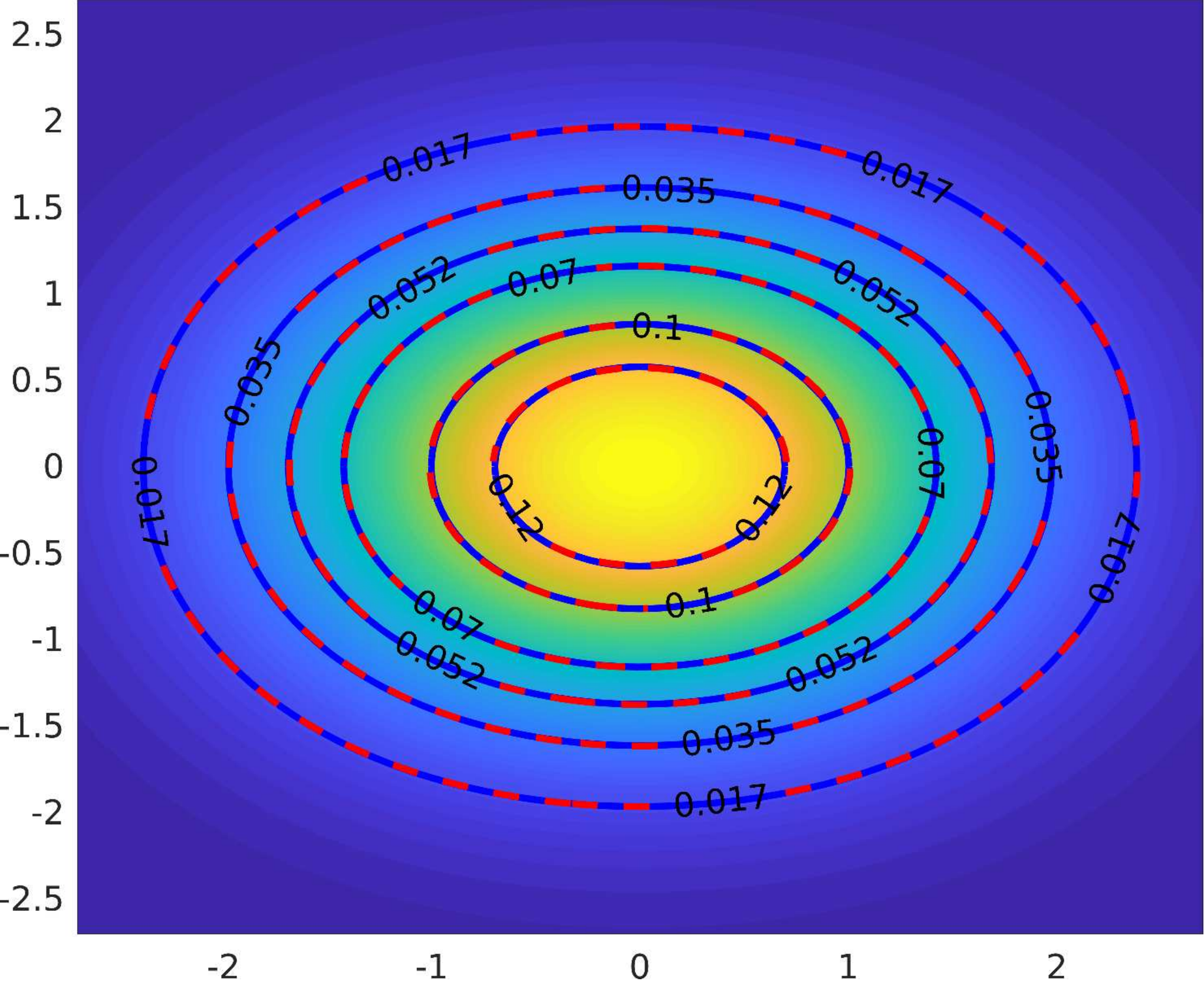}
}
\caption{The Coulombian case $\gamma = -3$.  Comparison of numerical
  results with the reference solution. The red dashed contours are the
  reference solution as $M_0 = 15$. The blue solid contours at
  different columns are respectively the results for $M_0 = 5$ and
  $M_0 = 10$.}
\label{fig:ex2_2d}
\end{figure}

Now we consider the evolution of stress tensor and heat flux. In this
example, we always have $\sigma_{11} = -2\sigma_{22} = -2\sigma_{33}$
and $q_i = 0, i =1, 2,3$.  Therefore we focus only on the evolution of
$\sigma_{11}$, which is plotted in Figure \ref{fig:ex2_sigma11}.  It
can be seen that even for $M_0 = 5$, the evolution of the stress
tensor is almost exact, where the distribution function is not
approximated very well. The three lines are all on top of each other.

\begin{figure}[!ht]
\centering
\includegraphics[width=.4\textwidth]{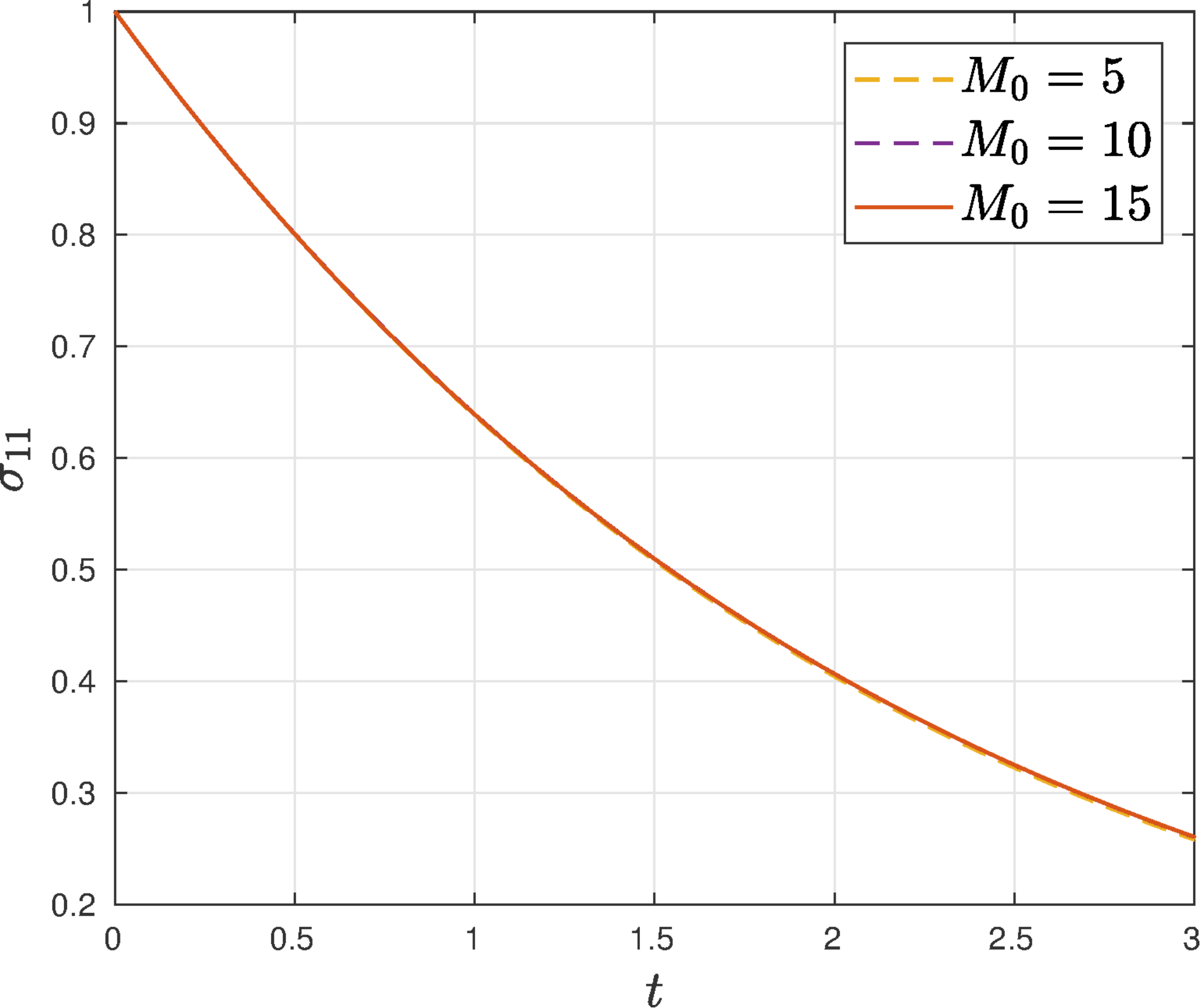}
\caption{The Coulombian case $\gamma = -3$. Evolution of
  $\sigma_{11}(t)$. Three lines are on top of each other.}
\label{fig:ex2_sigma11}
\end{figure}

In order to test the computational capacity of our new model, the same
example with a very small $\gamma$ as $\gamma = -4.9$ is tested. Here
we also set $M=20$, and choose the numerical results with $M_0=15$ as
the reference solution. Figure \ref{fig:ex2_4p9_2d} shows the marginal
distribution $h(t, v_1, v_2)$ at $t = 0.4, 1$ and $3$ with $M_0= 5$
and $10$. It illustrates that when $\gamma$ equals $-4.9$, the
numerical solutions are converging to the reference solution as
$M_0=15$, and that the solution with $M_0=10$ is almost the same as
the reference solution. The time evolution of $\sigma_{11}$ is plotted
in Figure \ref{fig:ex2_4p9_sigma11}, where the three results are also
on top of each other, even with $M_0=5$. Moreover, from Figure
\ref{fig:ex2_4p9_2d} and \ref{fig:ex2_4p9_sigma11}, we can find that
the time evolution of the distribution function with $\gamma = -4.9$
is slower than that with $\gamma=-3$, which is also consistent with
the form of the FPL collision operator.

\begin{figure}[!ht]
\centering
\subfloat[$t=0.4, M_0 = 5$]{%
  \includegraphics[width=.33\textwidth]{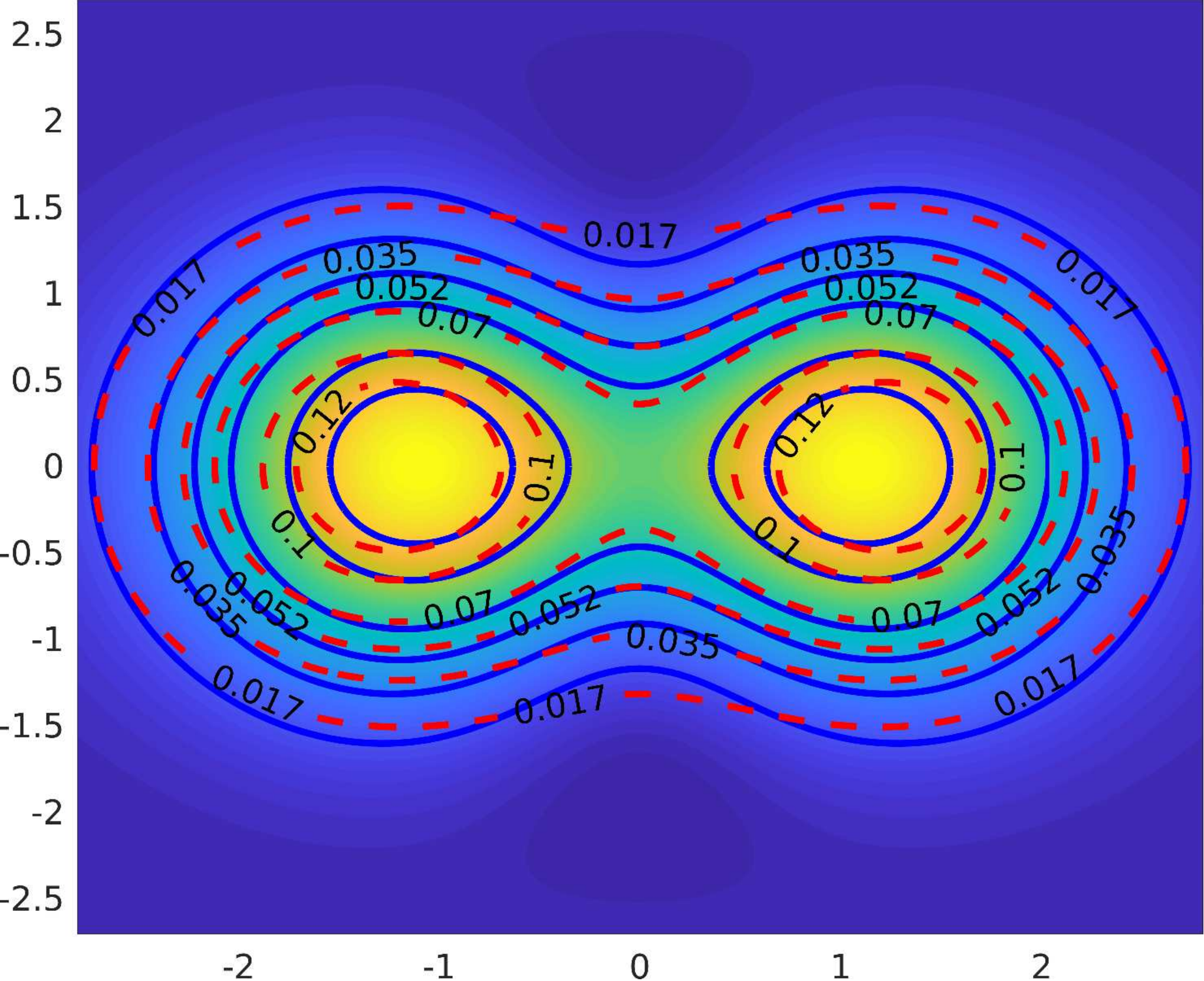}
} 
\subfloat[$t=1, M_0 = 5$]{%
  \includegraphics[width=.33\textwidth]{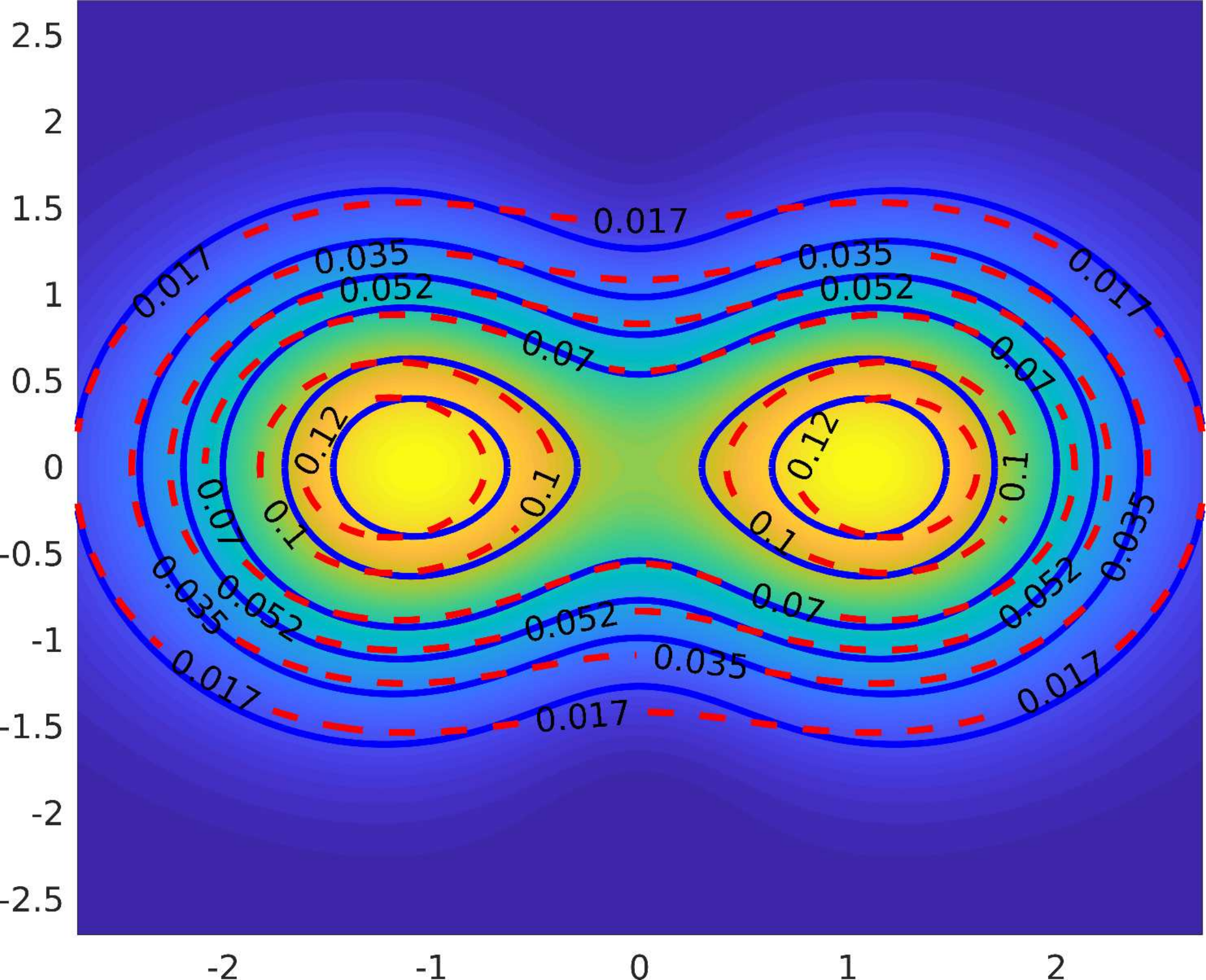}
} 
\subfloat[$t=3, M_0 = 5$]{%
  \includegraphics[width=.33\textwidth]{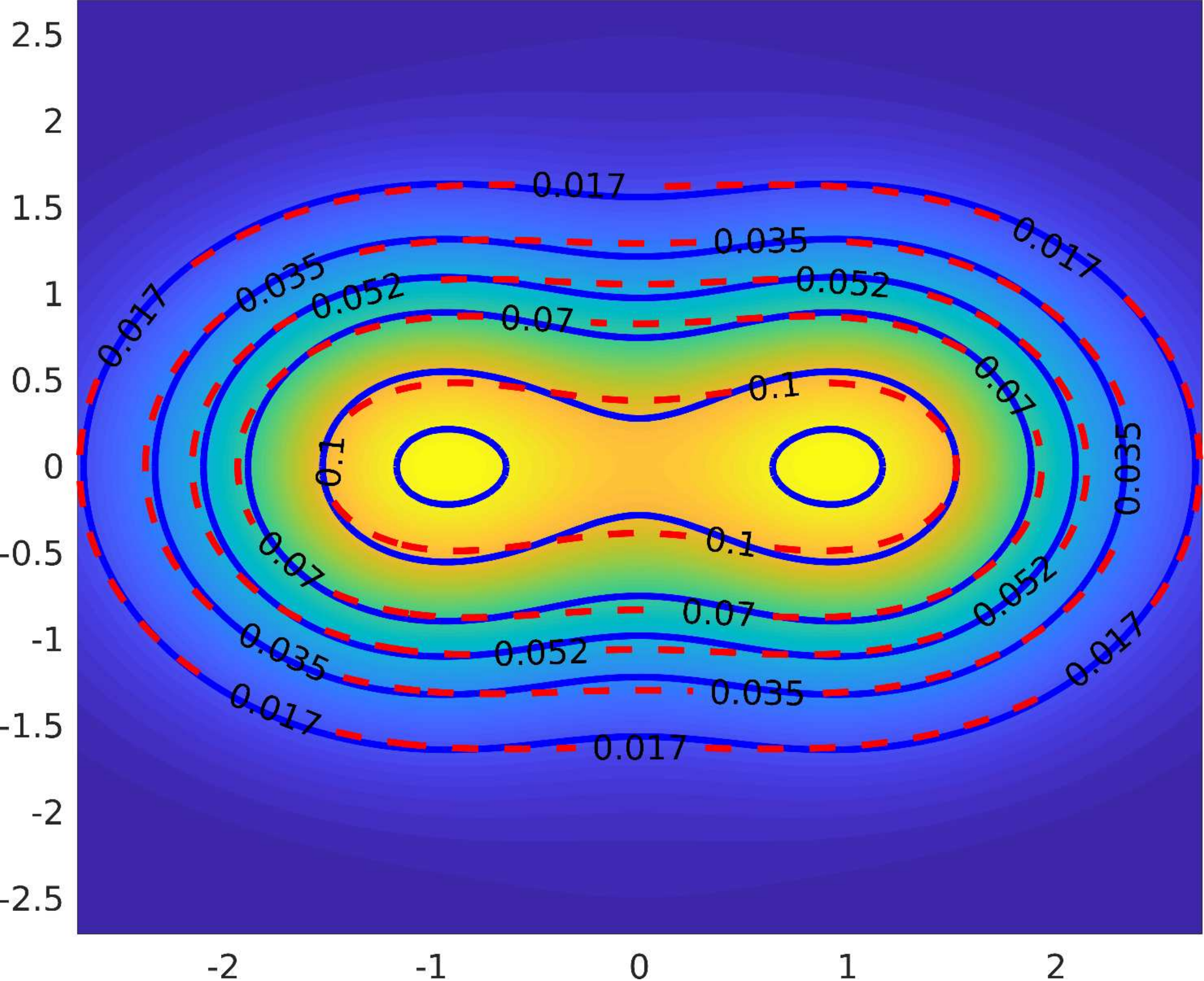}
} \\
\subfloat[$t=0.4, M_0=10$]{%
  \includegraphics[width=.33\textwidth]{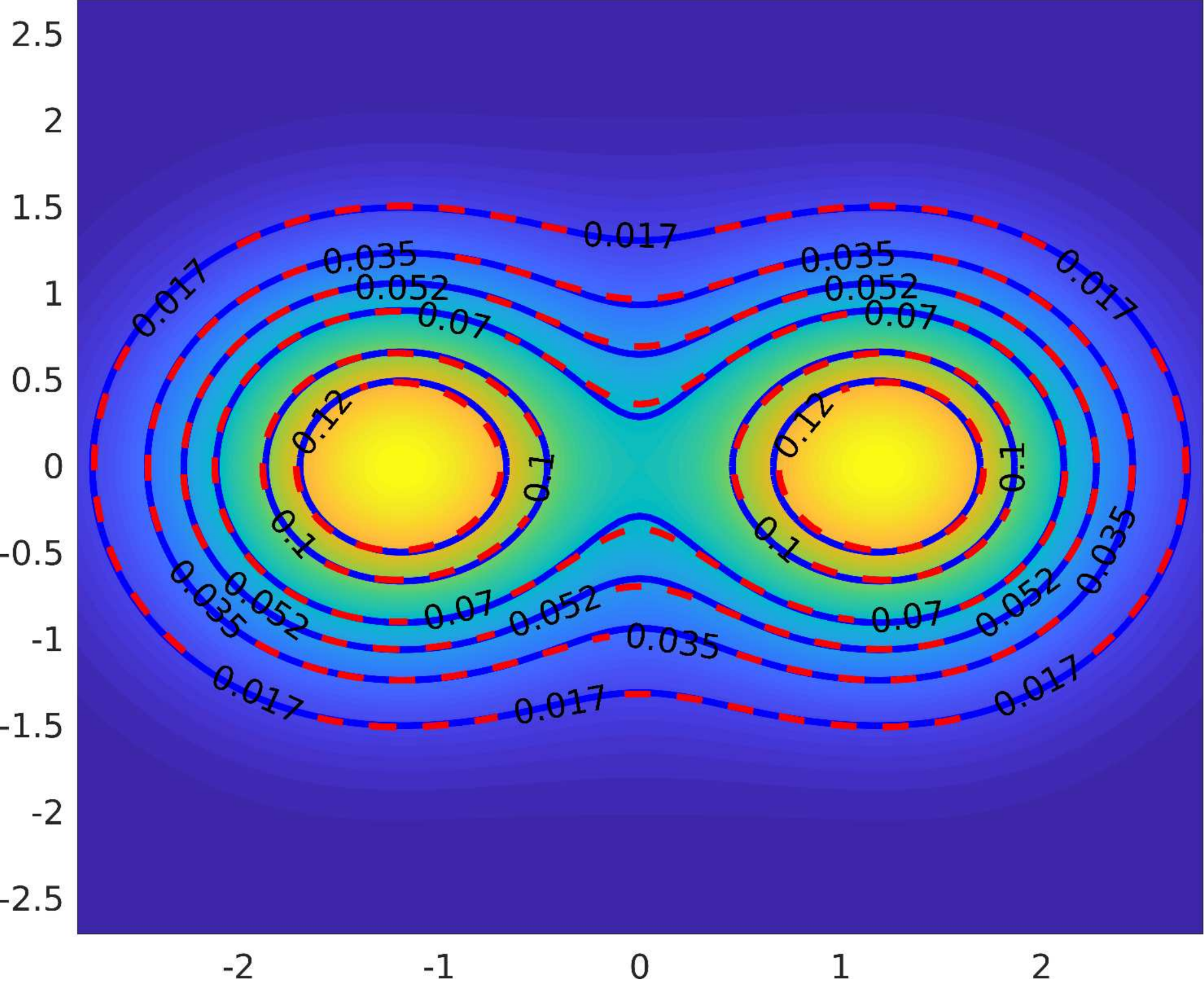}
} 
\subfloat[$t=1, M_0=10$]{%
  \includegraphics[width=.33\textwidth]{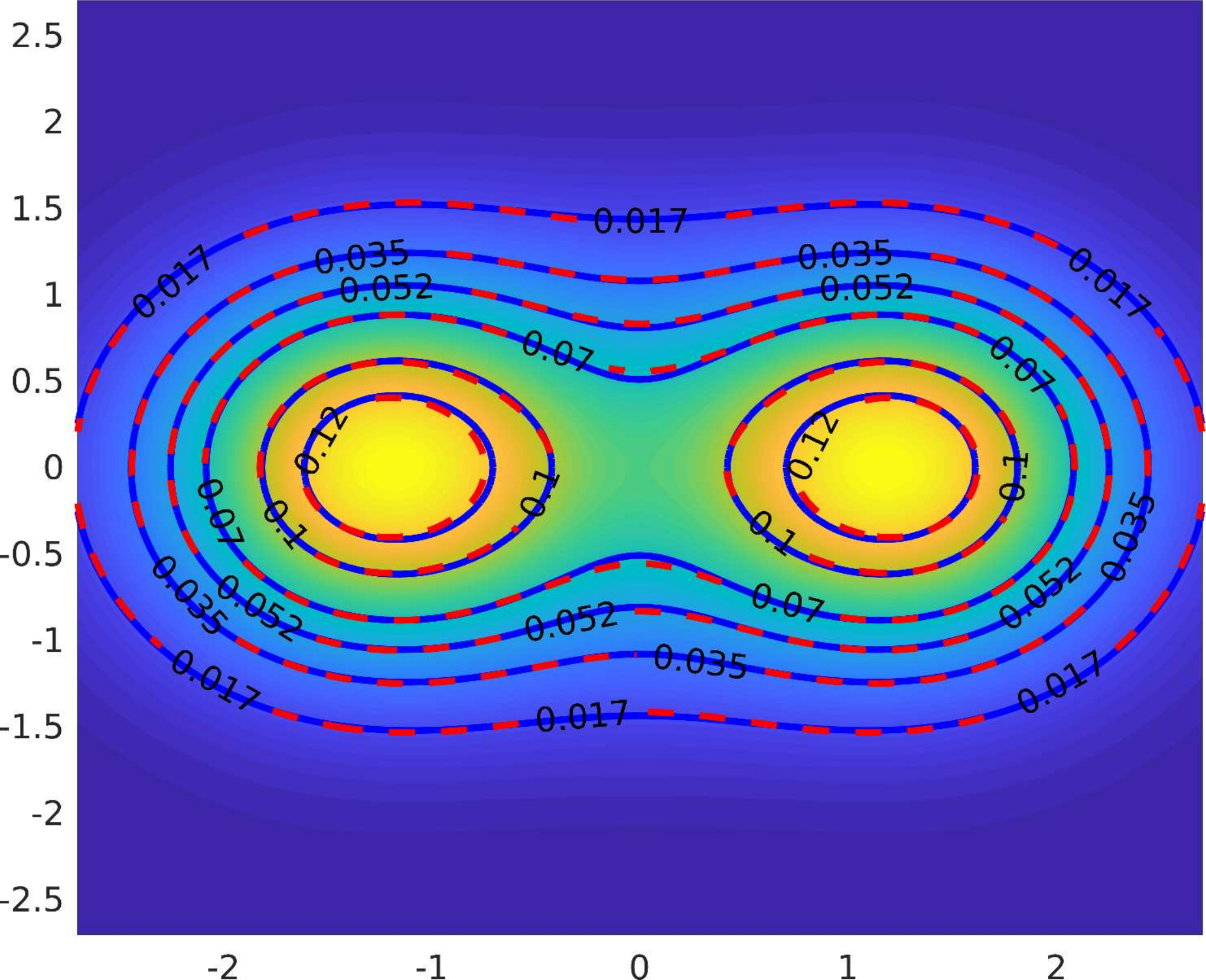}
} 
\subfloat[$t=3, M_0 =10$]{%
  \includegraphics[width=.33\textwidth]{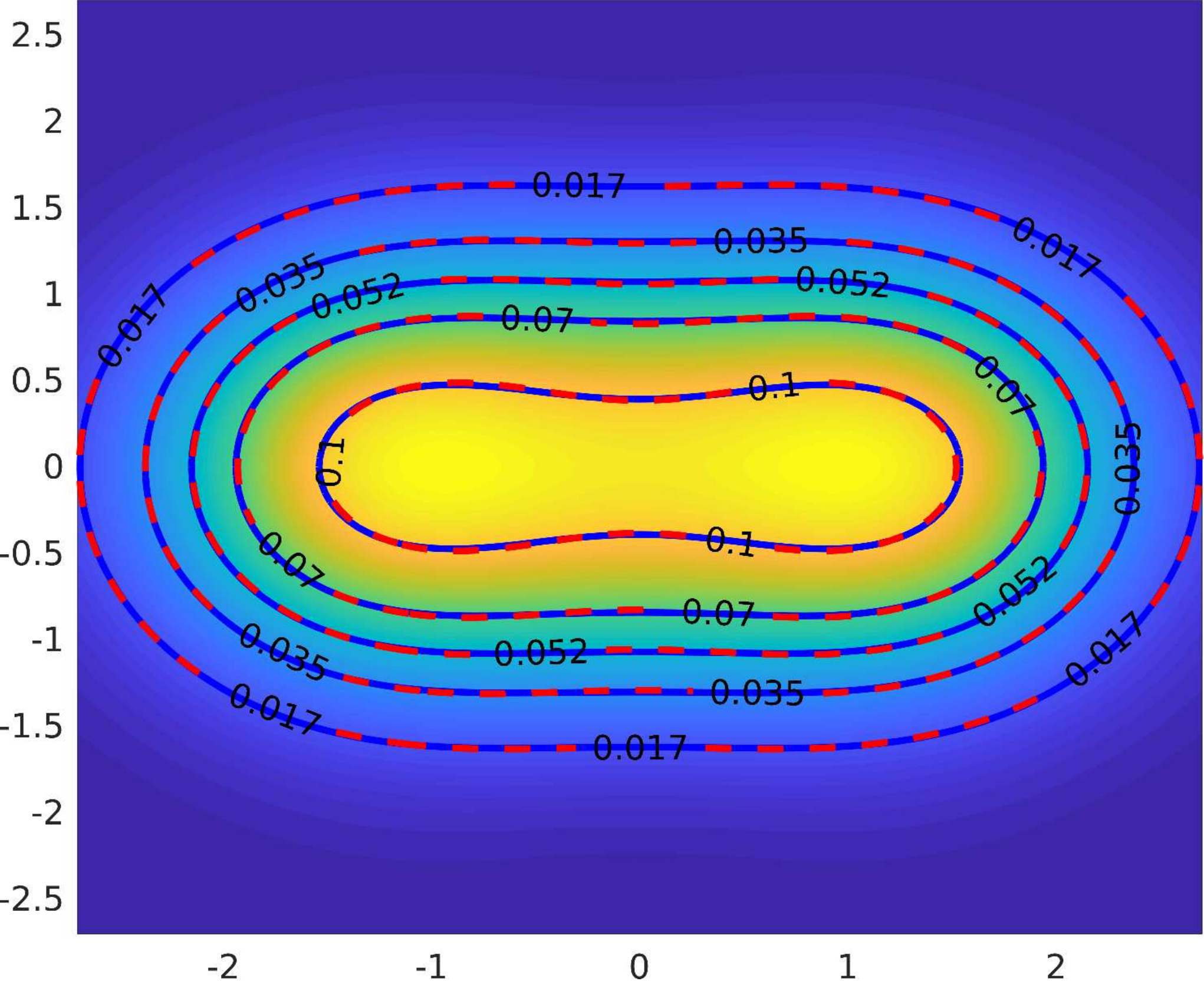}
}
\caption{The case $\gamma = -4.9$. Comparison of numerical solutions
  and the reference solution. The red dashed contours are the
  reference solutions $M_0 = 15$. The blue solid contours at different
  columns are respectively the numerical solutions $M_0 = 5$ and
  $M_0=10$.}
\label{fig:ex2_4p9_2d}
\end{figure}

\begin{figure}[!ht]
\centering
\includegraphics[width=.4\textwidth]{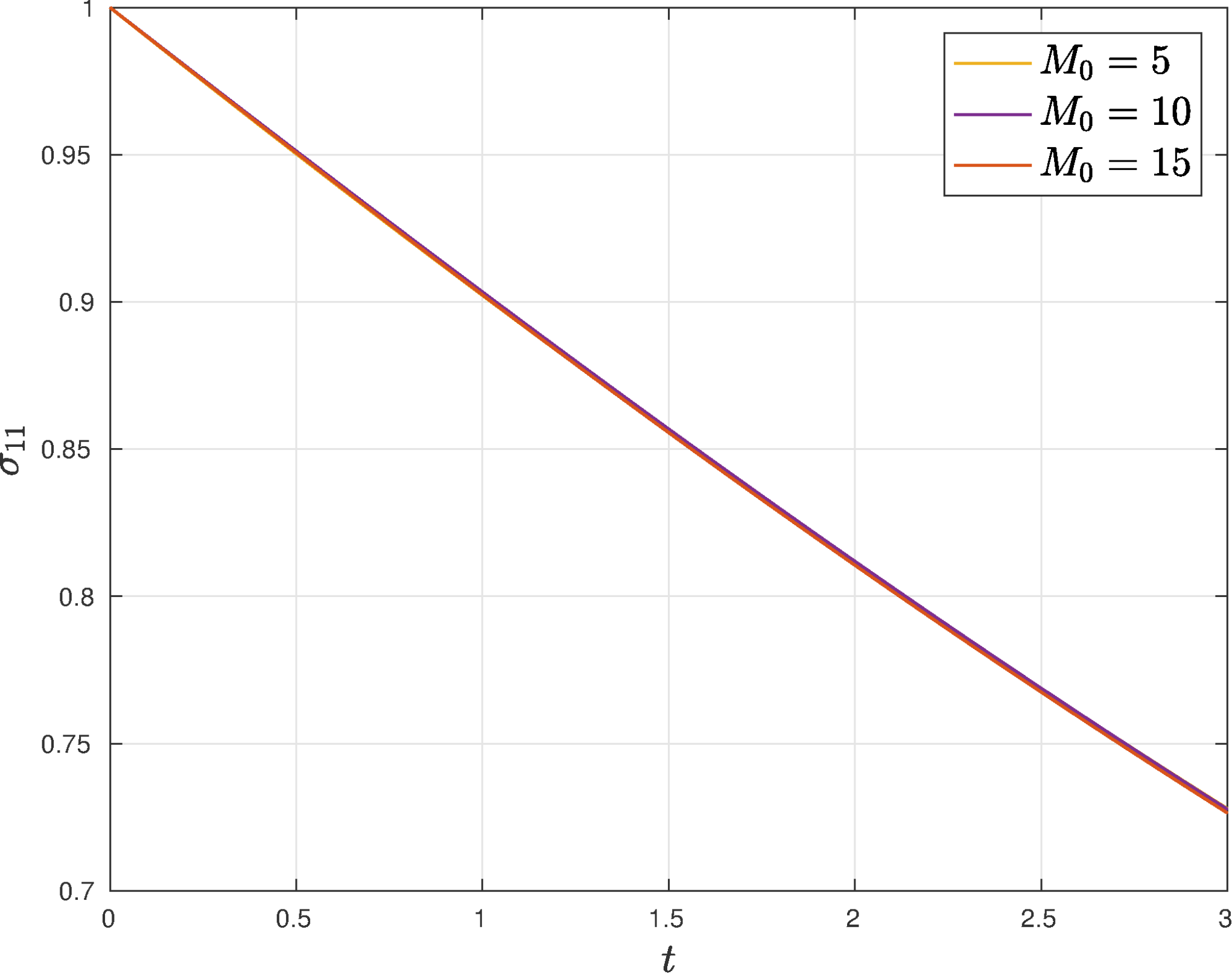}
\caption{The case $\gamma = -4.9$. Evolution of
  $\sigma_{11}(t)$. Three lines are on top of each other.}
\label{fig:ex2_4p9_sigma11}
\end{figure}

\subsection{Rosenbluth problem} 
In this example, the Rosenbluth problem is tested.  Also the Coulombian
case $\gamma=-3$ and the case $\gamma = -4.9$ are tested.  The initial
condition is from \cite{PARESCHI2000} as
\begin{equation}
  \label{eq:ex3_ini}
  f(0, \bv) = A \exp(-(B|\bv| -1)^2).
\end{equation}
The parameter $A, B$ are standardized to satisfy the condition that
the initial density and temperature all equal $1$, precisely
\begin{equation}
  \label{eq:ex3_coe}
  A = \frac{(b/3)^{3/2}}{a^{5/2}}, \qquad  B =
  \frac{(b/3)^{1/2}}{a^{1/2}}, 
\end{equation}
where $a = \pi (3\sqrt{\pi}({\rm erf}(1) + 1 + 2/e)$ and
$b = \pi(9.5\sqrt{\pi}({\rm erf}(1)+ 1) + 7/e)$. Here $e$ is the Euler
number and $\rm erf (x) = \frac{1}{\sqrt{\pi}}\int_{0}^{x}e^{-t^2}\dd
t$ is the error function.  

In order to approximate the initial distribution function well, in
this numerical test, we set $M=40$. The initial MDFs are plotted in
Figure \ref{fig:ex3_init}, which illustrates the perfect numerical
approximation to the exact distribution function.

For this example, also the three cases $M_0 = 5, 10, 15$ are
tested. The numerical solution with $M_0=15$ is treated as the
reference solution.  The corresponding one and two dimensional
marginal distribution functions for the Coulombian $\gamma = -3$ are
shown in Figure \ref{fig:ex3_1d} and \ref{fig:ex3_2d}, where the
numerical solutions are converging to the reference solution and that
with $M_0 = 10$ is almost the same as the reference solution.

Moreover, our new model can also approximate the $\gamma = -4.9$ case
well. The two dimensional marginal distribution functions for
$\gamma = -4.9$ are presented in Figure \ref{fig:ex3_4p9}.  The
numerical solution with $M_0=15$ is also chosen as the reference
solution. Similar to the example in Sec \ref{sec:ex2}, the time
evolution of the distribution function with $\gamma = -4.9$ is also
slower than that with $\gamma = -3$. Further more, though there are
distinct differences between the numerical solution $M_0=5$ and
$M_0=15$, the numerical solutions are converging to the reference
solutions $M_0 = 15$.

\begin{figure}[!ht]
\centering
\subfloat[Initial MDF $g(0,v_1)$\label{fig:ex3_init_1d}]{%
  \includegraphics[height=.24\textwidth]{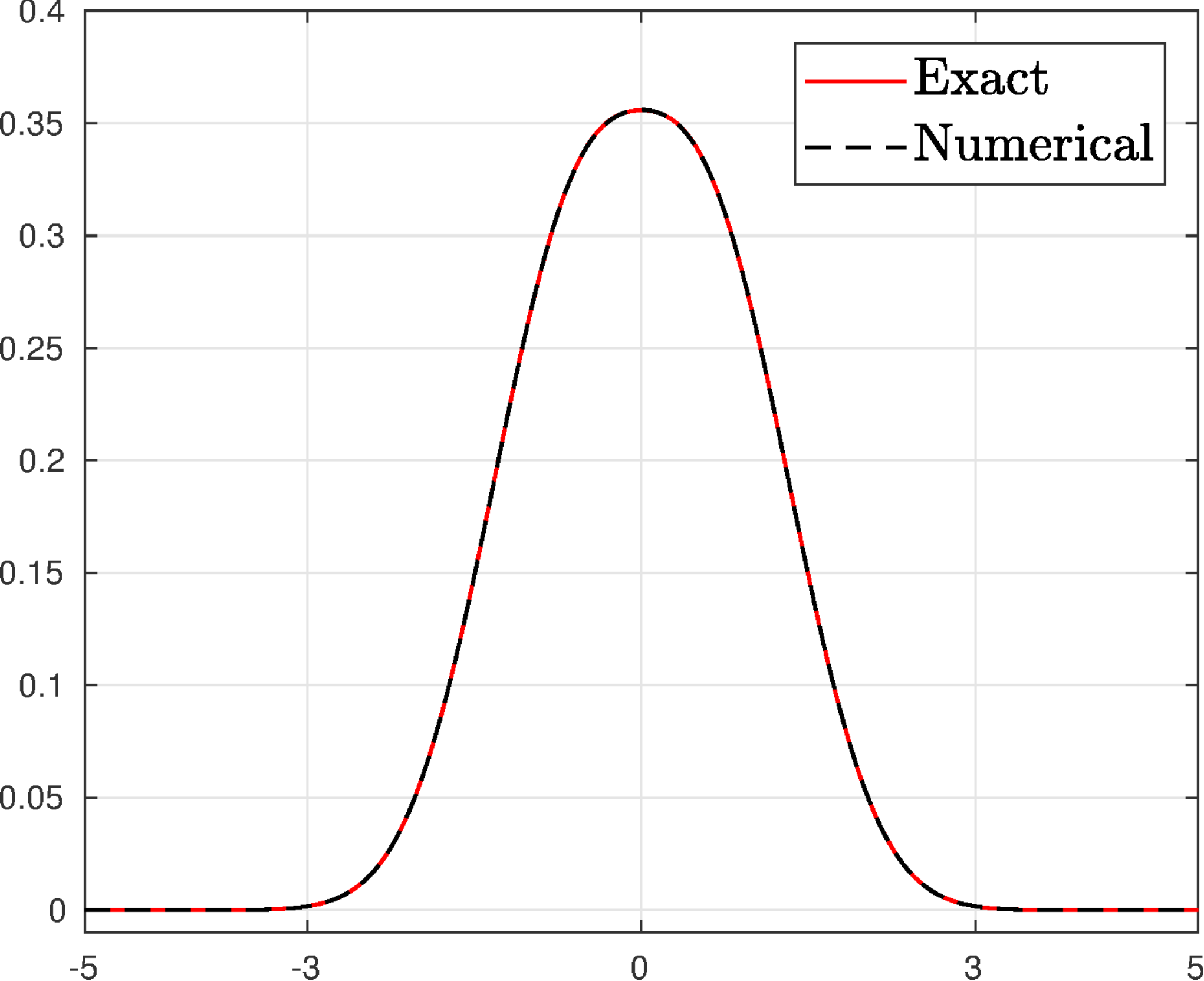}
}\hfill
\subfloat[Contours of $h(0,v_1,v_2)$\label{fig:ex3_init_2d_contour}]{%
  \includegraphics[height=.24\textwidth]{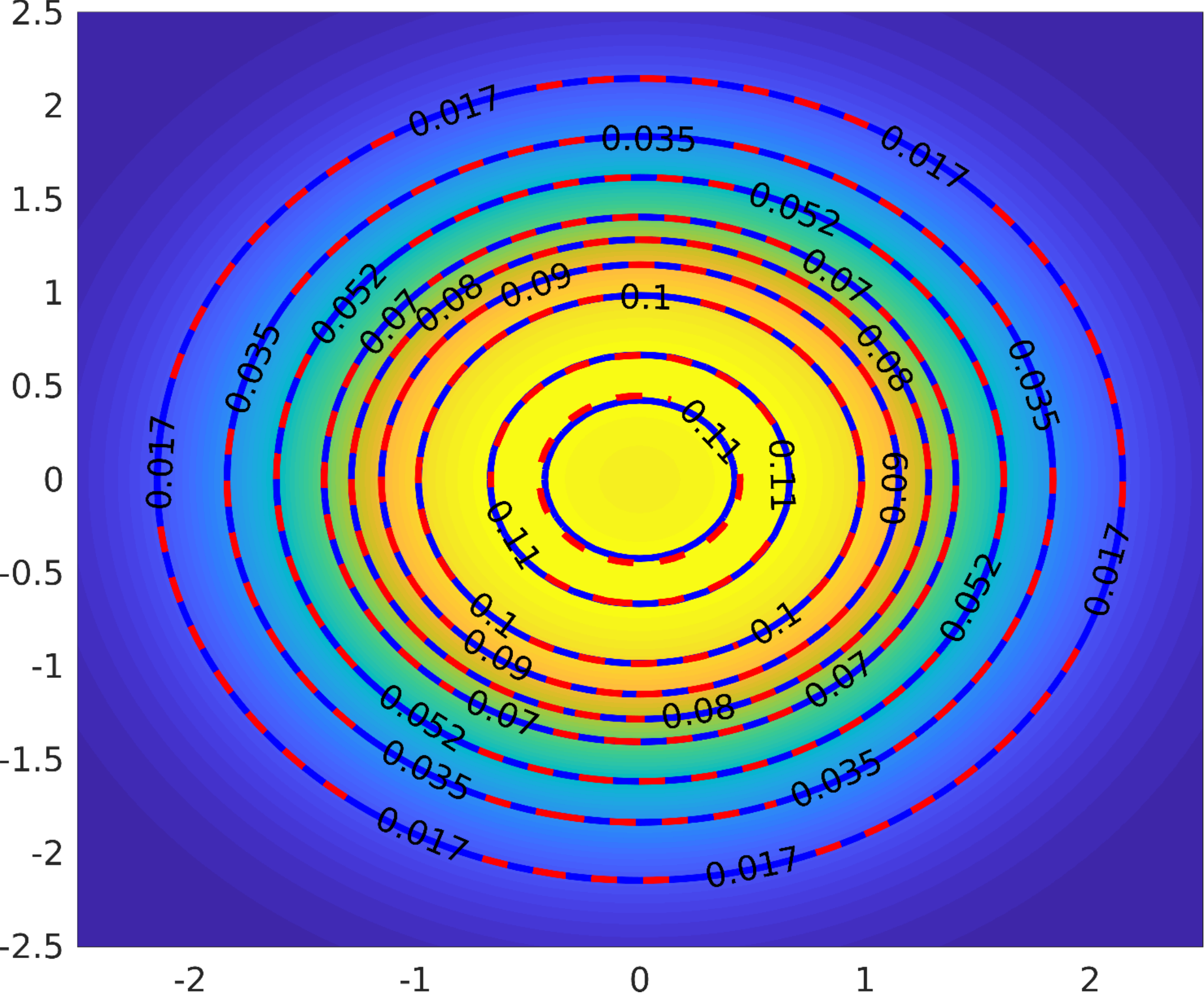}
}\hfill
\subfloat[Initial MDF $h(0,v_1,v_2)$\label{fig:ex3_init_2d}]{%
  \includegraphics[height=.25\textwidth]{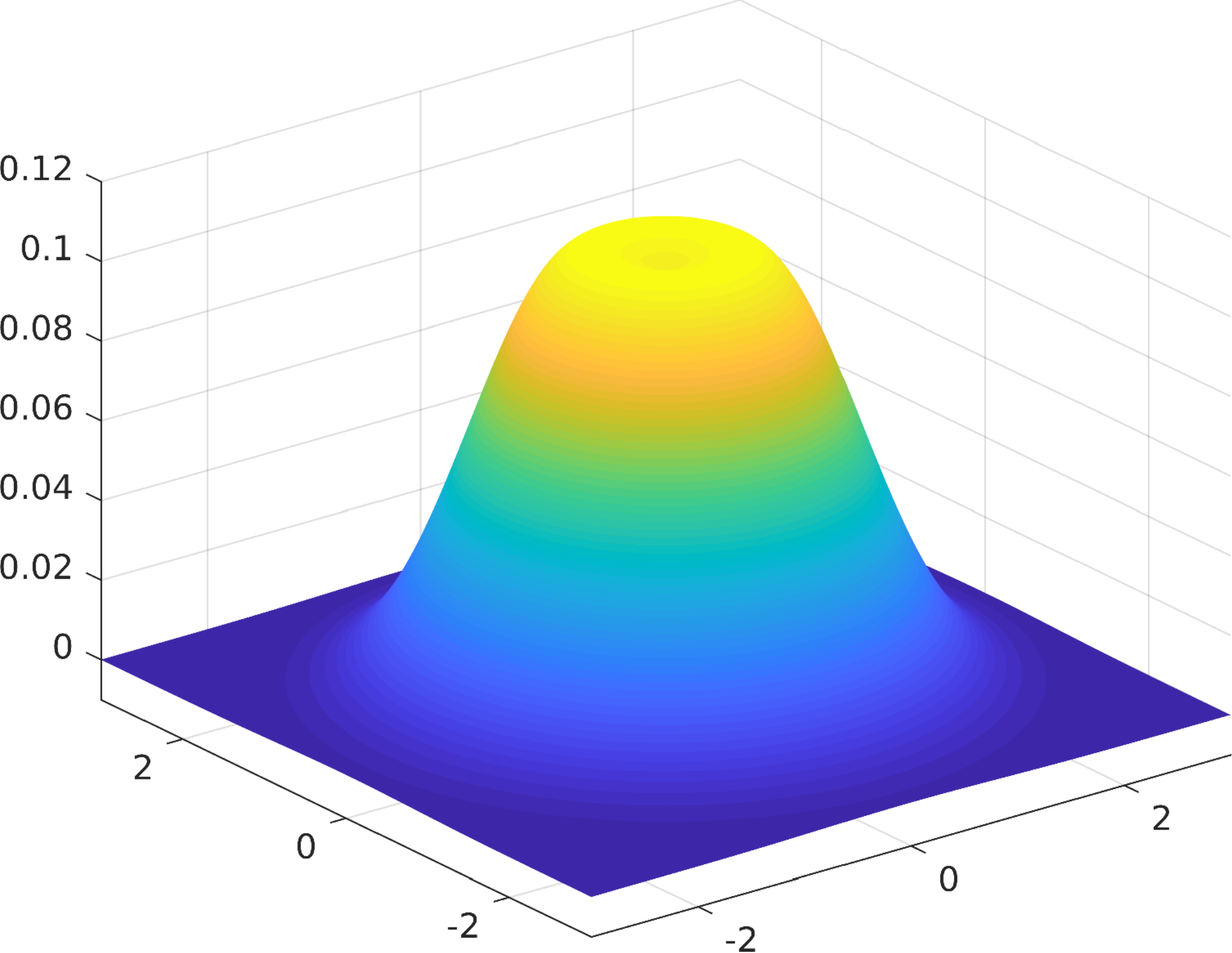}
}
\caption{Figure (a) is the initial marginal distribution function
  $g(0, v_1)$.  The red solid line corresponds to the exact solution,
  and the blue dashed line corresponds to the numerical approximation.
  Figure (b) is the initial marginal distribution functions
  $h(0, v_1, v_2)$. The blue solid lines correspond to the exact
  solution, and the red dashed lines correspond to the numerical
  approximation. Figure (c) shows only the numerical approximation.}
\label{fig:ex3_init}
\end{figure}

\begin{figure}[!ht]
\centering
\subfloat[$t=0.4$]{%
  \includegraphics[width=.3\textwidth]{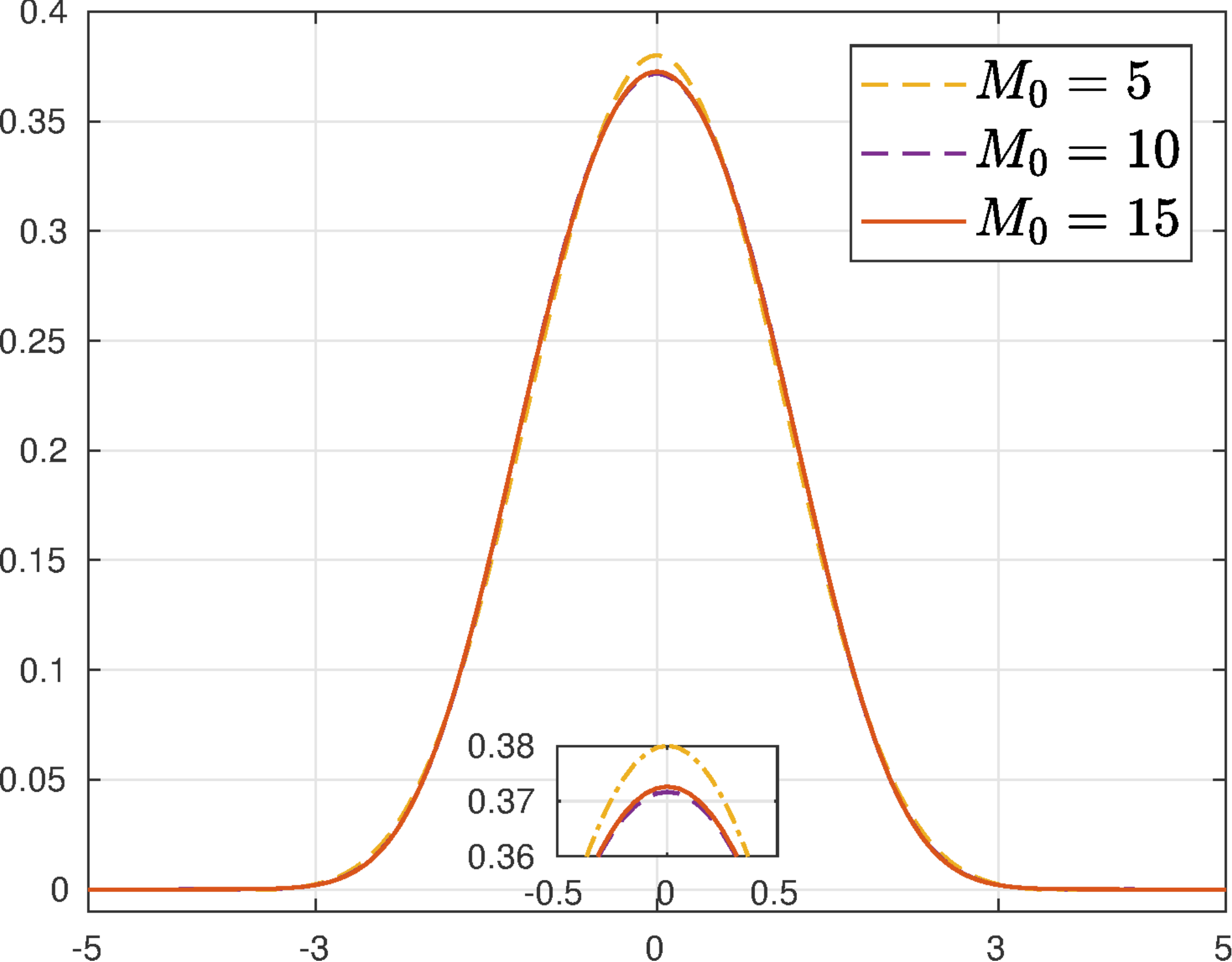}
} \hfill
\subfloat[$t=0.6$]{%
  \includegraphics[width=.3\textwidth]{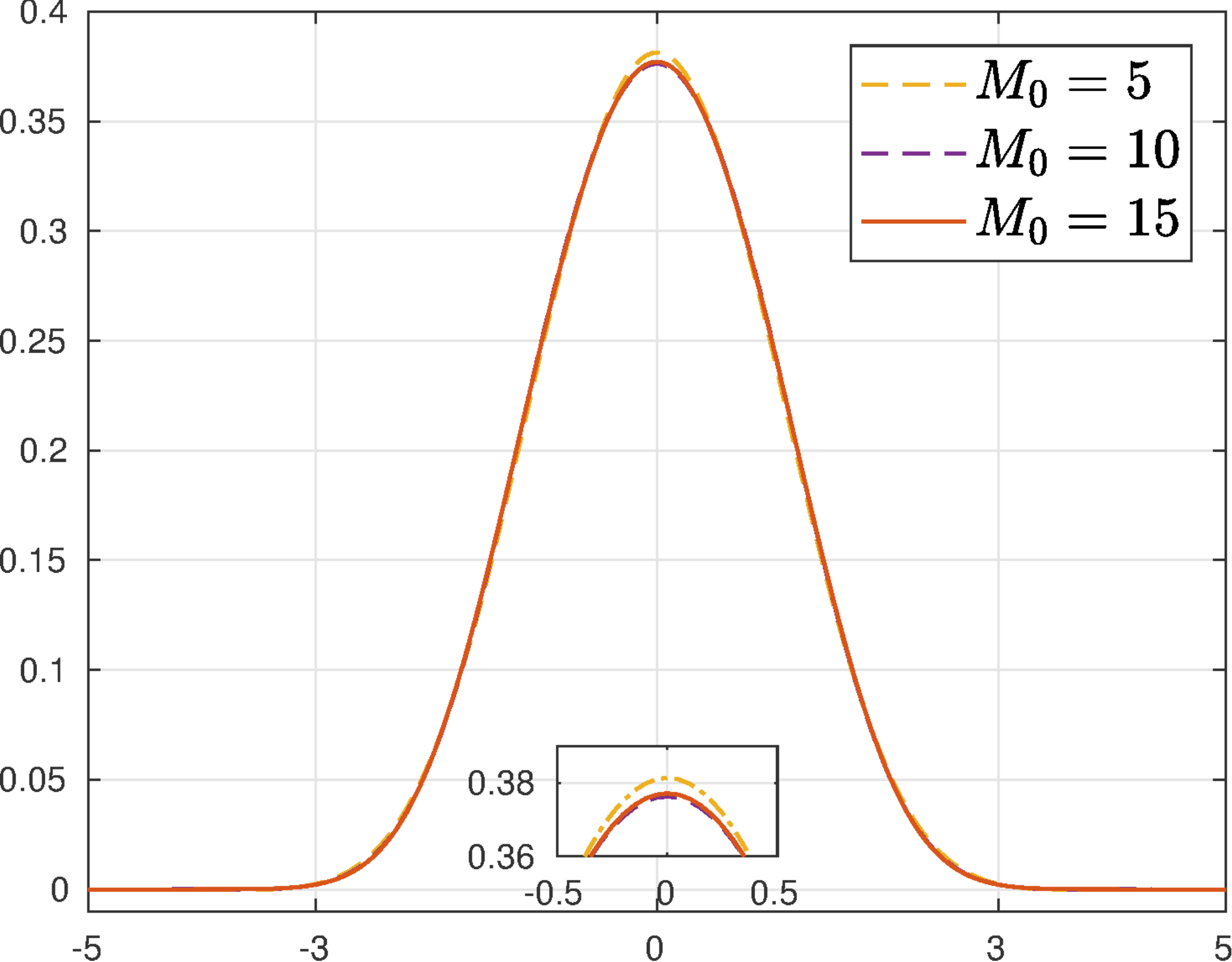}
} \hfill
\subfloat[$t=2$]{%
  \includegraphics[width=.3\textwidth]{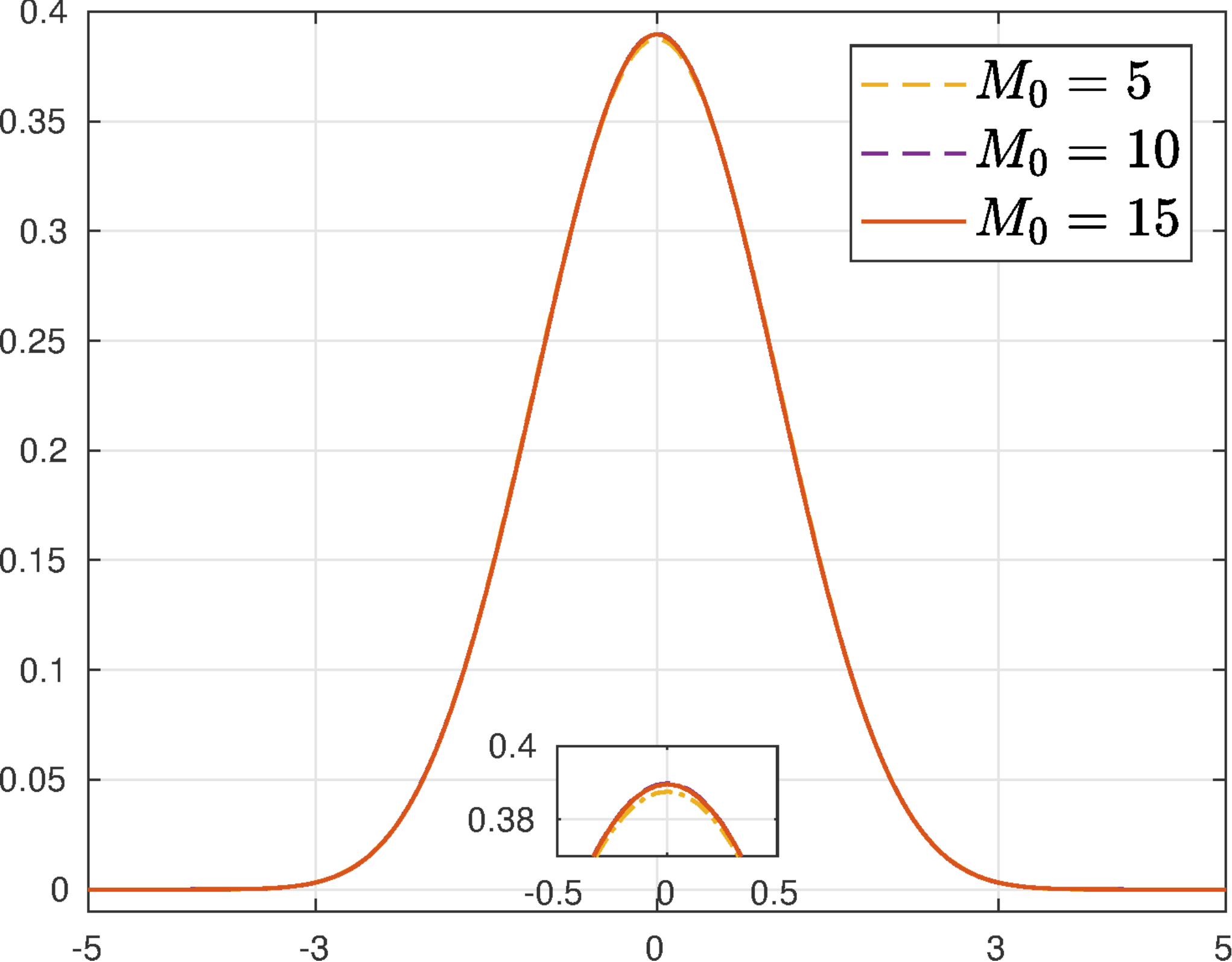}
}
\caption{The Coulombian case $\gamma = -3$. Marginal distribution
functions at different times.}
\label{fig:ex3_1d}
\end{figure}

\begin{figure}[!ht] \centering \subfloat[$t=0.4, M_0 = 5$]{%
  \includegraphics[width=.33\textwidth]{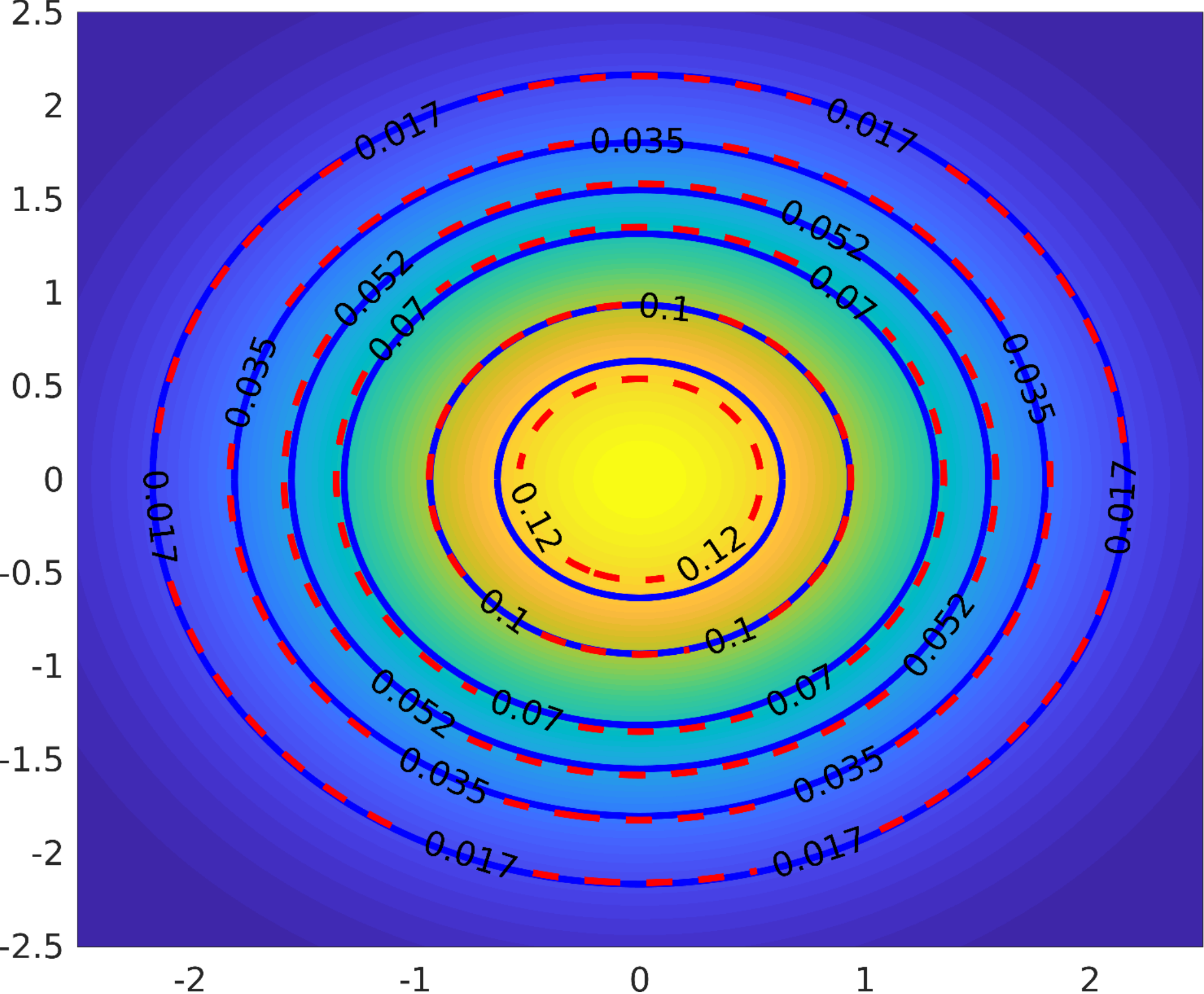}
} 
\subfloat[$t=0.6, M_0 = 5$]{%
  \includegraphics[width=.33\textwidth]{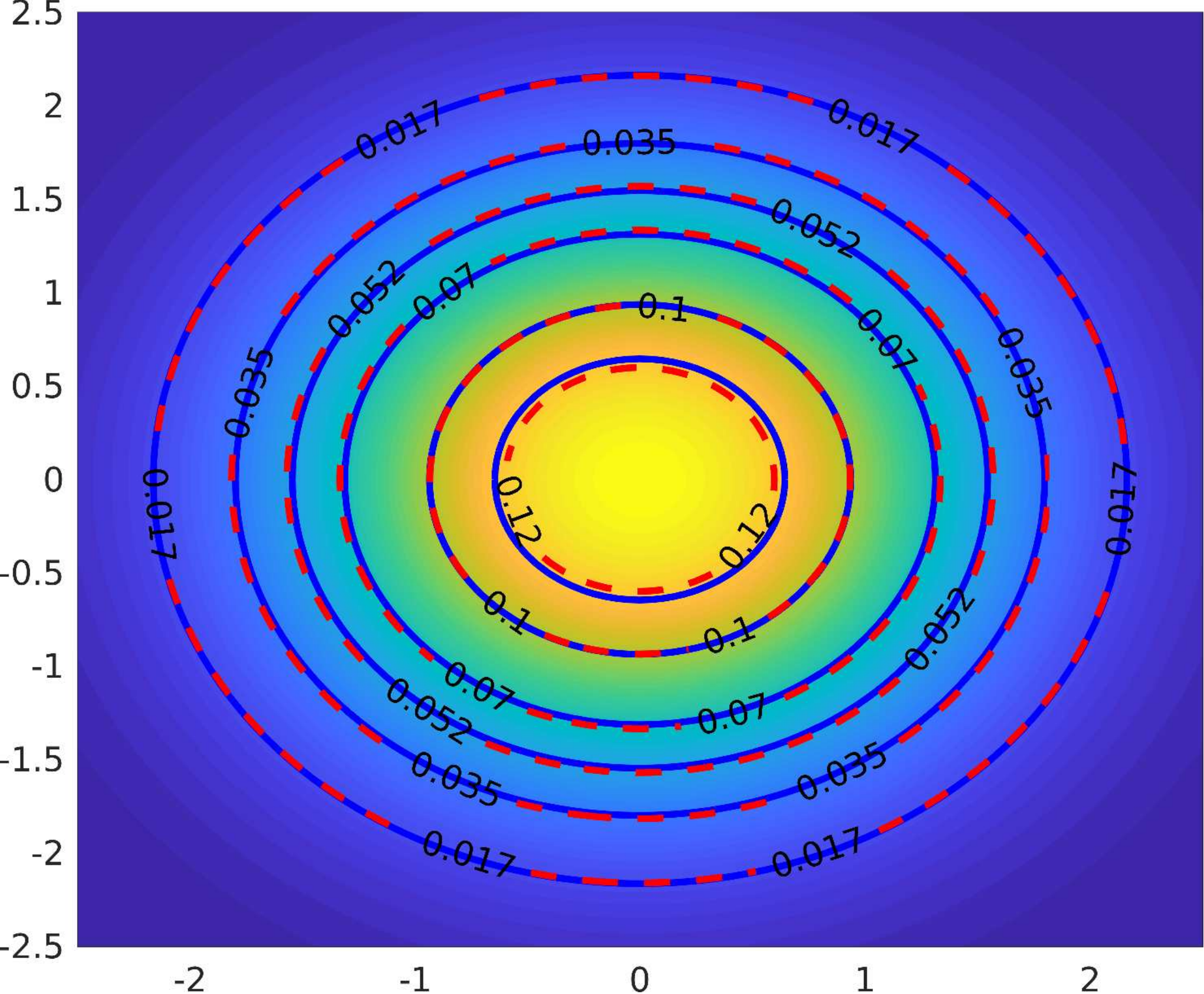}
} 
\subfloat[$t=2, M_0 = 5$]{%
  \includegraphics[width=.33\textwidth]{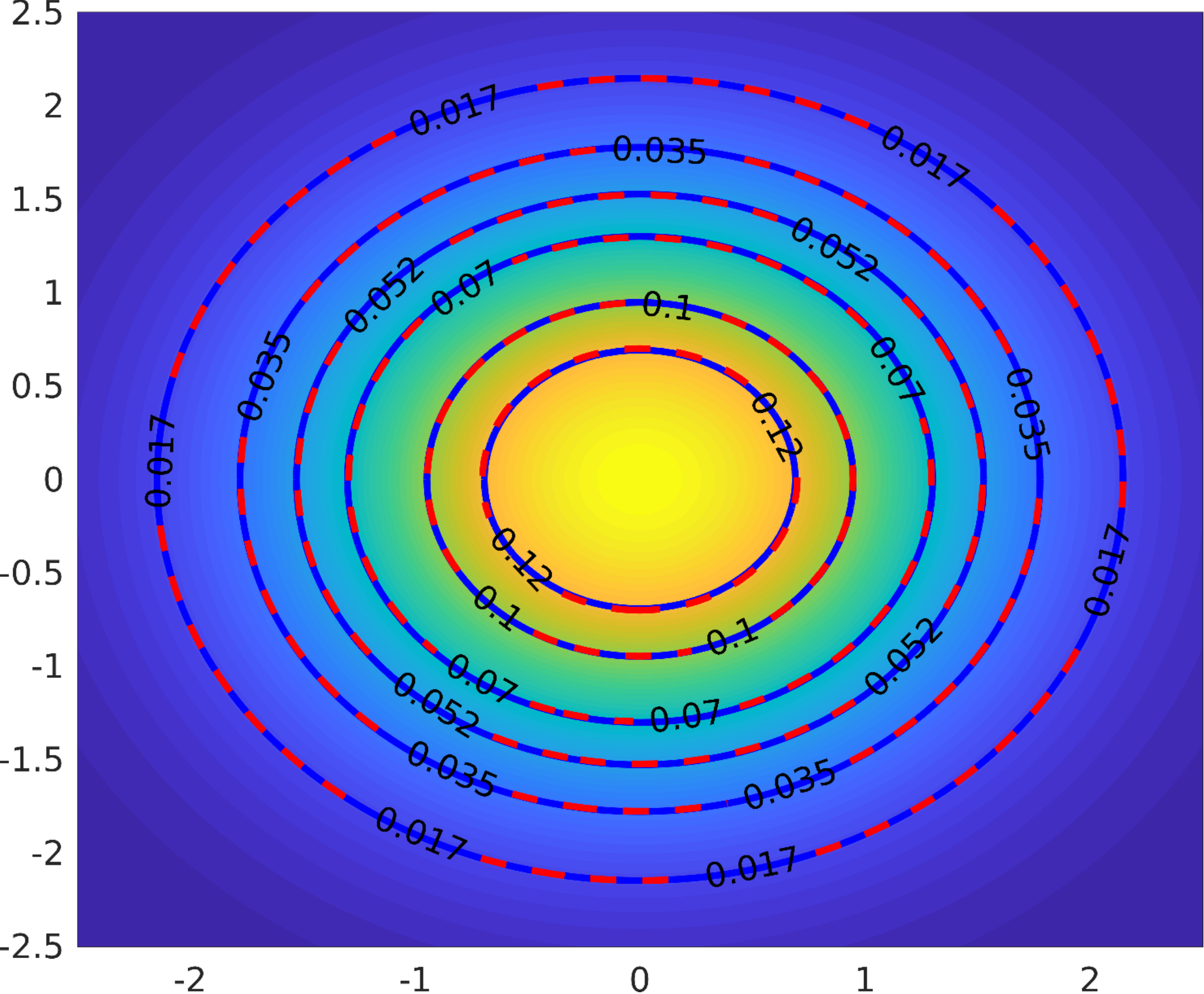}
} \\
\subfloat[$t=0.4, M_0 = 10$]{%
  \includegraphics[width=.33\textwidth]{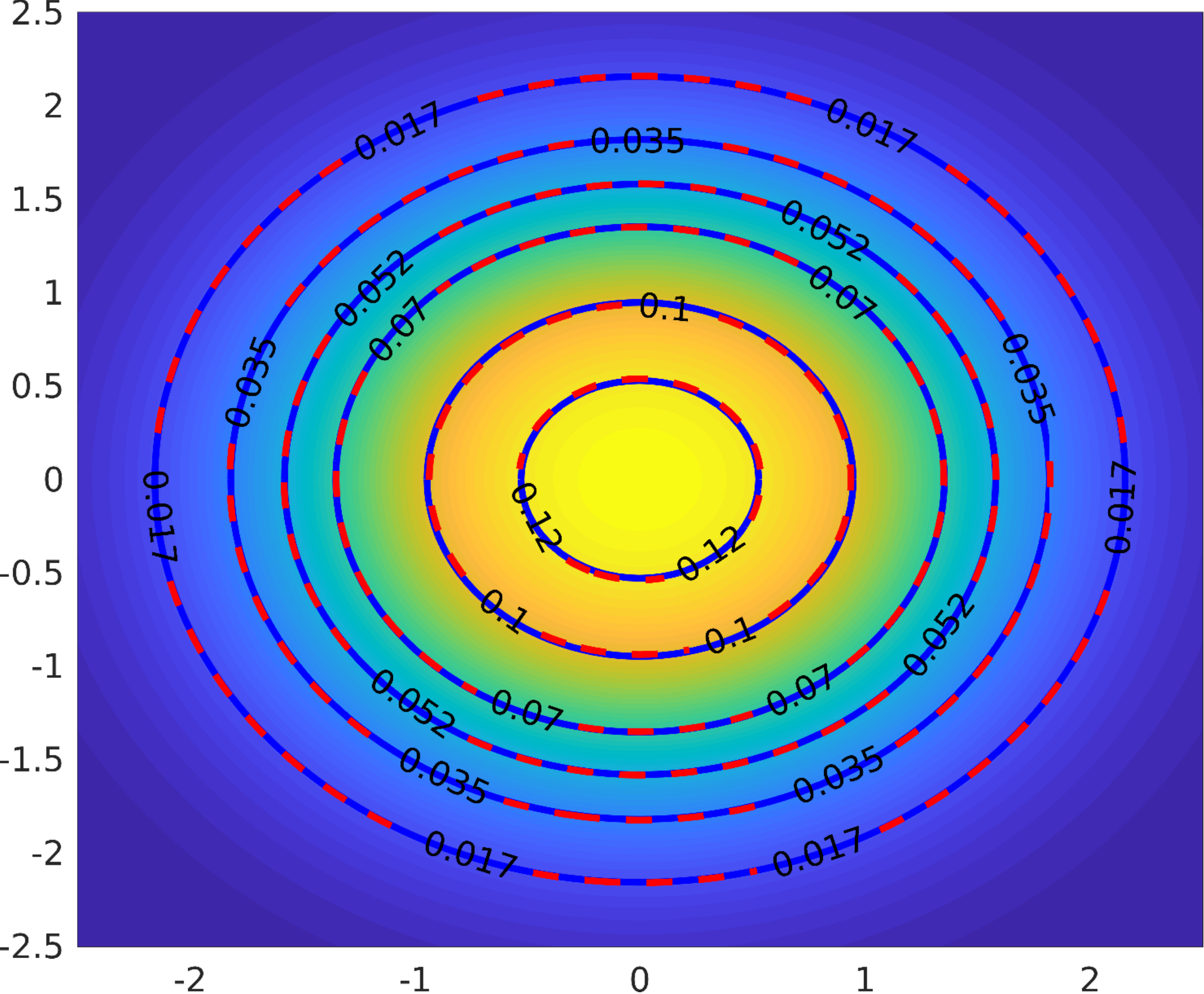}
} 
\subfloat[$t=0.6, M_0 = 10$]{%
  \includegraphics[width=.33\textwidth]{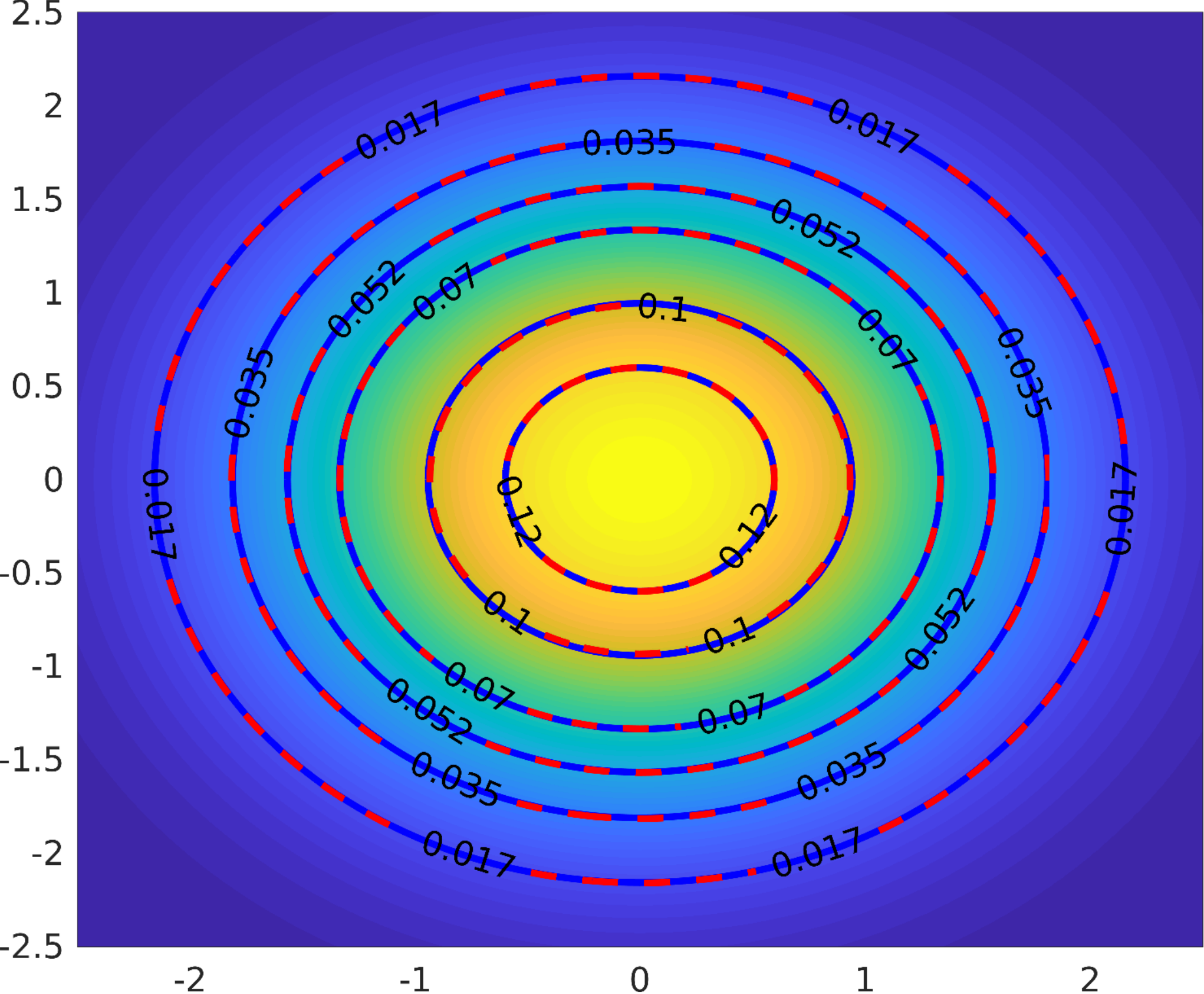}
} 
\subfloat[$t=2, M_0 = 10$]{%
  \includegraphics[width=.33\textwidth]{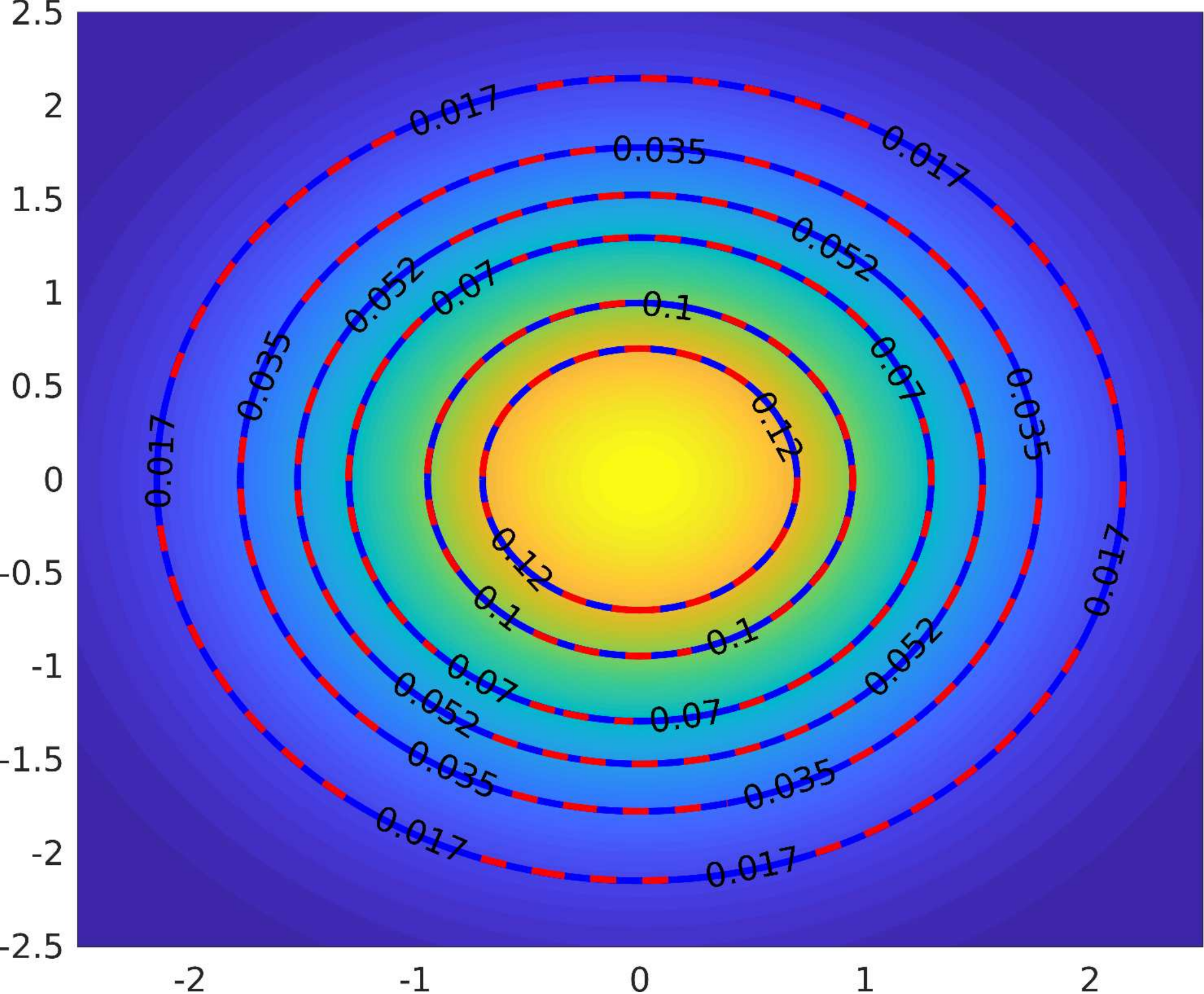}
}
\caption{The Coulombian case $\gamma = -3$. Comparison of numerical
  solutions and the reference solution $M_0 =15$. The red dashed
  contours are the reference solution. The blue solid contours in
  different columns are respectively the numerical solutions $M_0 = 5$
  and $M_0=10$.}
\label{fig:ex3_2d}
\end{figure}

\begin{figure}[!ht]
\centering
\subfloat[$t=0.4$]{%
  \includegraphics[width=.3\textwidth]{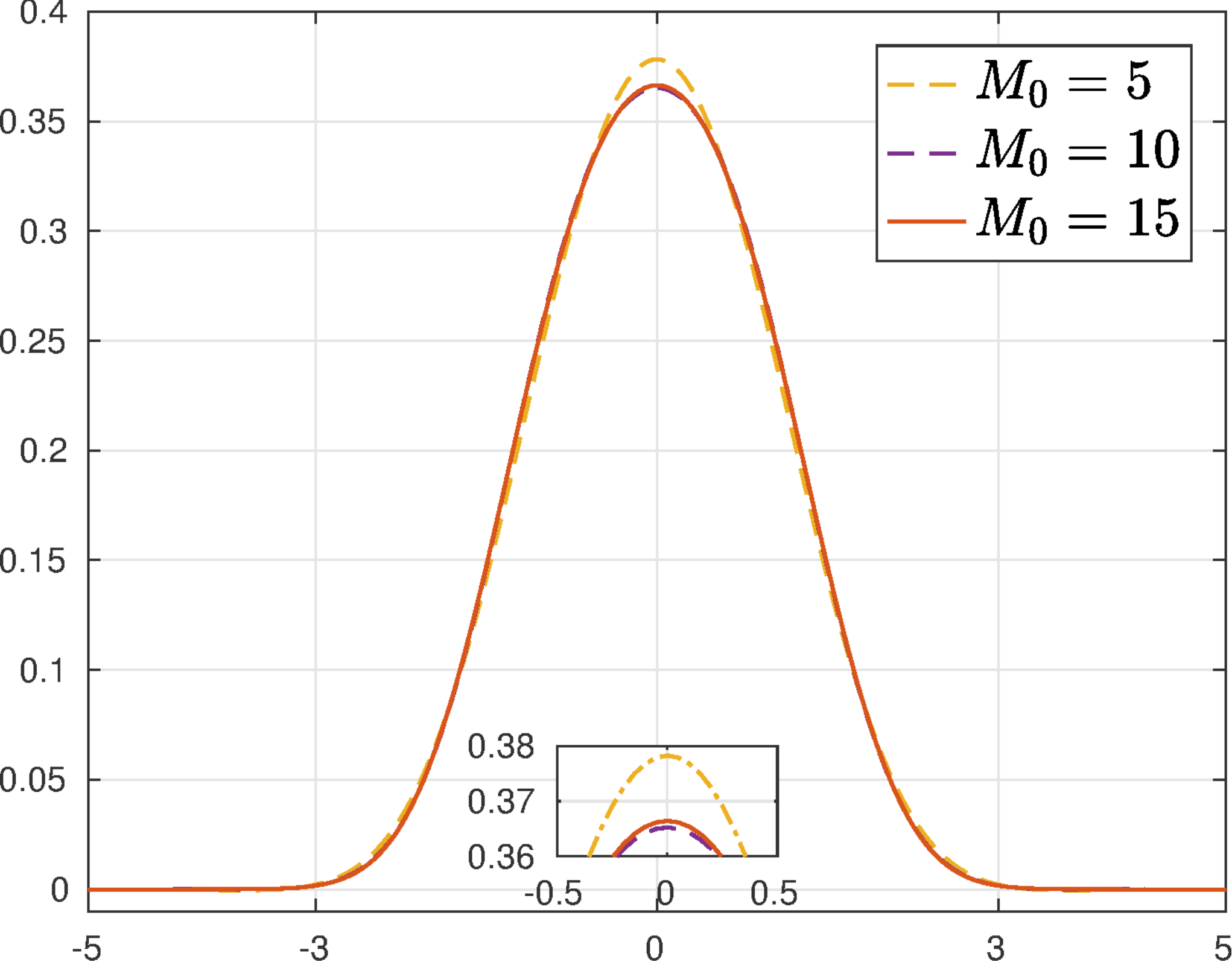}
} \hfill
\subfloat[$t=0.6$]{%
  \includegraphics[width=.3\textwidth]{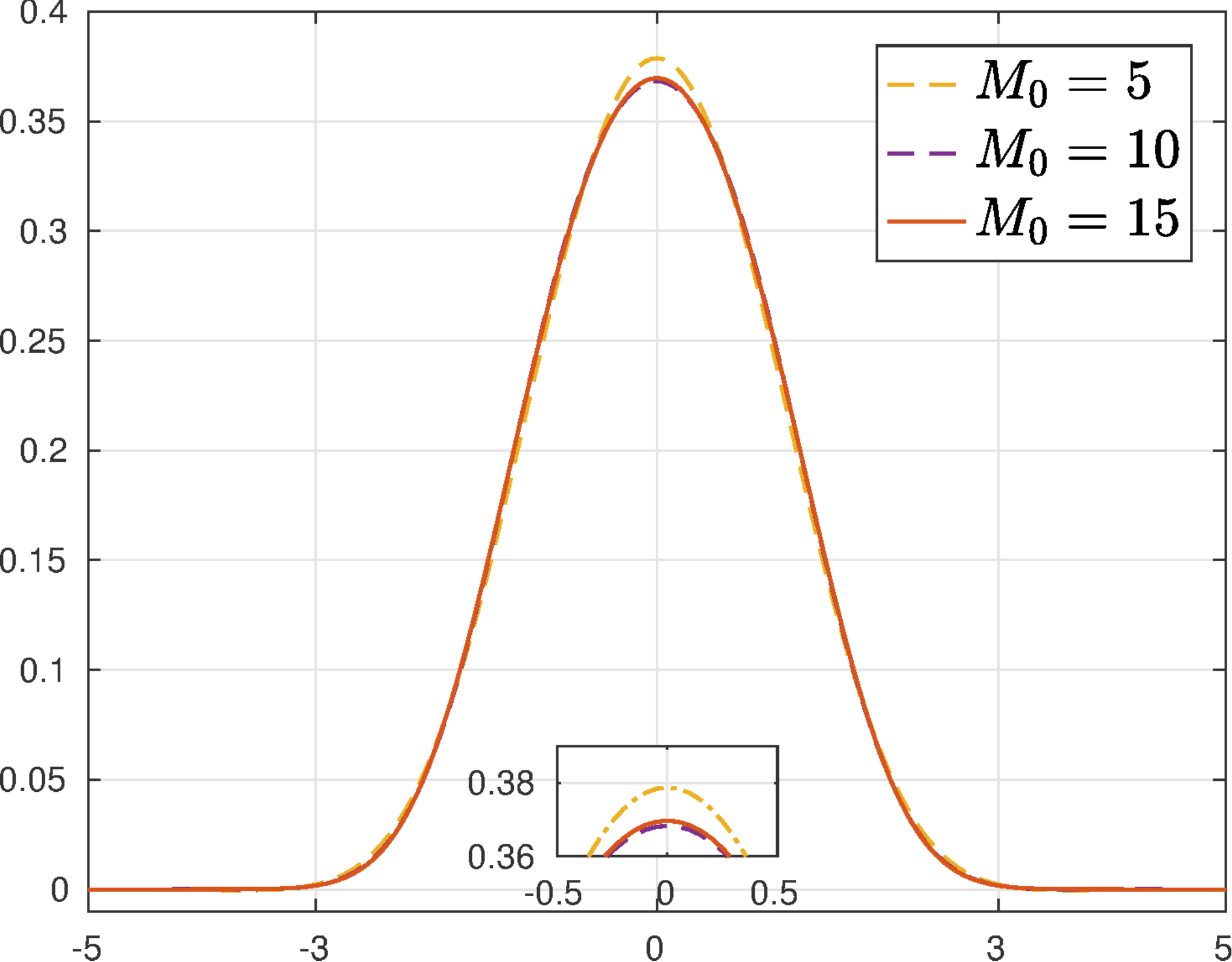}
} \hfill
\subfloat[$t=2$]{%
  \includegraphics[width=.3\textwidth]{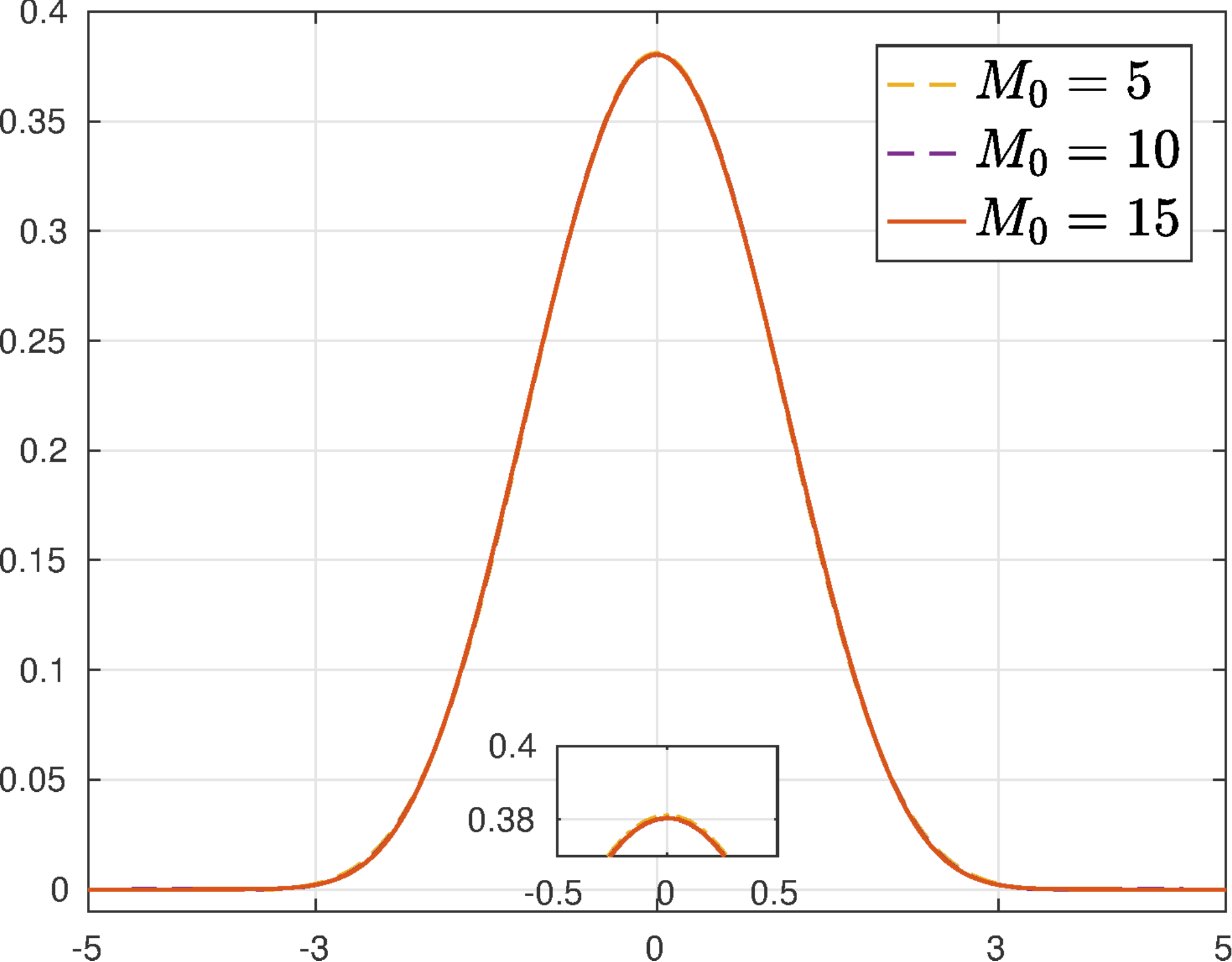}
} \\
\subfloat[$t=0.4, M_0 = 5$]{%
  \includegraphics[width=.33\textwidth]{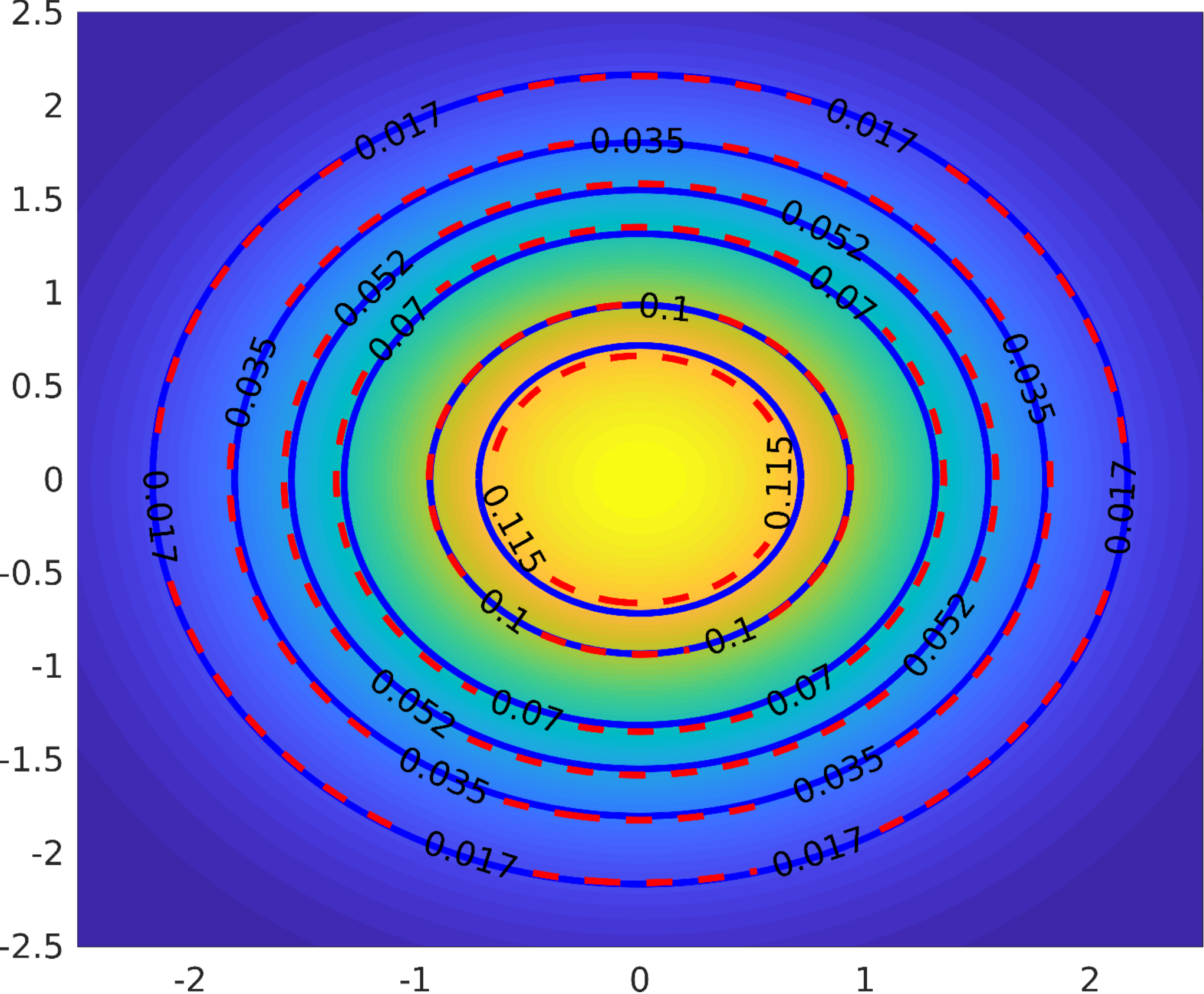}
} 
\subfloat[$t=0.6, M_0 = 5$]{%
  \includegraphics[width=.33\textwidth]{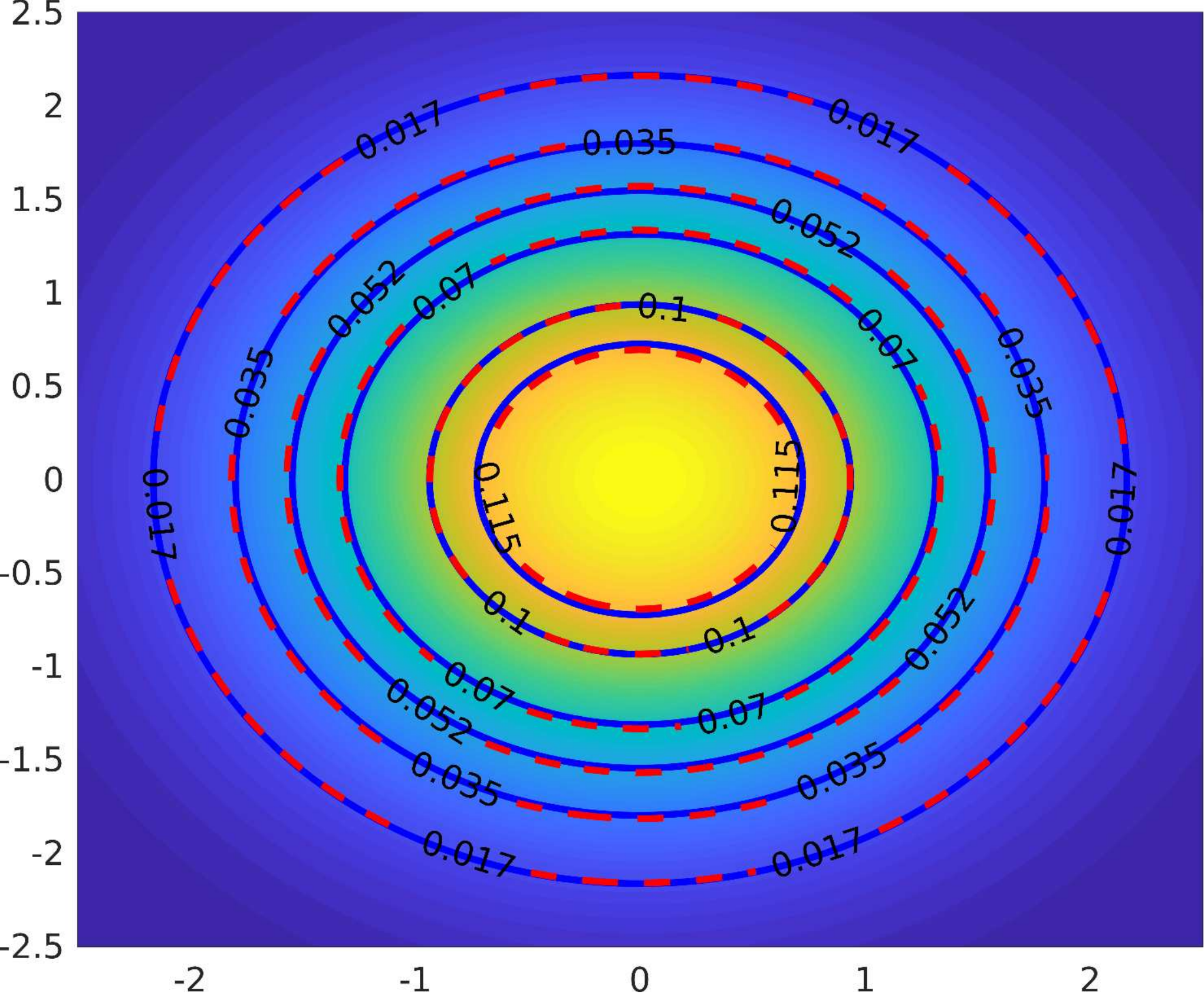}
} 
\subfloat[$t=2, M_0 = 5$]{%
  \includegraphics[width=.33\textwidth]{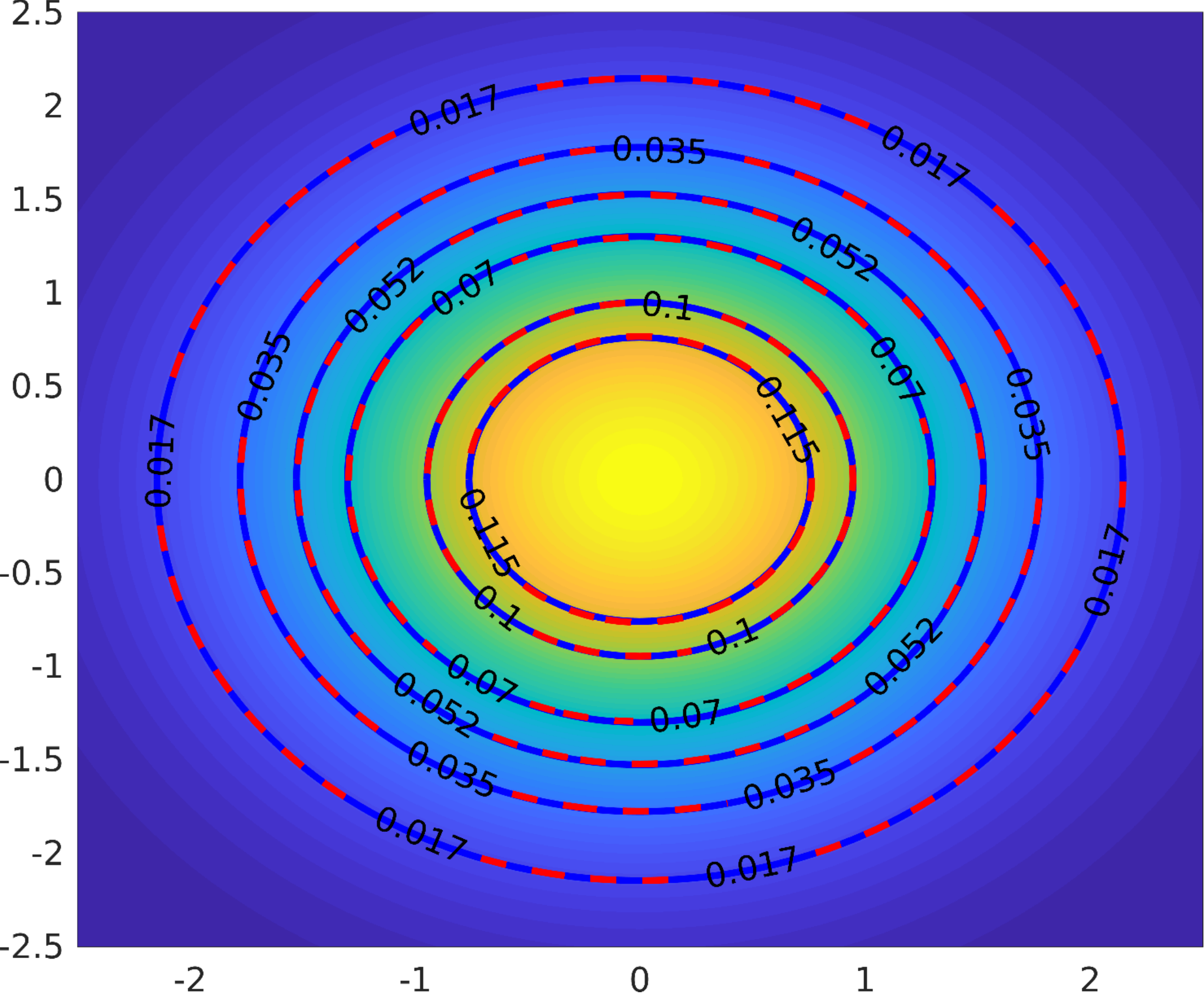}
}\\
\subfloat[$t=0.4, M_0 = 10$]{%
  \includegraphics[width=.33\textwidth]{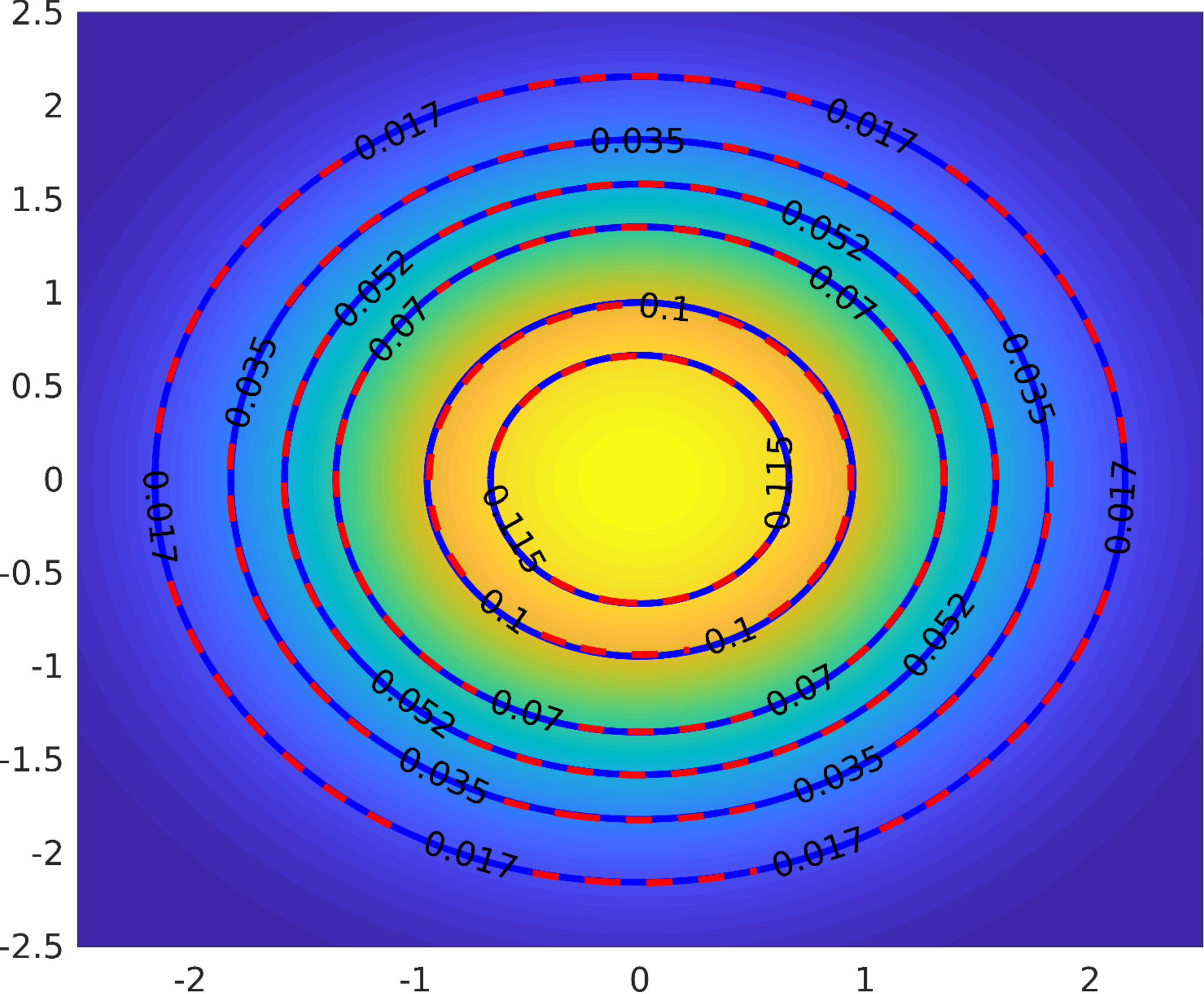}
} 
\subfloat[$t=0.6, M_0 = 10$]{%
  \includegraphics[width=.33\textwidth]{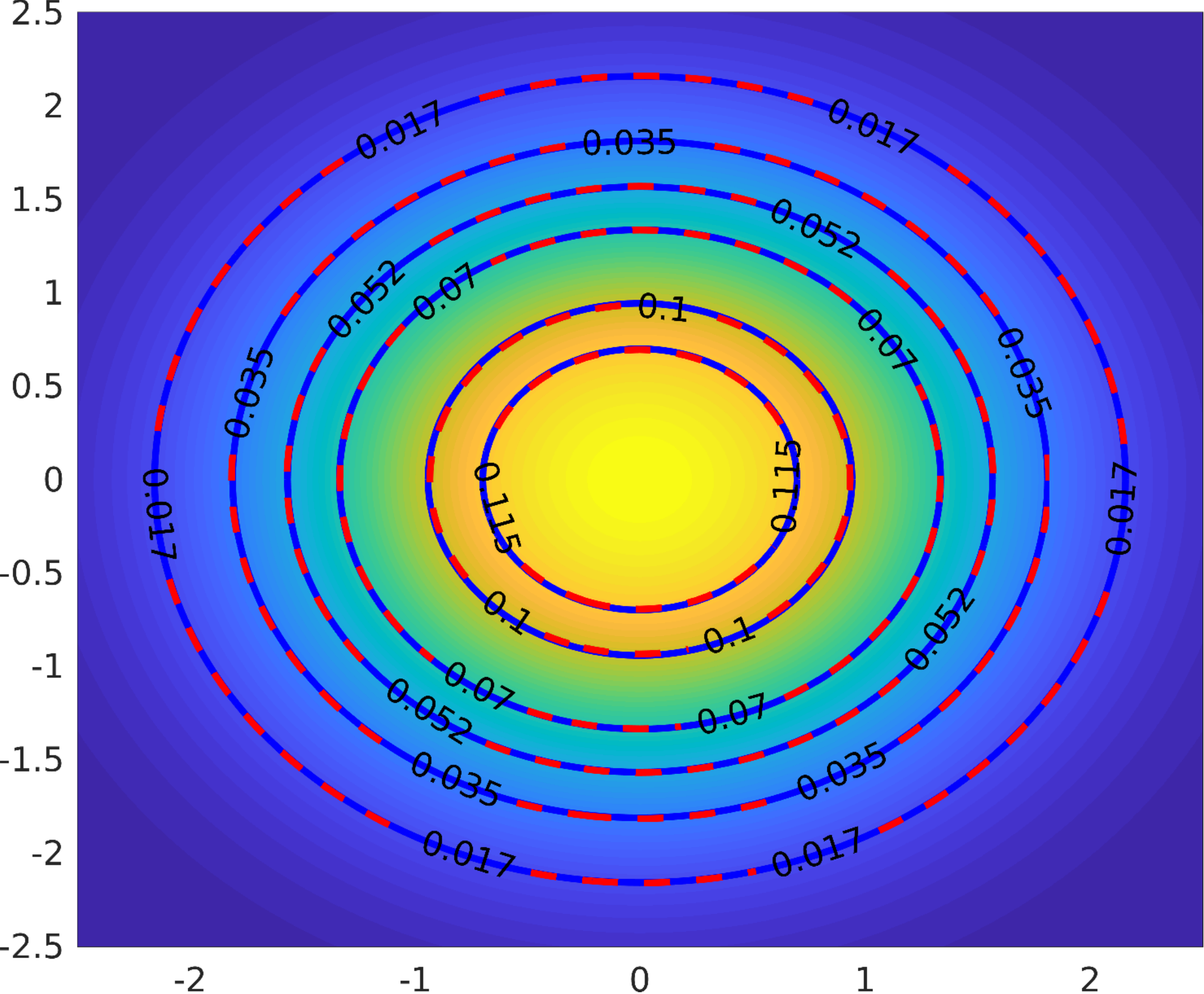}
} 
\subfloat[$t=2, M_0 = 10$]{%
  \includegraphics[width=.33\textwidth]{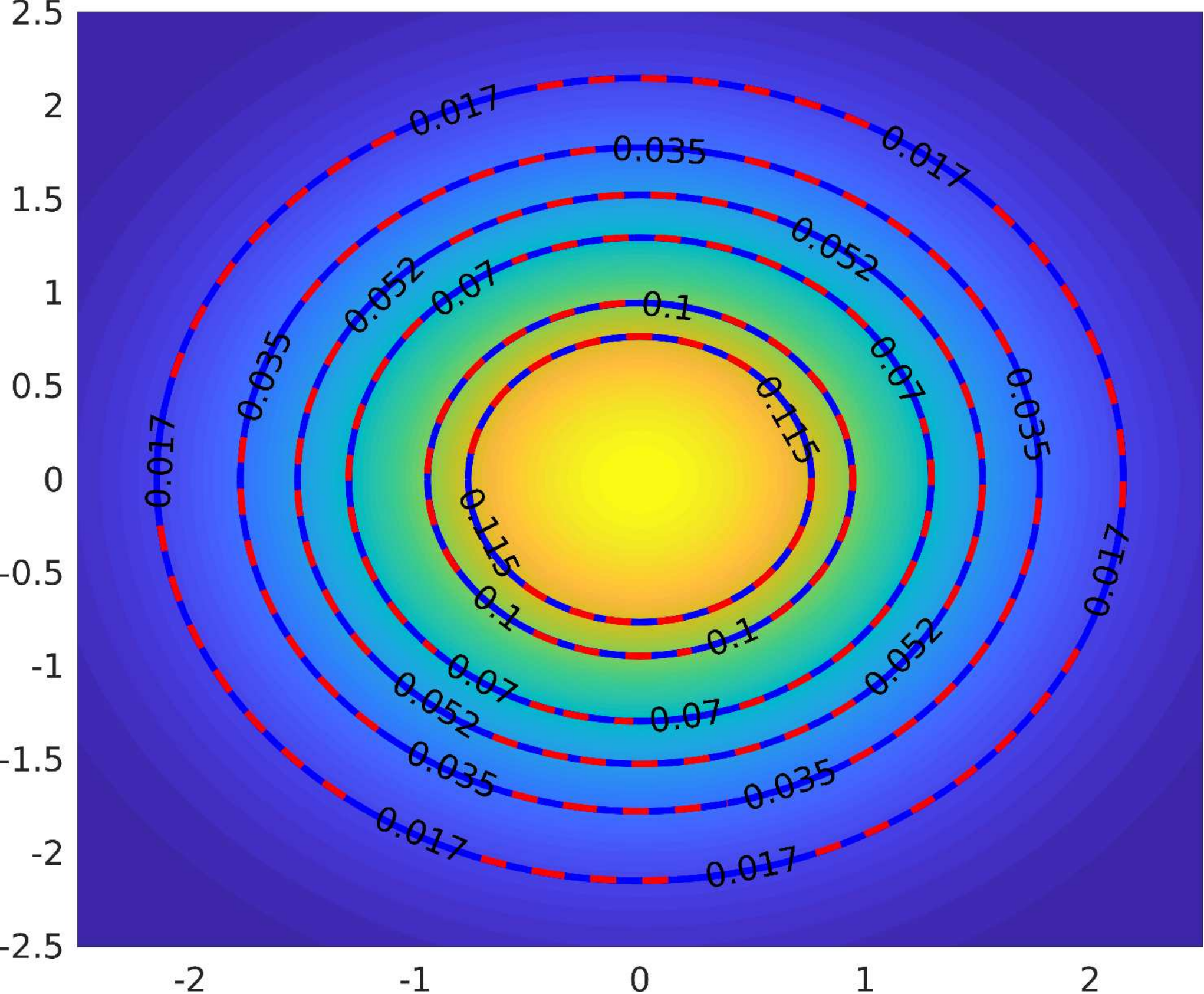}
}
\caption{The Coulombian case $\gamma = -4.9$. Marginal distribution
  functions at different times. The first row is the marginal
  distribution $g(t, v_1)$, and the latter two rows are the marginal
  distribution functions $h(t, v_1, v_2)$ with different $M_0$.  The
  red dashed contours are the reference solutions $M_0 = 15$. The blue
  solid contours at different rows are respectively the numerical
  solutions when $M_0 = 5$ and $M_0=10$.}
\label{fig:ex3_4p9}
\end{figure}

%%% Local Variables: 
%%% mode: latex
%%% TeX-master: "article"
%%% End

% vim: tw=70:spell
\section{Conclusion} 
\label{sec:conclusion} 
In this paper, we focus on applying the Hermite spectral method to
develop an efficient and accurate way of approximating and numerically
solving the Fokker-Planck-Landau equation. Basic properties of Hermite
polynomials are utilized to obtain a simplified expression of the
coefficients, which renders the numerical method feasible. Burnett
polynomials are introduced to deal with the super singular integral in
the computation. This method could cover more practical cases up to
$\gamma>-5$.

A novel collision model is built with a combination of quadratic
collision model and the linearized collision model brought up by
C. Villani \cite{Villani1998on}. The numerical experiments validate
the efficiency of this new model. With the model introduced, the
numerical solutions are of a high accuracy, as well as an affordable
computational cost. This method should be further validated in the
numerical tests for the full FPL equation with spatial variables,
which will be one of the future works.

\section*{Acknowledgements}
Ruo Li is supported by the National National Scientific Foundation of
China (Grant No. 91630310) and Science Challenge Project
(No. TZ2016002). Yanli Wang is supported by the National Natural
Scientific Foundation of China (Grant No. 11501042), and Chinese
Postdoctoral Science Foundation of China (2018M631233).

\section{Appendix}
\label{sec:appendix}
\subsection{Proof of Theorem \ref{thm:coeA}}
\label{sec:Appendix_coeA} We shall present of proof of Theorem 1 here.
In order to prove Theorem \ref{thm:coeA}, we first introduce the lemma
below:
\begin{lemma}
\label{thm:Hermite_cv}
Let $\bv = \bh + \bg/2$ and $\bw = \bh - \bg/2$. It holds that
\begin{displaymath}
\begin{split}
  & H^{\alpha}(\bv) H^{\kappa}(\bw) =  \sum_{\alpha'+\kappa' = \alpha+\kappa}
  a_{\alpha'\kappa'}^{\alpha\kappa} H^{\alpha'}(\sqrt{2}\bh) H^{\kappa'} \left( \frac{\bg}{\sqrt{2}} \right),
\end{split}
\end{displaymath}
where the coefficients $a_{\alpha'\kappa'}^{\alpha\kappa}$ are
defined in \eqref{eq:coea}.
\end{lemma}
\begin{corollary}
  \label{thm:Hermit_split}
  Let $\bv = \bh + \bg/2$. We have
  \begin{displaymath}
    H^{\alpha}(\bv) = \sum_{\kappa + \lambda = \alpha}
    \frac{2^{-|\alpha|/2} \alpha!}{\kappa!\lambda !}
    H^{\kappa}\left(\sqrt{2} \bh\right)
    H^{\lambda} \left( \frac{\bg}{\sqrt{2}} \right).
  \end{displaymath}
\end{corollary}
The proof of Lemma \ref{thm:Hermite_cv} can be found in
\cite{QuadraticCol}. Next we will prove Theorem \ref{thm:coeA}. 

{\renewcommand\proofname{Proof of Theorem \ref{thm:coeA}}

\begin{proof} 
Using an integration by parts and the recursion formula of Hermite
polynomials 
\begin{equation}
  \label{eq:recursion_H}
  \frac{\partial}{\partial v_s} \Big(\mM(\bv) H^{\alpha}(\bv)\Big) =
 (-1)\mM(\bv)H^{\alpha+e_s}(\bv), \qquad s = 1, 2, 3,
\end{equation}
the coefficients $A_{\alpha}^{\lambda,\kappa}$ \eqref{eq:coeA_detail}
can be simplified as
\begin{equation}
  \label{eq:coeA_detail1}
  \begin{aligned}
    A_{\alpha}^{\lambda,\kappa} = \sum_{s,t
      =1}^3 & \frac{\Lambda}{(\alpha-e_s)!}  \int_{\bbR^3\times\bbR^3} 
    \lvert \bv-\bv^*\rvert^\gamma G_{st}(\bv-\bv^*) \\
    & H^{\alpha - e_s}(\bv^*) \mM(\bv)\mM(\bv^*) 
    \left(H^{\lambda}(\bv^{\ast})H^{\kappa+e_t}(\bv)-H^{\lambda}(\bv)H^{\kappa+e_t}(\bv^*)\right)\dd
    \bv \dd \bv^{\ast},
\end{aligned}
 \end{equation}
 where $G_{st}(\bv)=-v_sv_t+\delta_{st}|\bv|^2$.  Further
 simplification of \eqref{eq:coeA_detail1} follows the method in
 \cite{QuadraticCol}, where the velocity of the mass center is
 defined as $\bh = (\bv+\bv^{\ast})/2$ and the relative velocity is
 defined as $\bg = \bv - \bv^{\ast}$. Hence, it holds that 
    \begin{gather}
      \label{eq:var_change}
      \bv = \bh + \frac{1}{2} \bg, \quad
      \bv^{\ast} = \bh - \frac{1}{2} \bg, \quad
      |\bv|^2 + |\bv^{\ast}|^2 = \frac{1}{2} |\bg|^2 + 2|\bh|^2, \qquad
      \mathrm{d}\bv \, \mathrm{d}\bv^{\ast} = \mathrm{d}\bg \, \mathrm{d}\bh.
    \end{gather}
    Combing Lemma \ref{thm:Hermite_cv}, \eqref{eq:coeA_detail1} and
    \eqref{eq:var_change}, the integral in
    $A_{\alpha}^{\lambda,\kappa}$ can be rewritten with respect to
    $\bg$ and $\bh$
    \begin{equation}
      \label{eq:coeA_detail2}
      \begin{aligned}
        A_{\alpha}^{\lambda,\kappa} =2^{(\gamma/2+3 - |\alpha|)/2}\sum_{s,t=
          1}^3\sum_{p+q=k-e_s}
        \sum_{r+\beta=\lambda+\kappa} \int_{\bbR^3\times \bbR^3} \lvert \bg
        \rvert^\gamma  \frac{\Lambda}{p!q!}
        G_{st}(\bg) H^q(\bg) \mM(\bg) \mM(\bh)  \\
        \left(a_{(\beta+e_t)r}^{(\kappa+e_t)\lambda}H^r(\bg)H^{\beta+e_t}(\bh)H^p(\bh)
          - a_{r(\beta+e_t)}^{\lambda(\kappa+e_t)}H^{\beta+e_t}(\bg)H^{r}(\bh)H^p(\bh)\right)
        \dd \bg \dd \bh.
      \end{aligned}
    \end{equation}
    Using the orthogonality of Hermite polynomials, we can finally
    prove Theorem \ref{thm:coeA}.

  \end{proof}
}

\subsection{Proof of Theorem \ref{thm:coeF}}
\label{sec:Appendix_coeF}

In order to prove Theorem \ref{thm:coeF}, we will  introduce the lemma
below:
\begin{lemma}
\label{thm:coe_Y}
  For three spherical harmonics $Y_l^m$, $Y_{l_1}^{m_1}$ and
  $Y_{l_2}^{m_2}$, if $m \neq m_1 + m_2$, or $l \not\in [|l_1 - l_2|,
  l_1 + l_2]$, then 
  \begin{equation}
    \int_{\bbS^2} Y_{l_1}^{m_1}({\bn})Y_{l_2}^{m_2}({\bn})
    \overline{Y_{l}^m({\bn})} \dd {\bn} = 0.
  \end{equation}
  Especially, we have 
  \begin{equation}
    \label{eq:Y}
     Y_l^m({\bn})Y_1^{\mu}(\bn) =
     \sqrt{\frac{3}{4\pi}}\left(\eta_{l+1,m}^{\mu}Y_{l+1}^{m+\mu}({\bn})
       + (-1)^{\mu}\eta_{-l, m}^{\mu}Y_{l-1}^{m+\mu}({\bn})\right),
       \qquad \mu = -1, 0, 1,
  \end{equation}
where $\eta_{lm}^{\mu}$ is defined in \eqref{eq:coe_eta}.
\end{lemma}
The result of this lemma can be found in Section 12.9 of
\cite{Arfken}.

{\renewcommand\proofname{Proof of Theorem \ref{thm:coeF}}
\begin{proof}
  Noting that 
  \begin{equation}
  \label{eq:n_Y}
  \begin{aligned}
    n_1 = \sqrt{\frac{2\pi}{3}}\left(Y_1^1 - Y_1^{-1}\right), \qquad 
    n_2 = -\imag \sqrt{\frac{2\pi}{3}}\left(Y_1^1 + Y_1^{-1}\right), \qquad 
    n_3 = 2\sqrt{\frac{\pi}{3}}Y_1^0.
  \end{aligned}
\end{equation}

% \begin{equation}
%   \label{eq:recursion_Y}
%   \begin{aligned}
%     n_1 Y_{l}^m(\bn)  &= \frac{1}{\sqrt{2}} 
%     \sum_{\mu = -1, 1} \left(\eta_{l+1,m}^{\mu}Y_{l+1}^{m+\mu}({\bn})
%   + (-1)^{\mu}\eta_{-l, m}^{\mu}Y_{l-1}^{m+\mu}({\bn})\right),\\
% & \left[-\eta_{l +1, m}^1 Y_{l+1}^{m+1} + \eta_{-l, m}^1Y_{l-1}^{m+1}
%     + \eta_{l+1,m}^{-1}Y_{l+1}^{m-1} - \eta_{-l,m}^{-1}Y_{l-1}^{m-1}\right], \\ 
%     n_2 Y_{l}^m(\bn)  &= \frac{\imag}{\sqrt{2}}\left[\eta_{l +1, m}^1 Y_{l+1}^{m+1} - \eta_{-l, m}^1Y_{l-1}^{m+1}
%     +\eta_{l+1,m}^{-1}Y_{l+1}^{m-1} - \eta_{-l,m}^{-1}Y_{l-1}^{m-1}\right], \\ 
%     n_3 Y_{l}^m(\bn)  &= \left[\eta_{l +1, m}^0 Y_{l+1}^{m} + \eta_{-l, m}^0Y_{l-1}^{m}\right].
%   \end{aligned}
% \end{equation}

% Recall that
% $F(3, 3, l_1, m_1, l_2, m_2) =   \int_{\bbS^2}   n_3({\bn})
%     Y_{l_1}^{m_1}(\bn)  n_3Y_{l_2}^{m_2}({\bn})\dd {\bn}$
    
%     $F(1, 3, l_1, m_1, l_2, m_2) = \int_{\bbS^2}
%     (n_1({\bn})Y_{l_1}^{m_1}(\bn) n_3Y_{l_2}^{m_2}(\hat{\bn})) \dd
%     {\bn}$
    
% The recursion formula of the spherical harmonic $Y_{l}^m$ is 
% \begin{equation}
%   \label{eq:recursion1}
%   Y_l^m({\bn})Y_1^{\mu}({\bn}) = \sqrt{\frac{3}{4\pi}}
% (\eta_{l+1, m}^{\mu}Y_{l+1}^{m+\mu} + (-1)^{\mu}\eta_{-l, m}^{\mu}Y_{l-1}^{m + \mu}({\bn})), \quad \mu = -1, 0, 1,
% \end{equation}
% where $$\eta_{lm}^{\mu} = \sqrt{\frac{[l + (2\delta_{1, \mu} - 1)m + \delta_{1, \mu}][l - (2\delta_{-1, \mu} - 1)m + \delta_{-1, \mu}]}
%   {2^{|\mu|}(2l -1)(2l+1)}}.$$
% Based on \ref{eq:recursion1}, we can get that 

Based on Lemma \ref{thm:coe_Y} and the property of spherical harmonic
$\overline{Y_l^m(\bn)} = (-1)^m Y_l^{-m}(\bn)$, we can derive the
results in Theorem \ref{thm:coeF} with the orthogonality property of
spherical harmonics
\begin{equation}
  \label{eq:orth_sp}
  \int_{\bbS^2} Y_{l_1}^{m_1}(\bn) \overline{Y_{l_2}^{m_2}({\bn})}
  \dd{\bn}=\delta_{l_1l_2}\delta_{m_1m_2}.
\end{equation}
\end{proof}
}

\subsection{Computation of Coefficients 
  $C_{\hat{\balpha}}^{\balpha}$}
\label{sec:Appendix_CoeC}
In this section, we will briefly introduce the algorithm to calculate
$C_{\hat{\balpha}}^{\balpha}$, and the original algorithm is in
\cite{BurnettCol}. 

Define 
\begin{equation}
  \label{eq:S}
  S_{-1} = \frac{1}{2}(v_1  - \imag v_2), \quad S_0 = v_3, \quad S_1 =
  -\frac{1}{2}(v_1 + \imag v_2), 
\end{equation}
and the recursive formula of the basis functions \cite{Cai2018} is
\begin{equation}
  \label{eq:recursive}
  \begin{aligned}
    & S_{\mu}B_{\hat{\balpha}}(\bv)= \frac{1}{2^{|\mu|/2}}
    \left[\sqrt{2(\hat{\alpha}_1+\hat{\alpha}_3)+3}\eta_{\halpha_1+1,
        m}^{\mu}B_{\hat{\balpha} +(1, \mu,0)^T}(\bv) -
      \sqrt{2\halpha_3}\eta_{\halpha_1+1, \halpha_2}^{\mu}B_{\hat{\balpha}+(1,\mu,-1)^T}(\bv) \right. \\
    & \left. + (-1)^{\mu}\sqrt{2(\halpha_3+\halpha_1)+1}
      \eta_{-\halpha_1,\halpha_2}^{\mu}B_{\hat{\balpha}+(-1,\mu,0)^T}(\bv)
      -(-1)^{\mu}\sqrt{2(\halpha_3+1)}\eta_{-\halpha_1,\halpha_2}^{\mu}B_{\hat{\balpha}
        + (-1,\mu,1)^T}(\bv) \right],
  \end{aligned}
\end{equation}
where $ \eta_{lm}^{\mu}$ is defined in \eqref{eq:eta} and we set
$B_{\hat{\balpha}}(\bv) = 0$ if $|\halpha_2| > \halpha_1$ or either of
$\halpha_1, \halpha_3$ is negative. Based on the recursion formula of
Hermite polynomials
\begin{equation}
  \label{eq:recursive_h}
  v_s H^{\balpha}(\bv) = H^{\balpha+e_s}(\bv)
  + k_sH^{ \balpha - e_s}(\bv), \qquad s = 1,2,3,
\end{equation}
we can get the recursive formula to compute
$C_{\hat{\balpha}}^{\balpha}$, precisely 
\begin{equation}
  \label{eq:recursive_c}
  \begin{aligned}
    a_{\hat{\balpha} + e_2}^{(-1)}C_{\hat{\balpha} + e_1}^{\balpha} +
    b_{\hat{\balpha}+e_2}^{(-1)}C_{\hat{\balpha} - e_1 +
      e_3}^{\balpha} &= \frac{1}{2}k_1C_{\hat{\balpha}+e_2}^{\balpha -
      e_1} -
    \frac{\imag}{2} k_2 C_{\hat{\balpha} + e_2}^{\balpha - e_2}, \\
    a_{\hat{\balpha}}^{(0)}C_{\hat{\balpha}+e_1}^{\balpha} +
    b_{\hat{\balpha}}^{(0)}C_{\hat{\balpha} - e_1 + e_3}^{\balpha} &=
    k_3C_{\hat{\balpha}}^{\balpha - e_3}, \\
    a_{\hat{\balpha} - e_2}^{(1)}C_{\hat{\balpha}+e_1}^{\balpha} +
    b_{\hat{\balpha} - e_2}^{(1)}C_{\hat{\balpha} - e_1 + e_3}^{\balpha} &=
    -\frac{1}{2}k_1C_{\hat{\balpha} - e_2}^{\balpha - e_1} - \frac{\imag}{2}
    k_2C_{\hat{\balpha} - e_2}^{\balpha- e_2},
    \end{aligned}     
\end{equation}
where $|\balpha| = |\hat{\balpha}|_B$ and 
\begin{equation}
  \label{eq:coe_a}
  a_{\hat{\balpha}}^{(\mu)} = \frac{1}{2^{|\mu|/2}}\sqrt{(2(\halpha_1+\halpha_3)
    +3)}\eta_{\halpha_1+1,\halpha_3}^{\mu},
  \quad b_{\hat{\balpha}}^{(\mu)} = \frac{(-1)^{\mu+1}}{2^{|\mu|/2}}\sqrt{2(\halpha_3+1)
  }\eta_{-\halpha_1,\halpha_3}^{\mu}, \quad \mu = -1, 0, 1,
\end{equation}
As is stated in \cite{BurnettCol}, we solve all the coefficients
$C_{\hat{\balpha}}^{\balpha}$ by the order of $|\balpha|$, so that the
right-hand sides of \eqref{eq:recursive_c} are always known.  The
initial condition and the boundary conditions are
$C_{\boldsymbol{0}}^{\boldsymbol{0}}=1$ and
$C_{\hat{\balpha}}^{\balpha} = 0$ if $|\halpha_2| > \halpha_1$ or
either of $\halpha_1, \halpha_3$ is negative. Moreover, the time
complexity for computing all the coefficients
$C_{\hat{\balpha}}^{\balpha}$ with
$|\hat{\balpha}|_B = |\balpha| \leqslant M$ is $O(M^5)$.

%%% Local Variables: 
%%% mode: latex
%%% TeX-master: "article"
%%% End

\bibliographystyle{plain}
\bibliography{../article}%, ../tiao}
%\addbibresource{../article.bib}
%\addbibresource{../tiao.bib}
\end{document}